\documentclass[a4paper,11pt,twoside]{amsart}

\usepackage{graphicx}
\usepackage{amsfonts}
\usepackage{amsmath}
\usepackage{amsthm}
\usepackage{amssymb}
\usepackage[all]{xy}
\usepackage{hyperref}
\usepackage{aliascnt}
\usepackage[mathscr]{euscript}

\newtheorem{thm}{Theorem}[section]

\newtheorem{thmInt}{Theorem}[section]

\newaliascnt{prop}{thm}
\newtheorem{prop}[prop]{Proposition}
\aliascntresetthe{prop}

\newaliascnt{lem}{thm}
\newtheorem{lem}[lem]{Lemma}
\aliascntresetthe{lem}

\newaliascnt{cor}{thm}
\newtheorem{cor}[cor]{Corollary}
\aliascntresetthe{cor}

\theoremstyle{definition}

\newaliascnt{definition}{thm}
\newtheorem{definition}[definition]{Definition}
\aliascntresetthe{definition}

\newaliascnt{remark}{thm}
\newtheorem{remark}[remark]{Remark}
\aliascntresetthe{remark}

\newaliascnt{ex}{thm}
\newtheorem{ex}[ex]{Example}
\aliascntresetthe{ex}

\newaliascnt{qn}{thm}
\newtheorem{qn}[qn]{Question}
\aliascntresetthe{qn}

\newaliascnt{stp}{thm}

\aliascntresetthe{stp}

\newaliascnt{setup}{thm}
\newtheorem{setup}[setup]{Setup}
\aliascntresetthe{setup}

\numberwithin{equation}{section}

\DeclareMathOperator{\im}{im} 
\DeclareMathOperator{\spec}{Spec} 
\DeclareMathOperator{\colim}{colim} 
\DeclareMathOperator{\holim}{Holim} 
\DeclareMathOperator{\hocolim}{Hocolim} 

\newcommand{\iso}{\cong}

\newcommand{\st}{:} 
\newcommand{\comp}{\circ} 
\newcommand{\rest}[1]{|_{#1}} 
\newcommand{\card}[1]{\lvert#1\rvert} 
\newcommand{\abs}[1]{\lvert#1\rvert} 

\newcommand{\id}{\mathrm{id}}
\newcommand{\Hom}{\mathrm{Hom}}

\newcommand{\Ob}{\mathrm{Ob}}

\newcommand{\real}{\mathrm{real}}

\newcommand{\lto}{\longrightarrow}
\newcommand{\mor}[1]{\xrightarrow{#1}} 
\newcommand{\lmor}[1]{\overset{#1}{\lto}} 
\newcommand{\mono}{\hookrightarrow} 
\newcommand{\epi}{\twoheadrightarrow} 
\newcommand{\isomor}{\mor{\sim}} 
\newcommand{\cone}[1]{\mathrm{Cone}\left(#1\right)} 

\newcommand{\kk}{\Bbbk} 
\newcommand{\NN}{\mathbb{N}}
\newcommand{\ZZ}{\mathbb{Z}}

\newcommand{\FF}{\mathbb{F}}

\newcommand{\UU}{\mathbb{U}}
\newcommand{\VV}{\mathbb{V}}

\newcommand{\s}[1]{\mathcal{#1}} 
\newcommand{\so}{\s{O}} 
\newcommand{\sHom}{\s{H}om} 

\newcommand{\opp}{^{\circ}} 
\newcommand{\farg}{-} 
\newcommand{\ort}[1]{\langle#1\rangle} 
\newcommand{\sh}[2][1]{#2[#1]} 
\newcommand{\ic}[1]{#1^{\mathrm{ic}}} 

\newcommand{\cat}[1]{{\mathscr{#1}}} 
\newcommand{\ca}{\cat{A}}
\newcommand{\cA}{\cat{A}}
\newcommand{\cb}{\cat{B}}
\newcommand{\cB}{\cat{B}}
\newcommand{\cc}{\cat{C}}

\newcommand{\cd}{\cat{D}}

\newcommand{\ce}{\cat{E}}

\newcommand{\cg}{\cat{G}}
\newcommand{\cG}{\cat{G}}
\newcommand{\cl}{\cat{L}}

\newcommand{\cn}{\cat{N}}

\newcommand{\cp}{\cat{P}}

\newcommand{\cR}{\cat{R}}
\newcommand{\cs}{\cat{S}}
\newcommand{\cS}{\cat{S}}
\newcommand{\ct}{\cat{T}}
\newcommand{\cT}{\cat{T}}
\newcommand{\cv}{\cat{V}}
\newcommand{\cV}{\cat{V}}

\newcommand{\scat}[1]{{\mathbf{#1}}} 
\newcommand{\Coh}{\scat{Coh}}
\newcommand{\Qcoh}{\scat{Qcoh}}
\newcommand{\dgCat}{\scat{dgCat}} 
\newcommand{\Hqe}{\scat{Hqe}} 
\newcommand{\D}{\scat{D}} 
\newcommand{\Da}{\D^?}
\newcommand{\Du}{\D^+}
\newcommand{\Dm}{\D^-}
\newcommand{\Db}{\D^b}
\newcommand{\Dh}{\widehat{\D}}
\newcommand{\Dc}{\check{\D}}
\newcommand{\Dq}{\D_{\scat{qc}}}
\newcommand{\Dqa}{\Da_{\scat{qc}}}

\newcommand{\Dp}{\scat{Perf}} 
\newcommand{\K}{\scat{K}} 
\newcommand{\Ka}{\K^?}

\newcommand{\C}{\scat{C}} 

\newcommand{\dgC}{\C_{\scat{dg}}} 
\newcommand{\dgCa}{\dgC^?}

\newcommand{\Acy}{\K_{\mathrm{acy}}} 
\newcommand{\Acya}{\Ka_{\mathrm{acy}}}

\newcommand{\dgAcy}{\scat{Acy}} 
\newcommand{\dgAcya}{\dgAcy^?}

\newcommand{\V}{\scat{V}}
\newcommand{\Va}{\V^?}

\newcommand{\B}{\scat{B}}
\newcommand{\Ba}{\B^?}

\newcommand{\dgD}{\mathcal{D}} 
\newcommand{\Perf}[1]{\mathrm{Perf}(#1)} 
\newcommand{\Pretr}[1]{\mathrm{Pretr}(#1)}
\newcommand{\Mod}[1]{\mathrm{Mod}(#1)} 
\newcommand{\dgMod}[1]{\mathrm{dgMod}(#1)}
\newcommand{\dgAc}[1]{\mathrm{dgAcy}(#1)}
\newcommand{\hproj}[1]{\mathrm{h}\text{-}\mathrm{proj}(#1)} 
\newcommand{\Inj}{\mathrm{Inj}}
\newcommand{\Loc}{\scat{Loc}}
\newcommand{\IndC}{\mathrm{Ind}}
\newcommand{\Lex}{\mathrm{Lex}}
\newcommand{\dgHom}{\underline{\sHom}} 

\newcommand{\fun}[1]{\mathsf{#1}} 
\newcommand{\fE}{\fun{E}}
\newcommand{\fF}{\fun{F}}
\newcommand{\fG}{\fun{G}}
\newcommand{\fH}{\fun{H}}
\newcommand{\fI}{\fun{I}}
\newcommand{\fJ}{\fun{J}}
\newcommand{\fK}{\fun{K}}

\newcommand{\fQ}{\fun{Q}}
\newcommand{\fP}{\fun{P}}

\newcommand{\Yon}[1][\ca]{\fun{Y}^{#1}} 
\newcommand{\dgYon}[1][\cc]{\fun{Y}^{#1}_{\mathrm{dg}}} 
\newcommand{\RYon}[1][\cd]{\widehat{\fun{Y}}^{#1}_{\mathrm{dg}}} 
\newcommand{\Ind}{\fun{Ind}} 
\newcommand{\Res}{\fun{Res}} 
\newcommand{\bfun}{\fun{B}} 

\newcommand{\ep}{\varepsilon}
\newcommand{\ph}{\varphi}
\newcommand{\wt}{\widetilde}
\newcommand{\wh}{\widehat}
\newcommand{\ds}{\displaystyle}
\newcommand{\gen}[2]{{\left\langle#1\right\rangle}^{}_{#2}}
\newcommand{\Ddg}{\D_{\mathbf{dg}}}
\newcommand{\Ddga}{\D^?_{\mathbf{dg}}}
\newcommand{\Dhdg}{\widehat{\D}_{\mathbf{dg}}}
\newcommand{\equiva}{\lmor{\sim}}
\newcommand{\Bsum}[1][n]{{}^{#1}\!B}
\newcommand{\qf}{u} 
\newcommand{\hff}[1]{{#1}^{\mathrm{hf}}}
\newcommand{\sma}[1]{{#1}^{\mathrm{sm}}}

\newcommand{\hl}[1]{{#1}^{\mathrm{hl}}}
\newcommand{\Homr}{\overline{\Hom}}

\newcommand{\m}{m}
\newcommand{\mr}{\overline{\m}}
\newcommand{\p}{p}

\begin{document}

	\title[Uniqueness of enhancements for derived categories]{Uniqueness of enhancements\\for derived and geometric categories}

	\author[A.\ Canonaco, A.\ Neeman, and P.\ Stellari]{Alberto Canonaco, Amnon Neeman, and Paolo Stellari}

	\address{A.C.: Dipartimento di Matematica ``F.\ Casorati'', Universit{\`a}
	degli Studi di Pavia, Via Ferrata 5, 27100 Pavia, Italy}
	\email{alberto.canonaco@unipv.it}
	
	\address{A.N.: Centre for Mathematics and its Applications, Mathematical Sciences Institute, Building 145, The Australian National University, Canberra, ACT 2601, Australia}
	\email{Amnon.Neeman@anu.edu.au}

	\address{P.S.: Dipartimento di Matematica ``F.\
	Enriques'', Universit{\`a} degli Studi di Milano, Via Cesare Saldini
	50, 20133 Milano, Italy}
	\email{paolo.stellari@unimi.it}
    \urladdr{\url{https://sites.unimi.it/stellari}}
    
    \thanks{A.~C.~was partially supported by the research project PRIN 2017 ``Moduli and Lie Theory''.
    	P.~S.~was partially supported by the ERC Consolidator Grant ERC-2017-CoG-771507-StabCondEn, by the research project PRIN 2017 ``Moduli and Lie Theory'', and by research project FARE 2018 HighCaSt (grant number R18YA3ESPJ).}



\begin{abstract}
We prove that the  derived categories of abelian categories have unique enhancements---all of them, the unbounded, bounded, bounded above and  bounded below derived categories. The unseparated and left completed derived categories of a Grothendieck abelian category are also shown to have unique enhancements. Finally we show that the derived category of complexes with quasi-coherent cohomology and the category of perfect complexes have unique enhancements for quasi-compact and quasi-separated schemes.
\end{abstract}

\maketitle

\setcounter{tocdepth}{1}
\tableofcontents

\section*{Introduction}

The intriguing relation between triangulated categories and their higher categorical enhancements---either pretriangulated dg, or pretriangulated $A_\infty$ or stable $\infty$-categorical---has been under investigation for several years now. One reason is that, while triangulated categories have grown remarkably important in representation theory and algebraic geometry, many of the constructions one wants to make rely on the functoriality that comes with an enhancement. Many instances of this phenomenon appeared in the recent developments of derived algebraic geometry, for example in relation to deformation theory and moduli problems.

In this paper we stick to the language of dg (differential graded) categories but, as we will explain later, our results apply to $A_\infty$ and stable $\infty$-enhancements as well. We recall that, roughly, a dg enhancement (or simply an enhancement) of a triangulated category $\cT$ is a pretriangulated dg category whose homotopy category is equivalent to $\cT$. It is relevant to 
note that the `natural' triangulated categories of algebra and geometry
come with `natural' dg enhancements. For example the derived categories $\Da(\ca)$ of an abelian category $\cA$, where $?=\emptyset,b,+,-$ (i.e.\ where the cohomology is assumed unbounded, bounded, bounded below or bounded above), as well as the category $\Perf{X}$ of perfect complexes on a quasi-compact and quasi-separated scheme, all have `natural' dg enhancements by construction.  This existence does not hold in general. Indeed, there are well known examples of `topological' triangulated categories that do not admit dg enhancements (see, for instance, \cite[Section 3.2]{CS4}). More recently it has been proved in \cite{RvdB} that there exist triangulated categories which are linear over a field and without a dg enhancement. 

A priori, there is no good reason to expect different enhancements of the same triangulated category to be `comparable'. This is important because constructions that take place in an enhancement may depend on the choice of enhancement. And the right notion of `comparability' turns out to be that two enhancements are declared
equivalent if they agree up to isomorphism in the homotopy category $\Hqe$ of the category of (small) dg categories. Consequently one says that a triangulated category has a unique enhancement if any two enhancements are isomorphic in $\Hqe$. As $\Hqe$ is the localization of the category of dg categories by quasi-equivalences, two dg categories which are isomorphic in $\Hqe$ have equivalent homotopy categories, but the converse need not be true. So far very few examples of triangulated categories admitting non unique enhancements have been produced. A `classical' one is reported in \cite[Section 3.3]{CS4} (see also \autoref{cor:negative}). If one requires that everything be linear over a field, the first example was recently found by Rizzardo and Van den Bergh \cite{RvdB2}.

Back to $\Hqe$: one can describe all morphisms in this category thanks to the seminal work of To\"en \cite{To} (see also \cite{COS,CSdg}). Indeed, for the natural enhancements of geometric categories, such as the bounded derived categories of coherent sheaves on smooth projective schemes, the morphisms in $\Hqe$ between them are all lifts of exact functors of a special form: the so called Fourier--Mukai functors (see \cite{LS,To} and \cite{CS4,CS3} for a survey).

The way the triangulated and dg sides of this picture should be related was pinned down, in the geometric setting, in the seminal work \cite{BLL} by Bondal--Larsen--Lunts where it is conjectured that
\begin{enumerate}
\item[(C1)] The geometric triangulated categories $\Db(\Coh(X))$, $\D(\Qcoh(X))$ and $\Dp(X)$ should have a unique dg enhancement, when $X$ is a quasi-projective scheme (i.e.\ any two dg enhancements should be isomorphic in $\Hqe$);
\item[(C2)] If $X_1$ and $X_2$ are smooth projective schemes, then all exact functors between $\Db(\Coh(X_1))$ and $\Db(\Coh(X_2))$ should lift to morphisms in $\Hqe$.
\end{enumerate}

Conjecture (C2) has recently been disproved in \cite{RvdBN} and even when a lift exists it is not unique in general by \cite{CSFMKer}. On the other hand, special cases of (C1) have been proved to be correct, in increasing generality, by several authors over the last decade. Let us briefly go through this part of the story; after all this article belongs to this string of 
results.

The first breakthrough in the direction of (C1) came from the beautiful work by Lunts and Orlov \cite{LO} which proved, among other things, that $\D(\cg)$ has a unique enhancement when $\cg$ is a Grothendieck abelian category with a small set of compact generators. This result implies that (C1) holds true for $\D(\Qcoh(X))$ when $X$ is a quasi-compact and separated scheme with enough locally free sheaves (see \cite[Theorem 2.10]{LO}). This was extended to all Grothendieck abelian categories in \cite{CSUn1} by using the theory of well generated triangulated categories. Hence (C1) holds for $\D(\Qcoh(X))$ when $X$ is any scheme or any algebraic stack.

As for $\Db(\Coh(X))$ and $\Dp(X)$, Lunts and Orlov show in \cite{LO} that they have unique enhancements when $X$ is a quasi-projective scheme (see Theorems 2.12 and 2.13 in \cite{LO}). Clearly, this together with the previous result implies that (C1) holds even in greater generality. Actually, an additional improvement of the argument in \cite{LO} allowed the first and third author to prove that $\Db(\Coh(X))$ and $\Dp(X)$ have unique enhancements when $X$ is any noetherian scheme with enough locally free sheaves (see Corollaries 6.11 and 7.2 in \cite{CSUn1}).

More recently, Antieau \cite{A1} reconsidered the problem of uniqueness of enhancements by taking a completely different perspective which involves Lurie's work on prestable $\infty$-categories (see \cite[Appendix C]{LurSpectr}). Using this powerful machinery, he proved the beautiful result that $\Da(A)$ has a unique enhancement when $?=b,+,-$ and $\cA$ is any small abelian category. It should be noted that restricting to small categories is not particularly relevant here and in the sequel, as explained in \autoref{subsect:dgenuniq}.

If $\ca$ is not only abelian but also a Grothendieck category, following Lurie's point of view, one can construct three interesting triangulated categories out of $\ca$: the usual derived category $\D(\ca)$, the \emph{unseparated derived category} $\Dc(\ca)$ and the \emph{left completed derived category} $\Dh(\ca)$. We know all about $\D(\ca)$ and, in particular, we know that it has a unique enhancement by \cite{CSUn1}. The triangulated category $\Dc(\ca)$ is nothing but the homotopy category of injectives in $\cA$ and has been extensively studied by Krause in \cite{K0}. The uniqueness of enhancements for $\Dc(\ca)$, when $\ca$ is not only Grothendieck but also locally coherent, is one of the main results of Antieau (see \cite[Theorem 1]{A1}). The triangulated category $\Dh(\ca)$ is more mysterious. It does not seem to have a purely triangulated description and it should be thought of as a remedy to the fact that, in general, $\D(\cA)$ is not left complete (see \cite{NNonLeftCompl}). In \cite{A1}, the uniqueness of the enhancement for $\Dh(\ca)$ is stated as an open and challenging problem (see \cite[Question 8.1]{A1}).

\medskip

The first main result of this paper has a twofold motivation. First of all, we want to provide a vast generalization of all known results about the uniqueness of enhancements for $\Da(\ca)$, when $\ca$ is an arbitrary abelian category, and of its close relatives $\Dc(\ca)$ and $\Dh(\ca)$, when $\ca$ is a Grothendieck abelian category. This will provide positive answers to several open problems in the literature. Secondly we want to bring the approach to these problems and the proofs back to the realm of triangulated and dg categories, removing the need for $\infty$-categories. In fact: $\infty$-categories 
 will play no role in this article, apart from the fact that they are needed to define the triangulated category $\Dh(\ca)$.

Our first precise statement is the following: 

\begin{thmInt}\label{thm:main1}
Let $\ca$ be an abelian category.
\begin{enumerate}
\item The triangulated category $\Da(\ca)$ has a unique dg enhancement when $?=b,+,-,\emptyset$.
\item If $\ca$ is a Grothendieck abelian category, then $\Dc(\ca)$ and $\Dh(\ca)$ have unique dg enhancements.
\end{enumerate}
\end{thmInt}

The striking and new part of (1) is the uniqueness for $\D(\ca)$, for \emph{every} abelian category $\ca$. Nonetheless our approach will uniformly and harmlessly produce uniqueness of enhancements for $\Da(\ca)$, for $?=b,-,+$, thus recovering Antieau's results in a completely different way. As easy applications, we deduce that the following triangulated categories have unique enhancements:
\begin{itemize}
\item  $\Da(\Qcoh(X))$, for $X$ any scheme or algebraic stack, and $\Da(\Coh(X))$, for any scheme or any locally noetherian algebraic stack, for $?=\emptyset,b,+,-$  (see \autoref{cor:cohquasicoh});
\item $\D(\cG)^\alpha$, where $\cG$ is a Grothendieck abelian category and $\alpha$ is a large regular cardinal (see \autoref{cor:Calpha}). Recall here that $\D(\cG)$ is a well generated triangulated category and $\D(\cG)^\alpha$ is the full triangulated subcategory consisting of its $\alpha$-compact objects;
\item $\Ka(\ca)$, for any abelian category $\ca$ and for  $?=\emptyset,b,+,-$ (see \autoref{cor:uniqKA}).
\end{itemize}
The first two items together yield a complete positive answer to Question 4.8 in \cite{CS3}. As an aside, we recover and generalize in \autoref{subsect:realization} the construction of the \emph{realization functor} due to Be{\u\i}linson, Bernstein and Deligne \cite{BeiBerDel82}.

Part (2) of \autoref{thm:main1} on the one hand generalizes \cite[Theorem 1]{A1} and, on the other hand, answers the questions in \cite{A1} about $\Dh(\ca)$ that we mentioned above. In particular, \autoref{thm:main1} (2) gives a positive answer to Question 4.7 in \cite{CS4}.

We may conclude this discussion by pointing out that there are still a couple of situations of high algebro-geometric interest where the (non-)uniqueness of the enhancements needs to be fully understood: the categories of matrix factorizations and the case of admissible subcategories of triangulated categories admitting a unique enhancement. If we work with categories and functors linear over $\ZZ$, then there are examples of categories of matrix factorizations with non-unique enhancements (see \cite{DS,Sc}). Similarly, in \autoref{subsect:applicationcounter}, we provide an example of a $\ZZ$-linear triangulated category with a unique enhancement (by \autoref{thm:main1}) but with an admissible subcategory with non-unique $\ZZ$-linear enhancements (see \autoref{cor:negative}). It remains open to understand if similar examples can be found for categories linear over a field and if one can find admissible subcategories with non-unique enhancements in $\Db(\Coh(X))$, when $X$ is a smooth projective scheme.

Back to the geometric setting. If $X$ is a quasi-compact and quasi-separated scheme, then the category $\D(\Qcoh(X))$ is in general not equivalent to $\Dq(X)$, the full triangulated subcategory of the category of complexes of $\so_X$-modules consisting of all complexes with quasi-coherent cohomology. Thus the uniqueness of dg enhancements for $\Dq(X)$ cannot directly be deduced from \autoref{thm:main1}. Nonetheless this category is of primary interest, for example because the category of perfect complexes on $X$ coincides with the subcategory of compact objects in $\Dq(X)$.

The uniqueness of dg enhancements for $\Dq(X)$ and $\Dp(X)$ was formulated as an open problem by Antieau in \cite[Question 8.14]{A1} and our second main result positively answers his question:

\begin{thmInt}\label{thm:main2}
Let $X$ be a quasi-compact and quasi-separated scheme. Then the categories $\Dqa(X)$ and $\Dp(X)$ have a unique dg enhancement, for $?=b,+,-,\emptyset$.
\end{thmInt}

The case of $\Dp(X)$ is covered by \cite[Corollary 9]{A1}, under the stronger assumption that the scheme is quasi-compact, quasi-separated and $0$-complicial. The latter condition, which we do not need to make explicit here, roughly refers to a property of $\Dp(X)$ induced by the t-structure on $\Dq(X)$ and it predicts how perfect complexes with only non-negative cohomologies are generated by perfect quasi-coherent sheaves. As we will explain below, part of the interest of \autoref{thm:main2} is that the proof introduces a new technique to study the uniqueness of enhancements, based on homotopy limits.

\subsection*{The strategy of the proofs}

The one-sentence summary of the proof of \autoref{thm:main1}~(1) would be that it is an elaborate study of special generators for $\Da(\cA)$, coupled with a suitable description of
$\Da(\cA)$ as a Verdier quotient. The same principle underlies all the existing papers in the literature proving the uniqueness of enhancements of derived categories of abelian categories.
The many papers differ in which generators they use and what quotient they
study.

In the current article we realize $\Da(\ca)$ as a quotient of the homotopy category $\Ka(\ca)$, and show that $\Ka(\ca)$ is generated in $3$ steps by objects which are direct sums of shifts of objects in $\cA$ (\autoref{prop:genK} and \autoref{cor:genD}). The key fact here, namely that 
$\Ka(\ca)$ is generated in $3$ steps by the simple objects described above,
can be seen by combining Krause's~\cite[Lemma~3.1 and its proof]{K0} with
Max Kelly's old result~\cite{Kelly65} (see 
\cite[Theorem~6.5 and its proof]{Neeman17B} for a
modern account). But we include a full proof in this
article, because we make use of the explicit three steps that
suffice.

This is different from and, in a sense, more natural than
the point of view of \cite{LO} and \cite{CSUn1,CSUn2}. In those earlier
papers, to prove that $\D(\cG)$ has a unique enhancement for $\cG$ a Grothendieck abelian category, one uses the strong and special property that $\cG$ has a generator. Thus one can take generators for $\D(\cg)$ which all live in degree $0$ and $\D(\cG)$ is a suitable quotient of the derived category of modules over the category formed by these generators.

The technical complications of our approach involve, at the triangulated level, a careful analysis of certain special products and coproducts. It is discussed in \autoref{sect:triangcat}, and reverberates at the dg level where one has to construct suitable dg enhancements of $\Ka(\ca)$ and $\Da(\ca)$ and an intricate zigzag of dg functors linking them. The dg work is carried out in \autoref{sect:enB}. Again, the dg part of the argument is simpler in \cite{LO,CSUn1} and involves a short zigzag diagram consisting of one roof of dg functors. The last part of the proof in \autoref{sect:D(A)} is then very close in spirit to the argument in Section 4 and 5 of \cite{LO} (and thus in Section 4 of \cite{CSUn1}).

The proof of the uniqueness of dg enhancements for $\Dc(\ca)$ in \autoref{subsect:sepdercat} is a reduction to Theorem C in \cite{CSUn1,CSUn2} (see \autoref{thm:crit1}). It uses the work of Krause \cite{K0} to show that $\Dc(\ca)$ is a well generated triangulated category, and can be written as a quotient of the derived category of the abelian category of modules over the abelian subcategory $\cA^\alpha$ of $\ca$, which consists of the $\alpha$-presentable objects in $\cA$ (here $\alpha$ is a sufficiently large regular cardinal).

Finally the case of $\Dh(\ca)$ is studied in \autoref{subsec:completeddercat} and the proof makes use of the natural t-structure induced on $\Dh(\ca)$ by $\D(\ca)$. With this t-structure we have a natural equivalence $\Dh(\ca)^+\iso\Du(\ca)$. We can invoke \autoref{thm:main1} (1), and then deduce the result by a careful analysis of the compatibility with homotopy colimits. It should be noted that here we need to use that $\Da(\ca)$ has a semi-strongly unique dg enhancement (see \autoref{rmk:semistrong}). Roughly, this means that if $\cc_1$ and $\cc_2$ are two dg enhancements of $\Da(\ca)$ (i.e.\ there are exact equivalences $\fE_i\colon H^0(\cc_i)\isomor\Da(\ca)$), then the isomorphism $f\colon\cc_1\isomor\cc_2$ in $\Hqe$ provided by \autoref{thm:main1} (1) is such that $H^0(f)(X)\iso \fE_2^{-1}\comp\fE_1(X)$, for all $X$ in $H^0(\cc_1)$.

\medskip

The strategy of the proof of \autoref{thm:main2} is new, and is based on the idea of realizing a dg enhancement of $\Dqa(X)$ (and of $\Dp(X))$ as the homotopy limit of dg enhancements of the derived category of the open subschemes in an affine open cover of $X$. More precisely: in \autoref{sect:uniquepullbacks} we prove that, given any enhancement $\cc$ of $\Dqa(X)$ (or $\Dp(X))$, one can produce an isomorphism in $\Hqe$ between $\cc$ and the homotopy limit of induced enhancements of the derived categories of the affine schemes in the cover (and of their finite intersections). This can be deduced from \autoref{thm:critpb}, which is a general criterion involving the simpler case of homotopy pullbacks.

This has the clear advantage that, for each Zariski open subset $U$ in the covering of $X$ and all their finite intersections, one knows that $\Dqa(U)\iso\Da(\Qcoh(U))$, since $U$ is quasi-compact and separated. Thus the uniqueness of their enhancements is guaranteed by \autoref{thm:main1} (1). The hard work comes up in showing that the constructions in \autoref{sect:enB}, and thus the proof of \autoref{thm:main1} (1) in \autoref{sect:D(A)}, are compatible with restriction to appropriate open subschemes. In \autoref{sect:uniqgeomcat} we show this compatibility with the special homotopy limits we are considering, concluding the proof.

\subsection*{Related work}

As we have mentioned before, \autoref{thm:main1} and \autoref{thm:main2} provide major generalizations of the results available in the literature about the uniqueness of dg enhancements. Nonetheless a careful comparison with Antieau's lovely results of \cite{A1} is needed, if only because of the difference in language. Antieau formulated his work in terms of the uniqueness of stable $\infty$-enhancements rather than dg enhancements.

One of the main differences between these worlds is that while dg categories always come with (at least) a $\ZZ$-linearization, stable $\infty$-categories do not in general. Thus, a priori, one may have more stable $\infty$-enhancements than dg ones. But Meta Theorem 13 of \cite{A1} shows that the stable $\infty$-enhancements of the triangulated categories studied in that paper (and in this one) possess canonical $\ZZ$-linear structures. Thus two stable $\infty$-enhancements of the same triangulated category in \autoref{thm:main1} and \autoref{thm:main2} are $\infty$-equivalent if there is an $\infty$-equivalence between them that preserves the canonical $\ZZ$-linear structure. At this point, the main result in \cite{Coh} shows that any such $\ZZ$-linear stable $\infty$-enhancement is equivalent to a $\ZZ$-linear pretriangulated dg category and that the two $\ZZ$-linearized stable $\infty$-enhancements are equivalent if and only if the corresponding $\ZZ$-linearized dg ones are.

Summarizing, by \cite[Meta Theorem 13]{A1} and \cite{Coh}, for the triangulated categories in \cite{A1} and the ones in this paper, it is the same to prove either that they have a unique dg enhancement or that they have a unique stable $\infty$ one. Thus our results recover those in \cite{A1}. Moreover, it should be noted that our setting is slightly more general as we do not need to restrict only to $\ZZ$-linearizations. If any of the categories in \autoref{thm:main1} and \autoref{thm:main2} are linear over a commutative ring $\kk$, then there is a unique $\kk$-linear dg or stable $\infty$-enhancement (here we use again \cite{Coh}).

It is still possible that the use of (pre)stable $\infty$-categories may give new important inputs into the study of this kind of problems. As an example, the very nice and powerful Theorem 8 in \cite{A1} does not have a dg counterpart at the moment.

To conclude this discussion let us note that, given a triangulated category $\ct$, one might consider $A_\infty$ enhancements as well. Here the situation is more elaborate, as one may take either strictly unital or cohomologically unital pretriangulated $A_\infty$ categories whose homotopy categories are equivalent to $\ct$. Nonetheless, it was proven in \cite{COS} for categories linear over a field and in \cite{OT} over any commutative ring $\kk$ that the localization of the category of (small) $\kk$-linear dg categories with respect to quasi-equivalences is equivalent to the localization of the ($\infty$-)category of (strictly or cohomologically unital) $A_\infty$ categories with respect to quasi-equivalences. This implies that a triangulated category $\ct$ which is linear over any commutative ring $\kk$ has a unique $\kk$-linear dg enhancement if and only if it has a unique $\kk$-linear $A_\infty$ enhancement. Thus \autoref{thm:main1} and \autoref{thm:main2} work also in the realm of $A_\infty$ categories and enhancements. 

\subsection*{Structure of the paper}

In \autoref{sect:triangcat} we deal with some foundational questions about products and coproducts in the triangulated categories $\Ka(\ca)$ and $\Da(\ca)$. We also discuss their behavior with respect to exact functors.

\autoref{sec:generation} introduces various notions of generation for triangulated categories. In particular, in \autoref{subsec:wellgentriacat} we define well generated triangulated categories and we set the stage for our study of $\Dc(\cg)$. \autoref{subsec:gendercat} is about the approach to the generation of $\Ka(\ca)$ and $\Da(\ca)$ which is crucial in the proof of \autoref{thm:main1} (1).

In \autoref{sect:dgenhancements} we introduce some standard material about dg categories and at the same time we slightly extend known results and constructions such as Drinfeld's notion of homotopy flat resolution (see \autoref{subsec:hprojhflat}). In  \autoref{subsec:modelpullbacks} we study homotopy limits and homotopy pullbacks. We also reconsider localizations in the dg context (\autoref{subsec:someextria}) and define carefully the notion of dg enhancement and why their uniqueness is independent of the universe (see \autoref{subsect:dgenuniq}).
	
\autoref{sect:enB} is devoted to the construction of appropriate enhancements for $\Ka(\ca)$ and $\Da(\ca)$. This naturally leads to the proof of \autoref{thm:main1} (1) in \autoref{sect:D(A)}. The second part of \autoref{thm:main1} is proved in the subsequent section and uses techniques which are somewhat different.

The proof of \autoref{thm:main2} occupies \autoref{sect:uniquepullbacks} and \autoref{sect:uniqgeomcat}. The first of the two sections sets up the general technique and criterion that link homotopy pullbacks and limits to dg enhancements. The second one combines this with \autoref{thm:main1} to finish the argument.

\subsection*{Notation and conventions}

All preadditive categories and all additive functors are assumed to be $\kk$-linear, for a fixed commutative ring $\kk$. By a $\kk$-linear category we mean a category
whose Hom-spaces are $\kk$-modules and such that the compositions
are $\kk$-bilinear, not assuming that finite coproducts exist.

With the small exception of \autoref{subsect:dgenuniq}, throughout the paper we assume that a universe containing an infinite set is fixed. Several definitions concerning dg categories need special care because they may, in principle, require a change of universe. The major subtle logical issues in this sense can be overcome in view of \cite[Appendix A]{LO} and \autoref{subsect:dgenuniq}. A careful reader should have a look at them, but in this paper, after these delicate issues are appropriately discussed, we will not mention explicitly the universe we are working in. The members of this universe will be called small sets. Unless stated otherwise, we always assume that the Hom-spaces in a category form a small set. A category is called \emph{small} if the isomorphism classes of its objects form a small set.

If $\cT$ is a triangulated category and $\cS$ a full triangulated subcategory of $\cT$, we denote by $\cT/\cS$ the Verdier quotient of $\cT$ by $\cS$. In general, $\cT/\cS$ is not a category according to our convention (meaning the Hom-spaces in $\cT/\cS$ need not be small sets), but it is in many common situations, for instance when $\cT$ is small.

\section{The triangulated categories}\label{sect:triangcat}

In this section we discuss some properties of most of the triangulated categories whose dg enhancements are studied in this paper. The focus is on the existence of (co)products of special objects and the commutativity of such (co)products with exact functors.

\subsection{The categories}\label{subsec:cat}

Let us recall that when $\ca$ is a small additive category, then $\K(\ca)$ denotes the homotopy category of complexes. Namely, its objects are cochain complexes of objects in $\ca$, while its morphisms are homotopy equivalence classes of morphisms of complexes. For $A^*\in\Ob(\K(\ca))$, we denote by $A^i$ its $i$-th component. We can then define the full subcategories $\K^b(\ca)$, $\K^+(\ca)$, $\K^-(\ca)$ of the category $\K(\ca)$ whose objects are
\begin{gather*}
\Ob(\K^b(\ca))=\left\{A^*\in\K(\ca)\st A^i=0\text{ for all }|i|\gg0\right\}\\
\Ob(\K^+(\ca))=\left\{A^*\in\K(\ca)\st A^i=0\text{ for all }i\ll0\right\}\\
\Ob(\K^-(\ca))=\left\{A^*\in\K(\ca)\st A^i=0\text{ for all }i\gg0\right\}
\end{gather*}
For $?=b,+,-,\emptyset$, we single out the full subcategory $\Va(\ca)\subseteq\Ka(\ca)$ consisting of objects with zero differentials. The properties of such a subcategory will be studied in \autoref{sec:generation} and will be crucial in the rest of this paper. Here we just point out that, for an object $A^*\in\Va(\ca)$, we will use the shorthand
\[
\bigoplus_{i\in\ZZ}\sh[-i]{A^i}
\]
to remind that the object $A^i\in\ca$ is placed in degree $i$.

\begin{remark}\label{rmk:univprop1}
It is not difficult to prove that $\bigoplus_{i\in\ZZ}\sh[-i]{A^i}$ satisfies both the universal property of product and coproduct in $\Ka(\ca)$. Namely, there are canonical isomorphisms
\[
\bigoplus_{i\in\ZZ}\sh[-i]{A^i}\iso\prod_{i\in\ZZ}\sh[-i]{A^i}\iso\coprod_{i\in\ZZ}\sh[-i]{A^i}.
\]
\end{remark}

When $\ca$ is an abelian category, the full triangulated subcategory $\Acya(\ca)\subseteq\Ka(\ca)$ consists of \emph{acyclic complexes}, i.e.\ objects in $\K(\ca)$ with trivial cohomology. The triangulated category $\Da(\ca)$ is then the Verdier quotient of $\Ka(\ca)$ by $\Acya(\ca)$, and it comes with a quotient functor
\begin{equation}\label{eqn:quotfun}
\fQ:\Ka(\ca)\lto\Da(\ca).
\end{equation}
We can then consider the full subcategory $\Ba(\ca)\subseteq\Da(\ca)$ as
\[
\Ba(\ca):=\fQ(\Va(\ca)).
\]

\begin{remark}\label{rmk:univprop2}
By definition of Verdier quotient, $\Ba(\ca)$ has the same objects as $\Va(\ca)$. Thus we will freely denote them by $\bigoplus\sh[-i]{A^i}$. But since the morphisms in $\Da(\ca)$ differ from those in $\Ka(\ca)$ in a significant way, we should not expect $\bigoplus\sh[-i]{A^i}$ to automatically satisfy the universal properties of either product or coproduct in $\Da(\ca)$. This will be discussed later.
\end{remark}

\begin{remark}\label{rmk:equivKD}
It is interesting to observe that if $\ca$ is a small abelian category, then \cite[Lemma 3.1]{K0} implies that there is a (small) abelian category $\cb$ and an exact equivalence $\Ka(\ca)\iso\Da(\cb)$, for $?=b,+,-,\emptyset$. To be precise, the category $\cb$ is the abelian category of coherent $\ca$-modules. The result follows from \cite[Proposition III 2.4.4 (c)]{Verd}, once we observe that any coherent $\ca$-module has a projective resolution of length at most $2$ (see the proof of \cite[Lemma 3.1]{K0}).
\end{remark}

If $\cg$ is a Grothendieck abelian category, then it contains enough injectives and one can take the full subcategory $\Inj(\cg)\subseteq\cg$ of injective objects. According to Lurie's terminology \cite{LurHA}, the \emph{unseparated derived category} of $\cg$ is the triangulated category
\[
\Dc(\cg):=\K(\Inj(\cg)).
\]
This category has been extensively studied by the second author \cite{NInj} and Krause \cite{K0}. Some of its properties will be recalled later on. For the moment, we content ourselves with the simple observation that it fits in the following localization sequence
\[
\Acy(\Inj(\cg))\lmor{\fJ}\Dc(\cg)\lmor{\fQ}\D(\cg).
\]
Namely, $\fQ$ and $\fJ$ have right adjoints $\fQ^\rho$ and $\fJ^\rho$, respectively. Under some additional assumptions (e.g.\ if $\D(\cg)$ is compactly generated), $\fQ$ and $\fJ$ have left adjoints $\fQ^\lambda$ and $\fJ^\lambda$ as well.

\begin{ex}\label{ex:recoll}
A case where the latter situation is realized can be obtained as follows. Let $p$ be a prime number and $\FF_p=\ZZ/p\ZZ$ the field with $p$ elements. Consider the rings $R_1:=\ZZ/p^2\ZZ$ and $R_2:=\FF_p[\ep]$ (where $\ep^2=0$). For $i=1,2$, set $\ct_i:=\Dc(\Mod{R_i})$, and denote by $\cs_i$ the full subcategory of $\ct_i$ consisting of acyclic complexes. The ring $R_i$ is noetherian and $\D(\Mod{R_i})$ is compactly generated. Thus the inclusion of $\cs_i$ has left adjoint and $\cs_i$ is actually a localizing and admissible subcategory of $\ct_i$.
\end{ex}

The last triangulated category we study in this paper is the completed derived category of a Grothendieck category. Since, to the best of our knowledge, its definition intrinsically involves pretriangulated dg categories, this discussion is postponed to \autoref{subsec:completedderivedcategories}.

\subsection{More on products and coproducts}\label{subsect:prodcoprod}

As we observed in \autoref{rmk:univprop2}, objects with zero differentials 
need not always agree with the obvious
products or coproducts in $\Da(\ca)$. In this section, we provide
sufficient conditions for agreement.

Let us introduce some notation, slightly generalizing the problem. Let $\ca$ be a small abelian category and let $\{A_n^*\}_{n\ge0}$ be a sequence of objects in $\Da(\ca))$. If either $A_n^i=0$ for all $i>-n$ or $A_n^i=0$ for all $i<n$, we can consider the complex $\bigoplus_{n=0}^\infty A_n^*\in\D(\ca)$ which is the termwise direct sum of the complexes $A_n$. Note that this makes sense because, under our assumptions, each term of the complex $\bigoplus_{n=0}^\infty A_n^*\in\D(\ca)$ consists of a finite direct sum.

With this in mind, we can prove the following.

\begin{lem}\label{lem:procoprod}
Let $\ca$ be a small abelian category and let $\{A_n^*\}_{n\ge0}\subseteq\Ob(\Da(\ca))$.
\begin{itemize}
\item[{\rm (1)}] If $A_n^i=0$ for all $i>-n$ and $?=-,\emptyset$, then $\bigoplus_{n=0}^\infty A_n^*$ is a coproduct in $\Da(\ca)$, i.e.\ there is a canonical isomorphism
\[
\bigoplus_{n=0}^\infty A_n^*\iso\coprod_{n=0}^\infty A_n^*.
\]
\item[{\rm (2)}] If $A_n^i=0$ for all $i<n$ and $?=+,\emptyset$, then $\bigoplus_{n=0}^\infty A_n^*$ is a product in $\Da(\ca)$, i.e.\ there is a canonical isomorphism
\[
\bigoplus_{n=0}^\infty A_n^*\iso\prod_{n=0}^\infty A_n^*.
\]
\end{itemize}
\end{lem}

\begin{proof}
The statements in (1) and (2) are obtained one from the other by passing to the opposite categories. Thus we just need to prove (1) and show that $\bigoplus_{n=0}^\infty A_n^*$ satisfies the universal property of a coproduct in $\Da(\ca)$, for $?=-,\emptyset$.

Suppose therefore that $B^*$ is an object of $\Da(\ca)$ and that we are given morphisms $\ph_n:A_n^*\to B^*$. For $n\geq 0$, $\ph_n$ can be represented in $\Ka(\ca)$ by a roof
\[
\xymatrix@C+10pt@R-20pt{
& &   \wt A_n^*\ar[dr]^-\beta\ar[dl]_-\alpha & \\
& A_n^*\ar[dl] & & B^* \\
N_n^* & & & 
}
\]
This means that $\wt A_n^*\to A_n^*\to N_n^*$ is a distinguished triangle in $\Ka(\ca)$ with $N_n^*\in\Acy(\ca)$, and $\ph_n=\beta\comp\alpha^{-1}$ in $\Da(\ca)$. 

As $A_n^i=0$ for all $i>-n$, the cochain map $A_n^*\to N_n^*$ must factor as $A_n^*\to N_n^{\le-n}\to N_n^*$, where $N_n^{\le-n}$ is still acyclic but vanishes in degrees $>-n$. Here $N_n^{\le-n}$ is the truncation
\[
\cdots\lto N^{-n-2}\lto N^{-n-1}\lto K\lto0\lto0\lto\cdots,
\]
where $K$ is the kernel of the map $N^{-n}\to N^{-n+1}$.

Now we form in $\Ka(\ca)$ the morphism of distinguished triangles
\[
\xymatrix{
\wh A_n^*\ar[r]\ar[d] & A_n^*\ar[r] \ar@{=}[d] & N_n^{\le-n}\ar[d] \\
\wt A_n^*\ar[r] & A_n^*\ar[r] & N_n^*
}
\]
which allows us to represent the morphism $\ph_n:A_n^*\to B^*$ by the roof
\[
\xymatrix@C+10pt@R-20pt{
    & &   \wh A_n^*\ar[dr]\ar[dl] & \\
   & A_n^*\ar[dl] & & B^* \\
N_n^{\le-n} & & & 
}
\]
with $\wh A_n^i=0$ for all $i>-n+1$. And now the roof
\[
\xymatrix@C+10pt@R-20pt{
    & &   \ds\bigoplus_{n=0}^\infty\wh A_n^*\ar[dr]\ar[dl] & \\
   & \ds\bigoplus_{n=0}^\infty A_n^*\ar[dl] & & B^* \\
\ds\bigoplus_{n=0}^\infty N_n^{\le-n} & & & 
}
\]
is a well-defined diagram in $\Ka(\ca)$ since, in each degree, the sums are 
finite. Hence we obtain a morphism $\bigoplus_{n=0}^\infty A_n^*\to B^*$ in $\Da(\ca)$, and obviously the composite $A_n^*\to\bigoplus_{n=0}^\infty A_n^*\to B^*$ is $\ph_n$ for every $n$.

It remains to prove that such a morphism is unique. This is equivalent to showing that, given a morphism 
$\ph:\bigoplus_{n=0}^\infty A_n^*\to B^*$ in $\Da(\ca)$ such that the composites
$A_n^*\to\bigoplus_{n=0}^\infty A_n^*\to B^*$ vanish for every $n$, then $\ph$ vanishes.

Such a $\ph$ may be represented in $\Ka(\ca)$
by a roof
\[\xymatrix@C+10pt@R-20pt{
 &&   &N^*\ar[dl] \\
 \ds\bigoplus_{n=0}^\infty A_n^*\ar[dr]_-{\alpha} & & B^*\ar[dl]^-\beta& \\
 & \wt B^*& & 
}\]
meaning that $N^*\to B^*\to\wt B^*$ is a distinguished triangle in $\Ka(\ca)$ with $N^*\in\Acy(\ca)$ and $\ph=\beta^{-1}\comp\alpha$ in $\Da(\ca)$. 

If the composite $A_n^*\to\bigoplus_{n=0}^\infty A_n^*\to B^*$ vanishes in
$\Da(\ca)$, then the
morphism $A_n^*\to\bigoplus_{n=0}^\infty A_n^*\to \wt B^*$ 
is a morphism in $\Ka(\ca)$ whose image in $\Da(\ca)$ vanishes,
and hence it must factor in $\Ka(\ca)$ as $A_n^*\to M_n^*\to\wt B^*$, with $M_n^*\in\Acy(\ca)$. 
As in the first part of the proof, the map $A_n^*\to M_n^*$ must
factor as $A_n^*\to M_n^{\le-n}\to M_n^*$, where $M_n^{\le-n}$ is acyclic and vanishes in degrees $>-n$. But then $\bigoplus_{n=0}^\infty A_n^*\to \wt B^*$ 
factors in $\Ka(\ca)$ as
\[
\bigoplus_{n=0}^\infty A_n^*\lto\bigoplus_{n=0}^\infty M_n^{\le-n}\lto\wt B^*,
\]
showing that $\ph$ vanishes in $\Da(\ca)$.
\end{proof}

The following is then a straightforward consequence.

\begin{cor}\label{cor:procoprod}
Let $\ca$ be a small abelian category and let $\{A^i\}_{i\ge0}\subseteq\Ob(\ca)$. Then $\bigoplus_{i=0}^\infty\sh[k+i]{A^i}$ is a coproduct in $\Dm(\ca)$ and $\D(\ca)$, while $\bigoplus_{i=0}^\infty\sh[k-i]{A^i}$ is a product in $\Du(\ca)$ and $\D(\ca)$, for all $k\in\ZZ$.
\end{cor}

\subsection{(Co)products and functors}\label{subsec:coprodfun}

We continue with some technical results which will be used later. In particular, in this section we investigate when special exact functors commute with products and, dually, with coproducts.

If $\ca$ is an abelian category and $n\in\ZZ$, we denote by $\Da(\ca)^{\geq n}$ (resp.\ $\Da(\ca)^{\leq n}$) the full subcategory of $\Da(\ca)$ consisting of objects with trivial cohomologies in degrees $<n$ (resp.\ $>n$).

\begin{prop}\label{prop:prodfun}
Let $\ca$ be a small abelian category and $\fF:\ct\to\Da(\ca)$
an additive functor, where $?=+,\emptyset$. Assume that $\{T_i\}_{1\leq i<\infty}\subseteq\Ob(\ct)$ is such that
\begin{enumerate}
\item[{\rm (i)}] The product $\prod_{i=1}^\infty T_i$ exists in $\ct$;
\item[{\rm (ii)}] For every integer $n>0$ there exists an integer $m(n)>0$ such that
$\fF\left(\prod_{i=m(n)}^\infty T_i\right)\in\D(\ca)^{\geq n}$.
\end{enumerate}
Then the product $\prod_{i=1}^\infty\fF(T_i)$ exists in $\Da(\ca)$
and the canonical map
\[
\fF\left(\prod_{i=1}^\infty T_i\right)\lto\prod_{i=1}^\infty \fF(T_i)
\]
is an isomorphism.
\end{prop}

\begin{proof}
Given $k>0$ and $j\geq k$, the object $\fF(T_j)$ is a direct summand of $\fF\left(\prod_{i=k}^\infty T_i\right)$. Thus, assumption (ii) implies that, for every $n>0$, there exists $m(n)>0$ with $F(T_j)\in\D(\ca)^{\geq n}$ for all $j\geq m(n)$. For $n>0$, we set
\[
A_n:=\fF\left(\prod_{i=m(n)}^{m(n+1)-1} T_i\right)\iso\prod_{i=m(n)}^{m(n+1)-1}\fF(T_i).
\]
By the previous discussion $A_n\in\D(\ca)^{\geq n}$ and \autoref{lem:procoprod} implies that $\prod_{n=1}^{\infty}A_n$ exists in $\Da(\ca)$. Thus
\[
\prod_{i=1}^\infty\fF(T_i)\iso\left(\prod_{i=1}^{m(1)-1}\fF(T_i)\right)\prod\left(\prod_{n=1}^{\infty} A_n\right)
\]
exists in $\Da(\ca)$.

Let us move to the second part of the statement. For $m\geq 2$ the natural map
\[
\theta\colon\fF\left(\prod_{i=1}^\infty T_i\right)\lto\prod_{i=1}^\infty\fF(T_i)
\]
can be identified with the product of the two natural maps 
\[
\theta_1\colon\fF\left(\prod_{i=1}^{m-1} T_i\right)\lto\prod_{i=1}^{m-1}\fF(T_i), \qquad \theta_2\colon\fF\left(\prod_{i=m}^\infty T_i\right)\lto\prod_{i=m}^\infty\fF(T_i).
\]
This clearly implies that $\cone{\theta}\iso\cone{\theta_1}\oplus\cone{\theta_2}$. Since $\fF$ (being additive) commutes with finite products, $\theta_1$ is an isomorphism, whence $\cone{\theta_1}\iso0$. On the other hand, condition (ii) tells us that when $m\gg0$ the object $\cone{\theta_2}$ will belong to $\Da(\ca)^{\geq n}$ with $n$ arbitrarily large. Hence $\cone{\theta}\iso\cone{\theta_2}$ must vanish and $\theta$ must be an isomorphism. 
\end{proof}

Clearly, \autoref{prop:prodfun} has the following dual version whose proof simply consists in reducing to the previous result by passing to the opposite category.

\begin{prop}\label{prop:coprodfun}
Let $\ca$ be a small abelian category and $\fF:\ct\to\Da(\ca)$
an additive functor, where $?=-,\emptyset$. Assume that $\{T_i\}_{1\leq i<\infty}\subseteq\Ob(\ct)$ is such that
\begin{enumerate}
\item[{\rm (i)}] The coproduct $\coprod_{i=1}^\infty T_i$ exists in $\ct$;
\item[{\rm (ii)}] For every integer $n>0$ there exists an integer $m(n)>0$ such that
$\fF\left(\coprod_{i=m(n)}^\infty T_i\right)\in\D(\ca)^{\leq -n}$.
\end{enumerate}
Then the coproduct $\coprod_{i=1}^\infty\fF(T_i)$ exists in $\Da(\ca)$
and the canonical map
\[
\coprod_{i=1}^\infty \fF(T_i)\lto\fF\left(\coprod_{i=1}^\infty T_i\right)
\]
is an isomorphism.
\end{prop}

\section{Generation}\label{sec:generation}

The key idea pursued in this paper is that uniqueness of dg enhancements is tightly related to suitable notions of generations. Those that are of interest in this paper are explained in this section.

\subsection{Well generated triangulated categories}\label{subsec:wellgentriacat}

Let $\ct$ be a triangulated category with small coproducts. For a cardinal $\alpha$, an object $S$ of $\ct$ is \emph{$\alpha$-small} if every map $S\to\coprod_{i\in I}X_i$ in $\ct$ (where $I$ is a small set) factors through $\coprod_{i\in J}X_i$, for some $J\subseteq I$ with $\card{J}<\alpha$. A cardinal $\alpha$ is called \emph{regular} if it is not the sum of fewer than $\alpha$ cardinals, all of them smaller than $\alpha$.

\begin{definition}\label{def:wellgen}
The category $\ct$ is \emph{well generated} if there exists a small set $\cs$ of objects in $\ct$ satisfying the following properties:
\begin{enumerate}
\item[(G1)]\label{G1} An object $X\in\ct$ is isomorphic to $0$, if and only if $\Hom_\ct(S,\sh[j]{X})=0$, for all $S\in\cs$ and all $j\in\ZZ$;
\item[(G2)]\label{G2} For every small set of maps $\{X_i\to Y_i\}_{i\in I}$ in $\ct$, the induced map $\Hom_\ct(S,\coprod_iX_i)\to\Hom_\ct(S,\coprod_i Y_i)$ is surjective for all $S\in\cs$, if $\Hom_\ct(S,X_i)\to\Hom_\ct(S, Y_i)$
is surjective, for all $i\in I$ and all $S\in\cs$;
\item[(G3)]\label{G3} There exists a regular cardinal $\alpha$ such that every object of $\cs$ is $\alpha$-small.
\end{enumerate}
\end{definition}

\begin{remark}\label{rmk:NeeKraeq}
The above notion was originally developed in \cite{N2}. Here we used the equivalent formulation in \cite{K1}. A nice survey on the subject is in \cite{K2}.
\end{remark}

Part of this paper is about enhancements of triangulated category constructed out of Grothendieck categories. For the non-expert reader, let us recall that an abelian category $\cg$ is a Grothedieck category if it is closed under small coproducts, has a small set of generators $\cs$ and the direct limits of short exact sequences are exact in $\cg$. The objects in $\cs$ are generators in the sense that, for any $C$ in $\cg$, there exists an epimorphism $S\epi C$ in $\cG$, where $S$ is a small coproduct of objects in $\cs$. By taking the coproduct of all generators in $\cs$, we can assume that $\cg$ has a single generator $G$.

\begin{ex}\label{ex:Grothcat}
(i) If $X$ is an algebraic stack, the abelian categories $\Mod{\so_X}$ and $\Qcoh(X)$ of $\so_X$-modules and quasi-coherent modules are Grothendieck categories.
	
(ii) If $\ca$ is a small, $\kk$-linear category, we denote by $\Mod{\ca}$ the Grothendieck category of additive functors $\ca{\opp}\to\Mod{\kk}$. For later use, recall that, if $\alpha$ is a regular cardinal, then we denote by $\Lex_{\alpha}(\ca{\opp},\Mod{\kk})$ the full subcategory of $\Mod{\ca}$ consisting of left exact functors which commute with $\alpha$-coproducts.
	
(iii) If $\ca$ is an abelian category, we denote by $\IndC(\ca)$ its Ind-category (see \cite[\S 8]{GrothSGA4}), which is a Grothendieck category. Roughly, it is obtained from $\ca$ by formally adding filtered colimits of objects in $\ca$.
\end{ex}

The following states an important property for the derived category of a Grothendieck category.

\begin{thm}[\cite{N3}, Theorem 0.2]\label{thm:Grothwellgen}
If $\cg$ is a Grothendieck category, then $\D(\cg)$ is well generated.
\end{thm}

A full triangulated subcategory of $\cT$ is \emph{$\alpha$-localizing} if it is closed under $\alpha$-coproducts and under direct summands (the latter condition is automatic if $\alpha>\aleph_0$). An $\alpha$-coproduct is a coproduct of strictly less than $\alpha$ summands. A full subcategory of $\ct$ is \emph{localizing} if it is $\alpha$-localizing for all regular cardinals $\alpha$

When the category $\cT$ is well generated and we want to put emphasis on the cardinal $\alpha$ in (G3) and on $\cS$, we say that $\cT$ is \emph{$\alpha$-well generated} by the set $\cS$. In this situation, following
\cite{K1}, we denote by $\cT^\alpha$ the smallest $\alpha$-localizing
subcategory of $\cT$ containing $\cS$. By \cite{K1,N2}, $\cT^\alpha$ does not depend on the choice of the set $\cS$ which well generates $\cT$.

Let $\cg$ be a Grothendieck category and let $\alpha$ be a sufficiently large regular cardinal. We are interested in describing the category $\D(\cg)^\alpha$. To this end, we denote by $\cg^\alpha$ the full subcategory of $\cg$ consisting of $\alpha$-presentable objects. An object $G$ in $\cg$ is \emph{$\alpha$-presentable} if the functor $\Hom_{\cg}(G,\farg)\colon\cg\to\Mod{\kk}$ preserves $\alpha$-filtered colimits (see, for example, \cite[Section 6.4]{K2}, for the definition of $\alpha$-filtered colimit).

\begin{thm}[\cite{K0}, Corollary 5.5 and Theorem 5.10]\label{thm:alphacompobj}
Let $\cg$ be a Grothendieck category and let $\alpha$ be a sufficiently large regular cardinal.
\begin{itemize}
\item[{\rm (1)}] The category $\cg^\alpha$ is abelian.
\item[{\rm (2)}] There is a natural exact equivalence $\D(\cg)^\alpha\iso\D(\cg^\alpha)$.
\end{itemize}
\end{thm}

The objects in $\cT^\alpha$ are called \emph{$\alpha$-compact}. Thus we will sometimes say that $\cT$ is
\emph{$\alpha$-compactly generated} by the set of \emph{$\alpha$-compact generators} $\cS$.
A very interesting case is when $\alpha=\aleph_0$. Indeed, with this choice, $\cT^\alpha=\cT^c$, the full
triangulated subcategory of compact objects in $\cT$. Recall that an object $C$ in $\cT$ is \emph{compact} if the functor $\cT(C,\farg)$ commutes with small coproducts. Notice that the compact objects in $\cT$ are precisely the $\aleph_0$-small ones.

The analogue of \autoref{thm:Grothwellgen} can be proven for the unseparated derived category.

\begin{thm}[\cite{K0}, Theorem 5.12]\label{thm:sepderwellgen}
If $\cg$ is a Grothendieck category, then $\Dc(\cg)$ is well generated.
\end{thm}

A weaker form of \autoref{thm:alphacompobj} (2) is also available. Indeed, by combining \cite[Theorem 5.12(3)]{K0} with
\autoref{thm:alphacompobj}(2), for $\alpha$ a sufficiently large regular cardinal, there is a quotient functor
\[
\Dc(\cg)^\alpha\lto\D(\cg)^\alpha.
\]
We will not use such a general result in this paper but we will elaborate more on the following easier case.

\begin{ex}\label{ex:sepdercomp}
If $\ca$ is a small abelian category, then one takes the Ind-category $\cg:=\IndC(\ca)$ (see \autoref{ex:Grothcat}). By \cite[Theorem 4.9]{K0}, there is a natural exact equivalence $\Dc(\cg)^c\iso\Db(\ca)$.
\end{ex}

\subsection{Generating derived categories}\label{subsec:gendercat}

In the general case when $\ca$ is any abelian category, not necessarily Grothendieck, we need a different approach to the generation of $\Da(\ca)$.

Let us first recall the following rather general definition (see \cite{BB}).

\begin{definition}\label{def:gen}
Let $\ct$ be a triangulated category and let $\cs\subset\Ob(\ct)$.
We define
\begin{enumerate}
\item $\gen\cs1$ is the collection of all direct summands of finite coproducts of shifts of objects in $\cs$;
\item $\gen\cs{n+1}$ consists of all direct summands of objects $T\in\ct$, for which there exists a distinguished triangle $T_1\to T\to T_2$ with $T_1\in\gen \cs n$ and $T_2\in\gen\cs1$.
\end{enumerate}
We set $\gen\cs\infty$ for the full subcategory consisting of all objects $T$ in $\ct$ contained in $\gen\cs n$, for some $n$.
\end{definition}

In our special case, we can prove the following.

\begin{prop}\label{prop:genK}
Let $\Va(\ca)\subset\Ka(\ca)$ be as 
defined in the opening paragraph of Subsection~\ref{subsec:cat}.
For $?=b,+,-,\emptyset$, we have that $\gen{\Va(\ca)}3=\Ka(\ca)$.
\end{prop}

\begin{proof}
Let $A^*\in\Ob(\Ka(\ca))$, which we write
as a complex
\[
\xymatrix@C+15pt{
\cdots \ar[r]& A^{-2}  \ar[r]& A^{-1}  \ar[r]& A^{0}  \ar[r]
&  A^{1} \ar[r]& A^{2} \ar[r] &\cdots
}
\]
Let $K^i$ be the kernel of the differential
$A^i\to A^{i+1}$. Then the map $A^{i-1}\to A^i$ factors uniquely as 
 $A^{i-1}\mor{\alpha^i} K^i\mono A^i$.
 
This yields the morphism 
\[
\xymatrix@C+50pt{
\ds\bigoplus_{i\in\ZZ}\sh[-i]{A^{i-1}}
    \ar[r]^-{\bigoplus_{i\in\ZZ}\alpha^i} &
\ds\bigoplus_{i\in\ZZ}\sh[-i]{K^{i}} 
}
\]
in $\Va(\ca)$. Denote by $C^*$ its mapping cone. It is clear that $C^*\in\gen{\Va(\ca)}2$ and it is the direct sum over 
$i\in\ZZ$ of the
complexes
\[
\xymatrix@C+15pt{
\cdots \ar[r]& 0  \ar[r]& A^{i-1}  \ar[r]^-{\alpha^i}& K^{i}  \ar[r]
&  0 \ar[r]& \cdots
}
\]

Now consider the cochain map 
\[
\xymatrix@C+40pt{
\ds\bigoplus_{i\in\ZZ}\sh[-i]{K^{i}} \ar[r]^-{\ph+\psi} & C^*
}
\]
whose components, out of $\sh[-i]{K^{i}}$, are (respectively) $\ph^i$ 
as below
\begin{equation}\label{eqn:phi}
\xymatrix@C+15pt{
\cdots \ar[r]& 0  \ar[r]\ar[d]& 0\ar[d]  \ar[r]& K^{i}\ar@{=}[d]  \ar[r]
&  0  \ar[r]\ar[d]& 0 \ar[d]\ar[r]& \cdots\\
\cdots \ar[r]& 0  \ar[r]& A^{i-1}  \ar[r]& K^{i}  \ar[r]&0  \ar[r]
&  0 \ar[r]& \cdots
}
\end{equation}
and $\psi^i$ as below
\begin{equation}\label{eqn:psi}
\xymatrix@C+15pt{
\cdots \ar[r]& 0  \ar[r]\ar[d]& 0  \ar[r]\ar[d]& K^{i}\ar@{^{(}->}[d]  \ar[r]& 0\ar[d]  \ar[r]
&  0 \ar[d]\ar[r]& \cdots\\
\cdots \ar[r]& 0  \ar[r]&0  \ar[r]& A^{i}  \ar[r]& K^{i+1}  \ar[r]
& 0 \ar[r]& \cdots
}
\end{equation}

It can be easily checked that the mapping cone of the morphism $\ph+\psi$ is isomorphic to the direct sum of the complex $A^*$ and of complexes of the form
\[
\xymatrix@C+15pt{
\cdots \ar[r]& 0  \ar[r]& K^i  \ar[r]& K^{i}  \ar[r]
&  0 \ar[r]& \cdots
}
\]
In other words, $\cone{\ph+\psi}\iso A^*$ in $\Ka(\ca)$. Therefore, as $C^*$ belongs to $\gen{\Va(\ca)}2$ and $\ph+\psi$ is a morphism from an object of $\Va(\ca)$ to $C^*$, we have that $A^*\in\gen{\Va(\ca)}3$.
\end{proof}

The reader might wish to compare the proof of \autoref{prop:genK} above
with the proof of \cite[Proposition 7.22]{R}; there are similarities.

As a straightforward consequence of \autoref{prop:genK} and of the fact that $\Da(\ca)$ is a quotient of $\Ka(\ca)$, as explained in \autoref{subsec:cat}, we obtain the following.

\begin{cor}\label{cor:genD}
Let $\Ba(\ca)\subset\Da(\ca)$ be as defined in the
paragraph between Remarks~\ref{rmk:univprop2} and \ref{rmk:univprop1}.
For $?=b,+,-,\emptyset$, we have that $\gen{\Ba(\ca)}3=\Da(\ca)$.
\end{cor}

\section{Dg categories and enhancements}\label{sect:dgenhancements}

We briefly introduce dg categories and some basic machinery in \autoref{subsect:dgcat}. Next we describe some constructions that will be important in the rest of the paper: Drinfeld quotients and h-flat resolutions (\autoref{subsec:hprojhflat}), the model structure and homotopy pullbacks (\autoref{subsec:modelpullbacks}) and, finally, localizations for dg categories (\autoref{subsec:someextria}). Dg enhancements, their uniqueness and dependence on the universe where the categories are defined are the contents of \autoref{subsect:dgenuniq}.

\subsection{Dg categories}\label{subsect:dgcat}

A \emph{dg category} is a $\kk$-linear category $\cc$ such that the morphism spaces $\cc\left(A,B\right)$ are complexes of $\kk$-modules and the composition maps $\cc(B,C)\otimes_{\kk}\cc(A,B)\to\cc(A,C)$ are morphisms of complexes, for all $A,B,C$ in $\Ob(\cc)$.

A \emph{dg functor} $\fF\colon\cc_1\to\cc_2$ between two dg categories is a $\kk$-linear functor such that the maps $\cc_1(A,B)\to\cc_2(\fF(A),\fF(B))$ are morphisms of complexes, for all $A,B$ in $\Ob(\cc_1)$.

The \emph{underlying category} $Z^0(\cc)$ (respectively, the \emph{homotopy category} $H^0(\cc)$) of a dg category $\cc$ is the $\kk$-linear category with the same objects of $\cc$ and such that $Z^0(\cc)(A,B):=Z^0(\cc(A,B))$ (respectively, $H^0(\cc)(A,B):=H^0(\cc(A,B))$), for all $A,B$ in $\Ob(\cc)$ (with composition of morphisms naturally induced from the one in $\cc$). Two objects of $\cc$ are \emph{dg isomorphic} (respectively, \emph{homotopy equivalent}) if they are isomorphic in $Z^0(\cc)$ (respectively, $H^0(\cc)$). One can also define $Z(\cc)$ (respectively $H(\cc)$) to be the graded (namely dg with trivial differential) category whose objects are the same as those of $\cc$, while $Z(\cc)(A,B):=\oplus_{i\in\ZZ}Z^i(\cc(A,B))$ (respectively, $H(\cc)(A,B):=\oplus_{i\in\ZZ}H^i(\cc(A,B))$), for all $A,B$ in $\Ob(\cc)$.

\begin{ex}\label{dgC}
I $\ca$ is a $\kk$-linear category, there is a natural dg category $\dgCa(\ca)$ such that $H^0(\dgCa(\ca))=\Ka(\ca)$, for $?=b,+,-,\emptyset$. Explicitly,
\[
\Hom_{\dgCa(\ca)}(A^*,B^*)^n:=\prod_{i\in\ZZ}\Hom_\ca(A^i,B^{n+i})
\]
for every $A^*,B^*\in\Ob(\dgCa(\ca))$ and for every $n\in\ZZ$. While the composition of morphisms is the obvious one, the differential is defined on a homogeneous element $f\in\Hom_{\dgCa(\ca)}(A^*,B^*)^n$ by $d(f):=d_B\comp f-(-1)^nf\comp d_A$.
\end{ex}

A dg functor $\fF\colon\cc_1\to\cc_2$ induces a $\kk$-linear functor $H^0(\fF)\colon H^0(\cc_1)\to H^0(\cc_2)$. We say that  $\fF$ is a \emph{quasi-equivalence} if the maps $\cc_1(A,B)\to\cc_2(\fF(A),\fF(B))$ are quasi-isomorphisms, for all $A,B$ in $\Ob(\cc_1)$, and $H^0(\fF)$ is an equivalence.

If $\UU$ is a universe, we denote by $\dgCat_\UU$ (or simply by $\dgCat$, if there can be no ambiguity about $\UU$) the category whose objects are $\UU$-small dg categories and whose morphisms are dg functors. It is known (see \cite{Ta}) that $\dgCat$ has a model structure whose weak equivalences are quasi-equivalences and such that every object is fibrant. We denote by $\Hqe$ (or $\Hqe_\UU$, if needed) the corresponding homotopy category, namely the localization of $\dgCat$ with respect to quasi-equivalences. Since $H^0$ sends quasi-equivalences to equivalences, for every morphism $f\colon\cc_1\to\cc_2$ in $\Hqe$ there is a $\kk$-linear functor $H^0(f)\colon H^0(\cc_1)\to H^0(\cc_2)$, which is well-defined up to isomorphism.

Dg functors between two dg categories $\cc_1$ and $\cc_2$ form in a natural way the objects of a dg category $\dgHom(\cc_1,\cc_2)$. For every dg category $\cc$ we set $\dgMod{\cc}:=\dgHom(\cc\opp,\dgC(\Mod{\kk}))$ and call its objects (right) dg $\cc$-modules. Observe that $\dgMod{\kk}$ can be identified with $\dgC(\Mod{\kk})$.

For every dg category $\cc$ the map defined on objects by $A\mapsto\cc(\farg,A)$ extends to a fully faithful dg functor $\dgYon\colon\cc\to\dgMod{\cc}$ (the dg Yoneda embedding). It is easy to see that the image of $\dgYon$ is contained in the full dg subcategory $\hproj{\cc}$ of $\dgMod{\cc}$ whose objects are \emph{h-projective} dg $\cc$-modules. By definition, $M\in\Ob(\dgMod{\cc})$ is h-projective if $H^0(\dgMod{\cc})(M,N)=0$ for every $N\in\Ob(\dgAc{\cc})$, where $\dgAc{\cc}$ is the full dg subcategory of $\dgMod{\cc}$ whose objects are acyclic (in the sense that $N(A)$ is an acyclic complex for every $A\in\Ob(\cc)$).

If $\fF\colon\cc_1\to\cc_2$ is a dg functor, composition with $\fF\opp$ yields a dg functor $\Res(\fF)\colon\dgMod{\cc_2}\to\dgMod{\cc_1}$. It turns out that there exists also a dg functor $\Ind(\fF)\colon\dgMod{\cc_1}\to\dgMod{\cc_2}$ which is left adjoint to $\Res(\fF)$ and such that $\Ind(\fF)\comp\dgYon[\cc_1]\iso\dgYon[\cc_2]\comp\fF$. Moreover, $\Ind(\fF)$ preserves h-projective dg modules, and $\Ind(\fF)\colon\hproj{\cc_1}\to\hproj{\cc_2}$ is a quasi-equivalence if $\fF$ is. This last fact clearly implies that a(n iso)morphism $f\colon\cc_1\to\cc_2$ in $\Hqe$ induces a(n iso)morphism $\Ind(f)\colon\hproj{\cc_1}\to\hproj{\cc_2}$ in $\Hqe$.

As is explained for instance in \cite{BLL}, there is a notion of formal shift by an integer $n$ of an object $A$ in a dg category $\cc$ (denoted, as usual, by $\sh[n]{A}$). Similarly, one can define the formal cone of a morphism $f$ in $Z^0(\cc)$ (denoted, as usual, by $\cone{f}$). Now, shifts and cones need not exist in an arbitrary dg category, but, when they do, they are unique up to dg isomorphism and they are preserved by dg functors. The following property of the cone of a morphism will be useful later.

\begin{lem}\label{coneprop}
Let $A\mor{f}B\mor{g}C$ be morphisms in $Z^0(\cc)$ such that $\cone{f}$ exists and $g\comp f$ is a coboundary. Then there exists $h\colon\cone{f}\to C$ in $Z^0(\cc)$ such that $g=h\comp j$, where $j\colon B\to\cone{f}$ is the natural morphism.
\end{lem}

\begin{proof}
See \cite[Proposition 2.3.4]{Ge}.
\end{proof}

\begin{definition}
A dg category $\cc$ is \emph{strongly pretriangulated} if $\sh[n]{A}$ and $\cone{f}$ exist (in $\cc$), for every $n\in\ZZ$, every object $A$ of $\cc$ and every morphism $f$ of $Z^0(\cc)$.

A dg category $\cc$ is \emph{pretriangulated} if there exists a quasi-equivalence $\cc\to\cc'$ with $\cc'$ strongly pretriangulated.
\end{definition}

\begin{remark}
If $\cc$ is a pretriangulated dg category, then $H^0(\cc)$ is a triangulated category in a natural way. If $f$ is a morphism in $\Hqe$ between two pretriangulated dg categories, then the functor $H^0(f)$ is exact.
\end{remark}

If $\cc$ is a dg category, $\dgMod{\cc}$, $\dgAc{\cc}$ and $\hproj{\cc}$ are strongly pretriangulated dg categories. Moreover, the (triangulated) categories $H^0(\dgMod{\cc})$, $H^0(\dgAc{\cc})$ and $H^0(\hproj{\cc})$ have arbitrary coproducts, and there is a semi-orthogonal decomposition 
\begin{equation}\label{eqn:semiort}
H^0(\dgMod{\cc})=\ort{H^0(\dgAc{\cc}),H^0(\hproj{\cc}}.
\end{equation}
This clearly implies that there is an exact equivalence between $H^0(\hproj{\cc})$ and the Verdier quotient $\dgD(\cc):=H^0(\dgMod{\cc})/H^0(\dgAc{\cc})$ (which is by definition the \emph{derived category} of $\cc$). 

For every dg category $\cc$ we will denote by $\Pretr{\cc}$ (respectively, $\Perf{\cc}$) the smallest full dg subcategory of $\hproj{\cc}$ containing $\dgYon(\cc)$ and closed under homotopy equivalences, shifts, cones (respectively, also direct summands in $H^0(\hproj{\cc})$). It is easy to see that $\Pretr{\cc}$ and $\Perf{\cc}$ are strongly pretriangulated and that $\cc$ is pretriangulated if and only if $\dgYon\colon\cc\to\Pretr{\cc}$ is a quasi-equivalence. Moreover, $\Pretr{\cc}\subseteq\Perf{\cc}$ and $H^0(\Perf{\cc})$ can be identified with the idempotent completion $\ic{H^0(\Pretr{\cc})}$ of $H^0(\Pretr{\cc})$. Hence $\dgYon\colon\cc\to\Perf{\cc}$ is a quasi-equivalence if and only if $\cc$ is pretriangulated and $H^0(\cc)$ is idempotent complete.

\begin{remark}\label{wic}
Recall that an additive category $\ca$ is \emph{idempotent complete} if every idempotent (namely, a morphism $e\colon A\to A$ in $\ca$ such that $e^2=e$) splits, or, equivalently, has a kernel. Every additive category $\ca$ admits a fully faithful and additive embedding $\ca\mono\ic{\ca}$, where $\ic{\ca}$ is an idempotent complete additive category, with the property that every object of $\ic{\ca}$ is a direct summand of an object from $\ca$. The category $\ic{\ca}$ (or, better, the functor $\ca\to\ic{\ca}$) is called the \emph{idempotent completion} of $\ca$. It can be proved (see \cite{BS}) that, if $\ct$ is a triangulated category, then $\ic{\ct}$ is triangulated as well (and $\ct\mono\ic{\ct}$ is exact).

If $\cc$ is a dg category, then $H^0(\Perf{\cc})$ is idempotent complete, and from this it is easy to deduce that $Z^0(\Perf{\cc})$ is also idempotent complete.

If $\ca$ is an abelian category, it follows from \cite{BS,Sc2} that $\Da(\ca)$ is 
idempotent complete for $?=b,+,-,\emptyset$. More precisely: for $?\in\{-,+\}$
the result may be found in \cite[Lemma~2.4]{BS}, for $?=b$ see
\cite[Lemma~2.8]{BS}, while for 
$?=\emptyset$
see Theorem 6 of Section 10 in \cite{Sc2} combined with Lemma 7 of Section 9 in the same paper. 

Observe that, by \autoref{rmk:equivKD} combined with the paragraph
above, the categories $\Ka(\ca)$ are also idempotent
complete---as long as $\ca$ is abelian and with $?=b,+,-,\emptyset$.
\end{remark}

\subsection{Drinfeld quotients and h-flat resolutions}\label{subsec:hprojhflat}

Let $\cc$ be a dg category and $\cd\subseteq\cc$ a full dg subcategory. As explained in \cite[Section 3.1]{Dr}, one can form the \emph{Drinfeld quotient} of $\cc$ by $\cd$ which we denote by $\cc/\cd$. This is a dg category and its construction goes roughly as follows: given $D\in\Ob(\cd)$, we formally add a morphism $f_D\colon D\to D$ of degree $-1$ and we set $d(f_D)=\id_D$. 

If $\cc$ is pretriangulated and $\cd$ is a full pretriangulated dg subcategory, then $\cc/\cd$ is pretriangulated. In this case the natural dg functor $\cc\to\cc/\cd$ induces an exact functor $H^0(\cc)\to H^0(\cc/\cd)$, which sends to zero the objects of $\cd$. Thus it factors through the Verdier quotient $H^0(\cc)\to H^0(\cc)/H^0(\cd)$, yielding an exact functor
\begin{equation}\label{eq:quots}
H^0(\cc)/H^0(\cd)\lto H^0(\cc/\cd),
\end{equation}
which need not be an equivalence, in general.

\begin{definition}\label{def:hflat} We remind the reader of the terminology of
\cite{Dr}.
A dg category $\cc$ is \emph{h-flat} if, for all $C_1,C_2\in\Ob(\cc)$, the complex  $\Hom_\cc(C_1,C_2)$ is homotopically flat over $\kk$. The homotopic
flatness of $\Hom_\cc(C_1,C_2)$  means that, for any acyclic complex $M$ of $\kk$-modules, $\Hom_\cc(C_1,C_2)\otimes_\kk M$ is acyclic.
\end{definition}

\begin{ex}\label{ex:hflat}
If $\kk$ is a field, then every dg category is clearly h-flat.
\end{ex}

As a special case of \cite[Theorem 3.4]{Dr}, we have that if $\cc$ is an h-flat pretriangulated dg category and $\cd$ is a full pretriangulated subcategory of $\cc$, then \eqref{eq:quots} is an exact equivalence.

If $\cc$ is not h-flat, Drinfeld shows in \cite[Lemma B.5]{Dr} that one can construct an h-flat dg category $\widetilde\cc$  with a quasi-equivalence $\fI_\cc\colon\widetilde\cc\to\cc$. One can then define $\widetilde\cd$ to be the full dg subcategory $\fI_\cc^{-1}\cd\subset\widetilde\cc$, and take the morphism $q\in\Hom_\Hqe\left(\cc,\widetilde\cc/\widetilde\cd\right)$ represented as
\[
\xymatrix@C+20pt@R-20pt{
&\widetilde\cc\ar[dl]_-{\fI_\cc}\ar[dr] &\\
\cc & & \widetilde\cc/\widetilde\cd, 
}
\]
where the dg functor on the right is the natural one mentioned above.

This construction satisfies the following universal property, which is a special instance of \cite[Main Theorem]{Dr}. Assume that $\cc'$ is a pretriangulated dg category and $f\in\Hom_\Hqe(\cc,\cc')$ is such that $H^0(f)$ sends the objects of $\cd$ to zero. Then there is a unique $\overline{f}\in\Hom_\Hqe\left(\widetilde\cc/\widetilde\cd,\cc'\right)$ making the diagram
\begin{equation}\label{eqn:triauniv}
\xymatrix{
\cc\ar[r]^-{q}\ar[dr]_{f}& \widetilde\cc/\widetilde\cd\ar[d]^-{\overline{f}}\\
& \cc' 
}
\end{equation}
commute in $\Hqe$.

\medskip

In the rest of this section we describe two variants of $\widetilde\cc$ 
with properties that we will need in the rest of the paper. We start with a dg category $\cc$, we let $\ct$ be the graded
category $\ct=H(\cc)$, and the aim is to produce two sequences of dg categories and faithful dg functors
\[
\cc_0\mono\cc_1\mono\cc_2\mono\cc_3\mono\cdots,\qquad\qquad
\cc'_0\mono\cc'_1\mono\cc'_2\mono\cc'_3\mono\cdots,
\]
together with compatible dg functors $\fI_n\colon\cc_n\to\cc$
and $\fI'_n\colon\cc'_n\to\cc$ which are the identity on objects. Then we set 
$\hff{\cc}$ and $\sma{\cc}$ to be the 
respective colimits, with the induced dg functors $\hff{\fI}_\cc\colon\hff{\cc}\equiva\cc$
and $\sma{\fI}_\cc\colon\sma{\cc}\equiva\cc$.

We define $\cc_0=\cc'_0$ to be the discrete $\kk$-linear category with the same objects as $\cc$. This means that
\[
\Hom_{\cc_0}(A,B):=
\begin{cases}
\kk\cdot\id_A & \text{if $A=B$} \\
0 & \text{otherwise.}
\end{cases}
\]
The dg functor $\fI_0=\fI'_0$ is the obvious one which acts as the identity on objects and morphisms.

For $n=1$, we set
\[
D^1_\cc(A,B):=\left\{(f,0)\in\Hom_{\cc}(A,B)\times\Hom_{\cc_{0}}(A,B)\st d(f)=0\right\}.
\]
Now the composite
\[\xymatrix{
D^1_\cc(A,B)\ar[r] &Z\big(\Hom_{\cc}(A,B)\big)\ar[r] & H\big(\Hom_{\cc}(A,B)\big)
\ar[r]^-\sim &\Hom_{\ct}(A,B)
}\]
is surjective by construction, hence we may choose a splitting. We let
$\overline D^1_\cc(A,B)\subset D^1_\cc(A,B)$ be a subset such that 
the composite $\overline D^1_\cc(A,B)\to D^1_\cc(A,B)\to\Hom_{\ct}(A,B)$
is an isomorphism.
And we define $\cc_1$ so that $\Hom_{\cc_1}^{}(A,B)$ is
the graded $\kk$-module freely generated by the basis $D^1_\cc(A,B)-\{0\}$,
with the composition being the obvious on basis vectors.
And $\cc'_1$ is the graded $\kk$-linear category freely 
generated\footnote{The notion of graded ($\kk$-linear) category which is freely generated over another graded category which we use here is the same as the one in \cite[Lemma B.5]{Dr} (see also \cite[Section 3.1]{Dr})).}
 over $\cc_0$ by the sets
$\overline D^1_\cc(A,B)$.
The differentials of $\cc_1$ and $\cc'_1$ are trivial.

We continue for $n\geq 2$ by defining inductively, for all $A,B\in\Ob(\cc)$,
\[
D^n_\cc(A,B):=\left\{(f,b)\in\Hom_{\cc}(A,B)\times\Hom_{\cc_{n-1}}(A,B)\st d(f)=\fI_{n-1}(b)\right\}.
\]
The definition of $\overline D^n_\cc(A,B)$ is slightly more delicate. We 
begin by copying the procedure above with $\cc'_{n-1}$ in place of
$\cc_{n-1}$, setting
\[
\widehat D^n_\cc(A,B):=\left\{(f,b)\in\Hom_{\cc}(A,B)\times\Hom_{\cc'_{n-1}}(A,B)\st d(f)=\fI'_{n-1}(b)\right\}.
\]
And then we observe that $\widehat D^n_\cc(A,B)$ surjects to the kernel of the 
surjective map
\[
Z\left(\Hom_{\cc'_{n-1}}(A,B)\right)\longrightarrow\Hom_{\ct}(A,B),
\]
allowing us to choose a subset $\overline D^n_\cc(A,B)\subset 
\widehat D^n_\cc(A,B)$
so that the composite 
\begin{equation}\label{eqn:contrmor}
\xymatrix{
\overline D^n_\cc(A,B)\ar[r] & \widehat D^n_\cc(A,B)\ar[r] &
\mathrm{Ker}\left(Z\left(\Hom_{\cc'_{n-1}}(A,B)\right)\to\Hom_{\ct}(A,B)\right)
}
\end{equation}
is an isomorphism.
And now we are ready:
$\cc_n$ is the graded $\kk$-linear category freely generated over $\cc_{n-1}$ 
by the sets $D^n_\cc(A,B)$, while
$\cc'_n$ is the graded $\kk$-linear category freely generated over $\cc'_{n-1}$ 
by the sets $\overline D^n_\cc(A,B)$. 
It is easy to show that $\cc_n,\cc'_n$ become dg categories if we extend the 
differential on $\cc_{n-1}$ by sending a generator $(f,b)$, in
either $D_n(A,B)$ or $\overline D^n_\cc(A,B)$, to $b$.

For $n\ge1$, the dg functor $\fI_n\colon\cc_n\to\cc$ is determined by setting $\fI_n\rest{\cc_{n-1}}:=\fI_{n-1}$ and $\fI_n((f,b)):=f$, for all $(f,b)\in D^n_\cc(A,B)$.
Similarly the dg functor $\fI'_n\colon\cc'_n\to\cc$ is given by setting 
$\fI'_n\rest{\cc'_{n-1}}:=\fI'_{n-1}$ and $\fI'_n((f,b)):=f$, for all $(f,b)\in\overline D^n_\cc(A,B)$.

The following are easy consequences of the definitions.

\begin{lem}\label{lemmasmall}
If the category $\ct$ is $\UU$-small, if $\UU\in\VV$ is a larger
universe, and if $\cc$ is a $\VV$-small dg category with $H(\cc)\iso\ct$, then 
the dg functor $\sma{\fI}:\sma{\cc}\to\cc$ is a quasi-eqivalence with
$\sma{\cc}$ a $\UU$-small dg category. Moreover, if $\cc$ is pretriangulated then so is
$\sma{\cc}$.
\end{lem}

\begin{proof}
Obvious by construction, the passage from $\cc'_{n-1}$ to $\cc'_n$ keeps tight
control of the size of sets that come up (see, in particular, \eqref{eqn:contrmor}).
\end{proof}

\begin{remark}\label{rmk:Usm}
It is easy to see that if $\cc$ is pretriangulated (closed by shifts is enough) then $H(\cc)$ is $\UU$-small if and only if $H^0(\cc)$ is $\UU$-small.
\end{remark}

\begin{prop}\label{prop:genDr}
The construction taking a dg category $\cc$ to the dg functor 
$\hff{\fI}_\cc\colon\hff{\cc}\to\cc$ satisfies the following properties:
\begin{enumerate}
\item[{\rm (1)}] For every dg category
$\cc$, the dg functor $\hff{\fI}_\cc$ is a quasi-equivalence and 
$\hff{\cc}$ is h-flat.
\item[{\rm (2)}] 
$\hff{(-)}$ is a functor from the category of dg categories to itself,
and $\hff{\fI}:\hff{(-)}\to\id$ is a natural
transformation. This means that, if
$\cc$ and $\cd$ are dg categories and $\fF\colon\cc\to\cd$ is a dg functor,
there is a canonically defined dg functor $\hff{\fF}\colon\hff{\cc}\to\hff{\cd}$ 
making the following diagram commutative
\[
\xymatrix@C+10pt{
\hff{\cc}\ar[r]^-{\hff{\fF}}\ar[d]_-{\hff{\fI}_{\cc}}&\hff{\cd}\ar[d]^-{\hff{\fI}_{\cd}}\\
\cc\ar[r]^-{\fF}&\cd.
}
\]
Moreover the construction taking $\fF\colon\cc\to\cd$ to 
$\hff{\fF}\colon\hff{\cc}\to\hff{\cd}$ respects composition and identities. 
\item[{\rm (3)}] If $\cc$ is pretriangulated, then $\hff{\cc}$ is pretriangulated.
\end{enumerate}
\end{prop}

\begin{proof}
Part (1) has the same proof as in \cite[Lemma B.5]{Dr} since properties (i)--(iv) in the construction of \cite[Lemma B.5]{Dr} are obviously satisfied by our construction as well (note that (iii) is verified starting from $n=2$). The functoriality and naturality in (2) are clear; the construction is choice-free. Indeed, $\fF$ induces a natural dg functor $\fF_0\colon\cc_0\to\cd_0$ acting as $\fF$ on the objects and as the identity on the morphisms. Similarly, $\fF$ induces a dg functor $\fF_n\colon\cc_n\to\cd_n$ extending $\fF_{n-1}$ and which is determined by the assignment
\[
(f,b)\in D^n_\cc(A,B)\mapsto (\fF(f),\fF_{n-1}(b))\in D^n_\cd(\fF(A),\fF(B)).
\]
We set $\hff{\fF}$ to be the colimit of the dg functors $\fF_n$. Part (3) is easy from (1).
\end{proof}

Suppose now that $\cc$ is a dg category and $\cd\subseteq\cc$ is a full dg subcategory. Let moreover $\cd'$ be the full dg subcategory of $\hff{\cc}$ with the same objects as $\cd$. When $\cc$ and $\cd$ are pretriangulated there are natural exact equivalences $H^0(\cc)/H^0(\cd)\iso H^0(\hff{\cc})/H^0(\cd')\iso H^0\left(\hff{\cc}/\cd'\right)$. Furthermore, not only are $\widetilde\cc$ and $\hff{\cc}$ isomorphic in $\Hqe$, but by the universal property of Drinfeld quotients mentioned above the dg categories $\widetilde\cc/\widetilde\cd$ and $\hff{\cc}/\cd'$ are isomorphic in $\Hqe$.

\begin{remark}\label{rmk:univpropquot}
Now let $\cc$ be a pretriangulated dg category and $\cd$ a full pretriangulated dg subcategory of $\cc$. Denoting by  $\fQ\colon\hff{\cc}\to\hff{\cc}/\cd'$ the natural dg functor, assume that we have a dg functor $\fF\colon\hff{\cc}\to\cc'$ of pretriangulated dg categories such that $H^0(\fF)$ sends the objects of $\cd'$ to zero. Then the universal property pictured in \eqref{eqn:triauniv} can be made more explicit: there exists a dg functor $\overline{\fF}\colon\hff{\cc}/\cd'\to\cc'$ such that the diagram
\[
\xymatrix{
\hff{\cc}\ar[r]^-{\fQ}\ar[dr]_{\fF}& \hff{\cc}/\cd'\ar[d]^-{\overline{\fF}}\\
& \cc' 
}
\]
is commutative in $\dgCat$. The existence of $\overline{\fF}$ is simple enough to see from the construction of
$\hff{\cc}/\cd'$ and $\fQ$ in \cite[Section 3.1]{Dr}. Indeed, as $\hff{\cc}/\cd'$ has the same objects as $\hff{\cc}$, one sets $\overline{\fF}(C):=\fF(C)$, for all $C$ in $\hff{\cc}/\cd'$. If $D\in\cd'$ is an object, then $\Hom_{\cc'}(\fF(D),\fF(D))$ is acyclic, allowing us to choose a degree -1 morphism
$f(D)\colon \fF(D)\to \fF(D)$ with $d(f(D))=\id$. Then the dg functor $\overline{\fF}$ extends $\fF$ on morphisms and takes each degree $-1$ morphism $f_D\colon D\to D$, in the definition of the category $\hff{\cc}/\cd'$, to the degree -1 morphism $f(D)$ in $\cc'$.

The construction of the dg functor $\overline{\fF}$ in the paragraph above depends on making choices and is not unique; the uniqueness is only up to homotopy, meaning in the category $\Hqe$. But if $\cc_1$ and $\cc_2$ are pretriangulated dg categories with full pretriangulated dg subcategories $\cd_1\subset\cc_1$ and 
$\cd_2\subset\cc_2$, and $\fG\colon\cc_1\to\cc_2$ is a dg functor with $\fG(\cd_1)\subseteq\cd_2$, then $\hff{\fG}$ induces a natural dg functor $\hff{\cc}_1/\cd'_1\to\hff{\cc}_2/\cd'_2$. Both the passage from $\cc$ to $\hff{\cc}$ and Drinfeld's construction of the quotient $\hff{\cc}/\cd'$ are manifestly functorial.
\end{remark}

In the rest of the paper, we will sometimes sloppily denote by $\cc/\cd$ either the Drinfeld quotient of $\cc$ by $\cd$ when $\cc$ is h-flat or of $\hff{\cc}$ by $\cd'$ otherwise.

\subsection{The model structure and homotopy pullbacks}\label{subsec:modelpullbacks}

Let us recall that the pullback of a diagram
\begin{equation}\label{eq:pullback}
\cc_1\mor{\fF_1}\cd\overset{\fF_2}{\longleftarrow}\cc_2
\end{equation}
in $\dgCat$ is given by a dg category $\cc_1\times_\cd\cc_2$ defined in the obvious way. Unfortunately this notion of pullback does not behave well with respect to quasi-equivalences.

To overcome this issue, one has to note that, by the work of Tabuada \cite{Ta},  $\dgCat$ has a model category structure whose weak equivalences are the quasi-equivalences. We refer to \cite{Hov} for an exhaustive introduction to model categories. Here we content ourselves with some remarks about the special case of $\dgCat$. In particular, in Tabuada's model structure all dg categories are fibrant but not all of them are cofibrant. Furthermore, such a model structure is \emph{right proper}, i.e.\ every pullback of a weak equivalence along a fibration is a weak equivalence, thanks to the fact that all objects are fibrant (see \cite[Corollary 13.1.3]{Hir}). Finally, $\Hqe$ can be reinterpreted as the homotopy category of $\dgCat$ with respect to such a model structure.

As is explained for instance in \cite[Section 13.3]{Hir}, one can then consider the \emph{homotopy pullback} $\cc_1\times^h_\cd \cc_2$ of \eqref{eq:pullback}. By definition $\cc_1\times^h_\cd \cc_2:=\cc'_1\times_\cd\cc'_2$ is the usual pullback of a diagram
\begin{equation}\label{eq:pullback2}
\cc'_1\mor{\fF'_1}\cd\overset{\fF'_2}{\longleftarrow}\cc'_2,
\end{equation}
where at least one among $\fF'_1$ and $\fF'_2$ is a fibration and (for $i=1,2$) $\fF_i=\fF'_i\comp\fI_i$ with $\fI_i\colon\cc_i\to\cc'_i$ a quasi-equivalence. Notice that such a factorization of $\fF_i$ always exists, and in fact one could  choose $\fI_i$ to be a cofibration as well. The homotopy pullback does not depend, up to isomorphism in $\Hqe$, on the choice of the diagram \eqref{eq:pullback2}.

Let us spell out an explicit description of $\cc_1\times^h_\cd \cc_2$. We can take $\fF'_2=\fF_2$ and factor only $\fF_1$ as follows. Define $\cc'_1$ to be the dg category whose objects are triples, $(C_1,D,f)$ where $C_1\in\Ob(\cc_1)$, $D\in\Ob(\cd)$ and $f\colon\fF_1(C_1)\to D$ is a homotopy equivalence. A morphism of degree $n$ from $(C_1,D,f)$ to $(C'_1,D',f')$ in $\cc'_1$ is given by a triple $(a_1,b,h)$ with $a_1\in\Hom_{\cc_1}(C_1,C'_1)^n$, $b\in\Hom_\cd(D,D')^n$ and $h\in\Hom_\cd(\fF_1(C_1),D')^{n-1}$. The differential is defined by
\[
d(a_1,b,h):=(d(a_1),d(b),d(h)+(-1)^n(f'\comp\fF_1(a_1)-b\comp f))
\]
and the composition by
\[
(a'_1,b',h')\comp(a_1,b,h):=(a'_1\comp a_1,b'\comp b,b'\comp h+(-1)^nh'\comp \fF_1(a_1)).
\]
The dg functor $\fI_1$ is defined by $\fI_1(C_1):=(C_1,\fF_1(C_1),\id_{\fF_1(C_1)})$ on objects and $\fI_1(a_1):=(a_1,\fF_1(a_1),0)$ on morphisms. On the other hand, the dg functor $\fF'_1$ is defined as projection on the second component both on objects and on morphisms. It is not difficult to check that $\fI_1$ is a quasi-equivalence and $\fF'_1$ is a fibration.

With the above choice, $\cc_1\times^h_\cd\cc_2$ can be identified with the dg category whose objects are triples $(C_1,C_2,f)$, where $C_i\in\Ob(\cc_i)$, for $i=1,2$, and $f\colon\fF_1(C_1)\to\fF_2(C_2)$ is a homotopy equivalence. A morphism of degree $n$ from $(C_1,C_2,f)$ to $(C'_1,C'_2,f')$ in $\cc_1\times^h_\cd\cc_2$ is given by a triple $(a_1,a_2,h)$ with $a_i\in\Hom_{\cc_i}(C_i,C'_i)^n$, for $i=1,2$, and $h\in\Hom_\cd(\fF_1(C_1),\fF_2(C'_2)^{n-1}$. The differential is defined by
\[
d(a_1,a_2,h):=(d(a_1),d(a_2),d(h)+(-1)^n(f'\comp\fF_1(a_1)-\fF_2(a_2)\comp f))
\]
and the composition by
\[
(a'_1,a'_2,h')\comp(a_1,a_2,h):=(a'_1\comp a_1,a'_2\comp a_2,\fF_2(a'_2)\comp h+(-1)^nh'\comp \fF_1(a_1)).
\]

\begin{remark}\label{rmk:hompullback2}
By the universal property of the pullback, if $\cc$ is a dg category with dg functors $\fG_i\colon\cc\to\cc_i$, for $i=1,2$, making the diagram
\begin{equation}\label{eqn:pb1}
\xymatrix@C+10pt{
	\cc\ar[r]^-{\fG_2}\ar[d]_-{\fG_1}&\cc_2\ar[d]^-{\fF_2}\\
	\cc_1\ar[r]^-{\fF_1}&\cd
}
\end{equation}
commutative in $\dgCat$, then there is a unique dg functor $\fF\colon\cc\to\cc_1\times^h_\cd \cc_2$ making the diagram
\[
\xymatrix@C+10pt{
	\cc\ar[rr]^-{\fG_2}\ar[dd]_-{\fG_1}\ar[dr]^-{\fF}&&\cc_2\ar[d]^-{\fI_2}\\
	&\cc_1\times^h_\cd \cc_2\ar[r]\ar[d]&\cc'_2\ar[d]^-{\fF'_2}\\
	\cc_1\ar[r]^-{\fI_1}&\cc'_1\ar[r]^-{\fF'_1}&\cd
}
\]
commutative in $\dgCat$. Given the explicit description of the homotopy pullback discussed above, the dg functor $\fF$ is defined by
\[
\fF(C):=\left(\fG_1(C),\fG_2(C),\id_{\fF_1(\fG_1(C))}\right),
\]
on objects and
\[
\fF(a):=\left(\fG_1(a),\fG_2(a),0\right)
\]
on morphisms.
\end{remark}

\begin{remark}\label{rmk:genpb}
The homotopy pullback is a special instance of the concept of homotopy limit in a model category, for which exhaustive presentations are available---for example in \cite[Section 19.1]{Hir}.
What is important for us is that homotopy limits are well-defined once we invert weak equivalences. Concretely: given
a small category $\cn$ and
a functor $\fF\colon\cn\to\dgCat$,  the homotopy 
limit $\holim(\fF)$ is a well-defined object of $\Hqe$ up to canonical isomorphism.
Moreover, suppose we are given a diagram
\[
\xymatrix@C+20pt@R-22pt{
     & \ar@{=>}[dd]^-{\theta} & \\
\cn\ar@/^1.2pc/[rr]^-{\fF}\ar@/^-1.2pc/[rr]_-{\fG} & & \dgCat\ar[r]^-\pi& \Hqe \\
   & & 
}
\]
meaning that $\cn$ is a small category, $\fF,\fG\colon\cn\to \dgCat$ are functors, $\theta:\fF\to\fG$ is a natural transformation, and $\pi\colon\dgCat\to\Hqe$ is the natural functor. Then 
$\holim(\theta)$ delivers a well-defined 
morphism in $\Hqe$ (up to canonical isomorphism), and if $\pi\comp\theta\colon\pi\colon\fF\to\pi\colon\fG$ 
is an isomorphism then so is $\holim (\theta)$. Thus $\holim\colon\Hom\big(\cn,\dgCat\big)\to\Hqe$ takes morphisms in $\Hom\big(\cn,\dgCat\big)$
which induce isomorphisms in $\Hom\big(\cn,\Hqe\big)$ to
isomorphisms in $\Hqe$. 

What we will need in \autoref{sect:uniquepullbacks} is the following simple case. We start with a finite set of integers $N:=\{1,\dots,n\}$ and form the category $\cn$, whose objects are the subsets of $N$ and whose morphisms are the inclusions. We then consider a functor $\cn\to\dgCat$ and denote by $\cc_I$ the dg category corresponding to $I\subseteq N$. We can then form the homotopy limit
\[
\holim_{\emptyset\ne I\subseteq N}\cc_I. 
\]
The special case when $n=2$ is just a homotopy pullback, specifically the 
homotopy pullback of $\cc_{\{1\}}\to\cc_{\{1,2\}}\leftarrow\cc_{\{2\}}$.
In the discussion just preceding \autoref{rmk:hompullback2} we 
spelled out an explicit, functorial construction  for homotopy pullbacks,
allowing us to enhance the homotopy pullback to a functor
$\holim\colon\Hom\big(\{1\}\to\{1,2\}\leftarrow\{2\},\dgCat\big)\to\dgCat$. 
This will 
allow us iterate the construction. We will end up 
reducing the computation of the 
homotopy limit, in the case
where $n>2$, to iterated homotopy pullbacks.
\end{remark}

\subsection{Localizations in dg categories}\label{subsec:someextria}

Let us first recall a basic construction in the context of triangulated categories. Let $\ct$ be such a category and suppose further that $\ct$ is closed under arbitrary coproducts.
Let $\cs\subset\ct^c$ be a collection of compact objects closed under shifts. Then every object $T\in\ct$ sits in a distinguished triangle
\[
S\lto T\lto T_S
\]
with $S\in\Loc(\cs)$, the smallest full triangulated subcategory of $\ct$ containing $\cs$ and closed under arbitrary coproducts, and $T_S$ in $\cs^\perp:=\left\{T\in\ct\st\Hom_{\ct}(\cs,T)=0\right\}$.

The expert reader has certainly noticed that this observation follows from the existence of a semiorthogonal decomposition for $\ct$ with factors given by $\cs$ and $\cs^\perp\iso\ct/\cs$. But here we want to stress that the construction of S and $T_S$ is explicit.  Indeed, put $T_0=T$, and then inductively
construct distinguished triangles
\[
\coprod_{C\to T_i}C\lto T_i\lto T_{i+1}
\]
where the first map is the coproduct of all morphisms $C\to T_i$ with $C\in\Ob(\cs)$. Then we consider the map $T\lto \hocolim T_i$  and complete it to a distinguished triangle
\[
S\lto T\lto \hocolim T_i.
\]
A direct computation shows that $S\in\Ob(\Loc(\cs))$ and $T_S:=\hocolim T_i\in\Ob(\cs^\perp)$ (see, for example, \cite[Lemma 1.7]{N1}).

The following result will be used later and provides an enhancement of the above discussion to the setting of dg categories. Actually, the content of the proposition is more elaborate and deals with two distinct (and orthogonal) localizations.

\begin{prop}\label{prop:prelcolim}
Let $\cc$ be a dg category, and suppose that
$\cd_1$ and $\cd_2$ are full dg subcategories of $\cc$ such that
\begin{enumerate}
\item[{\rm (i)}] $\cd_1$ and $\cd_2$ are closed under shifts;
\item[{\rm (ii)}] $\Hom_{H^0(\cc)}(\cd_i,\cd_j)=0$ for $i\neq j\in\{1,2\}$.
\end{enumerate}
Then, for any $C\in\Ob(\hproj{\cc})$ there exists a
diagram in $Z^0(\hproj{\cc})$
\begin{equation}\label{eqn:diagrcolim}
\xymatrix{
& D_2\ar[d]\ar@{=}[r] & D_2\ar[d]   \\
D_1\ar[r]\ar@{=}[d] & C\ar[r]\ar[d] &C_{D_1}\ar[d] \\
D_1\ar[r] & C_{D_2}\ar[r] &C_{D_1,D_2}
}
\end{equation}
which is commutative in $H^0(\hproj{\cc})$ and such that
\begin{enumerate}
\item[{\rm (1)}] its rows and columns yield distinguished triangles in $H^0(\hproj{\cc})$;
\item[{\rm (2)}] $D_1\in\Ob(\hproj{\cd_1})$ and $D_2\in\Ob(\hproj{\cd_2})$;
\item[{\rm (3)}] the complexes $$\Hom_{\hproj{\cc}}\left(\hproj{\cd_i},C_{D_i}\right)\qquad\Hom_{\hproj{\cc}}\left(\hproj{\cd_i},C_{D_1,D_2}\right)$$ are acyclic, for $i=1,2$.
\end{enumerate}
\end{prop}

\begin{proof}
To obtain the objects $D_i$ and $C_{D_i}$ and the triangles in the second row and in the second column in \eqref{eqn:diagrcolim} we proceed by enhancing the construction at the beginning of this section to the context of dg categories. Since the situation is symmetric we denote by $\cd$ either of the two dg subcategories $\cd_1$ and $\cd_2$. Given $C\in\Ob(\hproj{\cc})$, we set $C_0=C$, 
and then inductively define $C_{i+1}$ to be the cone of the map
\[
\coprod_{P\to C_i}P\lto C_i
\]
where the coproduct is over all objects $P\in\Ob(\cd)$ and a set of morphisms in $\Hom_{Z^0(\hproj{\cd})}(P,C_i)$ representing all morphisms in $\Hom_{H^0(\hproj{\cd})}(P,C_i)$. As above, the natural map $C\to\hocolim C_i$ in $Z^0(\hproj{\cc})$ sits in the triangle
\[
D\lto C\lto\hocolim C_i
\]
which is distinguished in $H^0(\hproj{\cc})$.  The argument above now shows that $D$ is an object of $\hproj{\cd}$ and $C_D:=\hocolim C_i$ is such that $\Hom_{\hproj{\cc}}(Q,C_D)$ is acyclic, for all $Q\in\Ob(\hproj{\cd})$. 

Since $\hproj{\cc}$ is pretriangulated, the second row and the second column we just constructed can be completed to the diagram of \eqref{eqn:diagrcolim} where the rows and the columns are cofiber sequences,
in particular yield distinguished triangles in
$H^0(\hproj{\cc})$.

It remains to prove that the object $C_{D_1,D_2}$ that we constructed is such that the complex $\Hom_{\hproj{\cc}}\left(\hproj{\cd_i},C_{D_1,D_2}\right)$ is acyclic, for $i=1,2$. To this end, we apply the functor
$\Hom_{\hproj{\cc}}^{}\left(\cd_i,\farg\right)$ to the triangle
\[
D_j\lto C_{D_i}\lto C_{D_1,D_2}, 
\]
where $i\neq j\in\{1,2\}$.
The complex $\Hom_{\hproj{\cc}}\left(\cd_i,C_{D_i}\right)$ is clearly acyclic from the first part of the construction. By (ii) we know that $\Hom_\cc(\cd_i,\cd_j)$ is acyclic and thus $\Hom_{\hproj{\cc}}\left(\cd_i,\hproj{\cd_j}\right)$ is acyclic as well, because the objects in $\cd_i$ are in $\cc$ and thus they are compact in $H^0(\hproj{\cc})$. Since
$D_j$ belongs to $\hproj{\cd_j}$, we obtain that 
$\Hom_{\hproj{\cc}}\left(\cd_i,D_j\right)$ is acyclic. From this we deduce that $\Hom_{\hproj{\cc}}\left(\cd_i,C_{D_1,D_2}\right)$ is acyclic, and this implies that $\Hom_{\hproj{\cc}}\left(\hproj{\cd_i},C_{D_1,D_2}\right)$ is acyclic as well.
\end{proof}

\subsection{Dg enhancements and their uniqueness}\label{subsect:dgenuniq}

In this section we will be careful about the universe where the triangulated and dg categories live---mostly to show that such care is superfluous.

\begin{definition}\label{def:enhancement}
Let $\UU\in\VV$ be universes,
and consider a $\UU$-small triangulated category $\ct$.
A \emph{dg $\VV$-enhancement} (or simply a \emph{$\VV$-enhancement}) of $\ct$ is a pair
$(\cc,\fE)$, where $\cc$ is a $\VV$-small pretriangulated
dg category and $\fE\colon H^0(\cc)\to\ct$ is an exact
equivalence.
\end{definition}

By abuse of notation one says that $\cc$ is a $\VV$-enhancement of $\ct$ if $\fE$ is clear from the context.

\begin{ex}\label{ex:enhancementsabs}
If $\ca$ is a $\UU$-small additive category, then $\dgCa(\ca)$ is in a natural way a $\UU$-enhancement of $\Ka(\ca)$ (see \autoref{dgC}). If $\ca$ is abelian then, since $\Da(\ca)=\Ka(\ca)/\Acya(\ca)$, the discussion in \autoref{subsec:hprojhflat} allows us to conclude that the Drinfeld quotient 
\[
\Ddga(\ca):=\dgCa(\ca)/\dgAcya(\ca)
\]
 is a $\UU$-enhancement of $\Da(\ca)$.
Here $\dgAcya(\ca)$ denotes the full dg subcategory of $\dgCa(\ca)$ 
whose objects correspond to those of $\Acya(\ca)$.
\end{ex}

The core of this paper is to study when dg enhancements are unique in the following sense.

\begin{definition}\label{def:uniqueenh}
The $\UU$-small triangulated category $\ct$ \emph{has a unique $\VV$-enhancement} if, given two $\VV$-enhancements $(\cc_1,\fE_1)$ and $(\cc_2,\fE_2)$ of $\ct$, there is a $\VV$-small
pretriangulated dg category $\cc_3$ and quasi-equivalences $\fI_i\colon\cc_3\to\cc_i$, for $i=1,2$.
\end{definition}

Let us now prove the following result.

\begin{prop}\label{prop:changeuniv}
Let $\ct$ be a $\UU$-small triangulated category. If $\ct$ has a $\VV$--enhancement
for some universe $\VV$, with $\UU\in\VV$, then it also has a $\UU$--enhancement.
Furthermore: there exists a universe $\VV$, with $\UU\in\VV$ and 
such that $\ct$ has a unique $\VV$--enhancement, if and only
if $\ct$ has a unique $\UU$-enhancement. Moreover: if $\ct$ has a unique
$\UU$-enhancement then it has a unique $\VV$ enhancement for
all $\UU\in\VV$.
\end{prop}

\begin{proof}
Suppose $\ct$ has a $\VV$-enhancement $\cc$. In \autoref{lemmasmall} (see also \autoref{rmk:Usm})
we produced a quasi-equivalence $\sma{\fI}_\cc:\sma{\cc}\to\cc$,
with $\sma{\cc}$ a $\UU$-small pretriangulated dg category; 
this gives the existence of
a $\UU$-enhancement for $\ct$.

Now let $\VV$ be a universe with $\UU\in\VV$, and such that $\ct$ 
has a unique $\VV$-enhancement. Let $\cc_1$ and $\cc_2$ be two 
$\UU$-enhancements of $\ct$. By the uniqueness of $\VV$-enhancements
there exists a $\VV$-enhancement $\cc_3$ and 
quasi-equivalences $\fI_i\colon\cc_3\to\cc_i$, for $i=1,2$.
\autoref{lemmasmall} produces for us a quasi-equivalence
$\sma{\fI}_{\cc_3}:\sma{\cc}_3\to\cc_3$ with $\sma{\cc}_3$ a $\UU$-small
pretriangulated dg category, and the composites 
$\fI_i\comp\sma{\fI}_{\cc_3}\colon\sma{\cc}_3\to\cc_i$,
for $i=1,2$, show that the $\UU$-enhancement is unique.

Finally the ``moreover'' part. Suppose therefore that $\ct$ has a
unique $\UU$-enhancement and $\VV$ is some universe with $\UU\in\VV$. Let
$\cc_1$ and $\cc_2$ be two
pretriangulated dg $\VV$-categories enhancing $\ct$.
\autoref{lemmasmall} produces for us quasi-equivalences
$\sma{\fI}_i\colon\sma{\cc}_i\to\cc_i$ for $i=1,2$, with $\sma{\cc}_i$ both
$\UU$-small. By the uniqueness of $\UU$-enhancements there exists 
a $\UU$-small pretriangulated dg category $\cc_3$  and 
quasi-equivalences $\fI_i\colon\cc_3\to\sma{\cc}_i$ for $i=1,2$. But
now the composites $\sma{\fI}_i\comp\fI_i\colon\cc_3\to\cc_i$, for $i=1,2$,
show that the enhancements to $\VV$ are unique.
\end{proof}

Given this, in the rest of the paper it will rarely be necessary to specify in which universe we are working. For this reason we will freely talk about dg enhancements and their uniqueness without mentioning any universe.

\begin{remark}\label{rmk:uniqueoppeq}
(i) The property of having a unique dg enhancement is clearly invariant under exact equivalences. More precisely, if $\ct_1$ and $\ct_2$ are triangulated categories such that there is an exact equivalence $\ct_1\iso\ct_2$, then $\ct_1$ has a unique dg enhancement if and only if $\ct_2$ does.

(ii) Given a triangulated category $\ct$, the category $\ct{\opp}$ has a natural triangulated structure. Since there is a natural exact equivalence $H^0(\cd){\opp}\iso H^0(\cd{\opp})$, for any pretriangulated dg category $\cd$, it is clear that $\ct$ has a unique dg enhancement if and only if $\ct{\opp}$ does.
\end{remark}

An equivalent way to state \autoref{def:uniqueenh} is by saying that there is an isomorphism $f\in\Hom_\Hqe(\cc_1,\cc_2)$. With this in mind, we can state the following stronger definition.

\begin{definition}\label{def:semistrongenhancement}
A triangulated category $\ct$ has a \emph{semi-strongly unique enhancement} if given two enhancements $(\cc_1,\fE_1)$ and $(\cc_2,\fE_2)$ of $\ct$, there is an isomorphism $f\in\Hom_\Hqe(\cc_1,\cc_2)$ such that $\fE_1(C)\iso\fE_2(H^0(f))(C)$ in $\ct$, for every $C\in\Ob(\cc_1)$.
\end{definition}

\section{A special zigzag of dg functors}\label{sect:enB}

This slightly technical section provides a useful enhancement for the category $\Ba(\ca)$ of \autoref{subsec:cat}. This is done in \autoref{subsec:defB}.  \autoref{subsec:keyquasifunctor} deals with the important relation between enhancements of $\Va(\ca)$ and $\Ba(\ca)$. This is achieved with a complicated argument which involves a variant of the enhancement of $\Ba(\ca)$ constructed in \autoref{subsec:defB} (see \autoref{subsec:variantB}) and a discussion about homotopy limits in \autoref{subsec:hocolimdgcats} where we make explicit the construction in \autoref{subsec:modelpullbacks} in a very concrete context.

\subsection{An enhancement for $\Ba(\ca)$}\label{subsec:defB}

Remembering \autoref{ex:enhancementsabs}, it is clear that, for $?=b,+,-,\emptyset$, the full dg subcategory $\cv=\cv^?(\ca)$ of $\dgCa(\ca)$ whose objects are complexes with trivial differential is in a natural way an enhancement of $\Va(\ca)$.

Assume from now on that $\ca$ is abelian and fix an enhancement $(\cc,\fE)$ of $\Da(\ca)$. We are going to define a dg category $\cb^?_{\cc,\fE}(\ca)$ which will turn out to be an enhancement of $\Ba(\ca)$. Recall that the latter category is the full subcategory of $\Da(\ca)$ consisting of complexes with trivial differential. The notation for $\cb^?_{\cc,\fE}(\ca)$ alludes to the fact that its definition depends on the pair $(\cc,\fE)$. But in the sequel, when there is no risk of confusion, we will use the shorthand $\cb:=\cb^?_{\cc,\fE}(\ca)$.

An object $B=(B^-,B^+,B^i,\alpha^i,\beta^i)_{i\in\ZZ}$ of $\cb$ is given by objects $B^-$, $B^+$ and $B^i$ of $\cc$ together with morphisms $\sh[-i]{B^i}\mor{\alpha^i}B^?\mor{\beta^i}\sh[-i]{B^i}$ of $Z^0(\cc)$ (where $?=+$ if $i>0$ and $?=-$ if $i\le0$) such that the following conditions are satisfied:
\begin{enumerate}
\item[(B.1)] If $i,j$ are both $>0$ or both $\le0$, then $\beta^j\comp\alpha^i$ is $\id_{\sh[-i]{B^i}}$ for $i=j$ and is $0$ for $i\ne j$.
\item[(B.2)] $\fE(B^i)\in\Ob(\ca)$, for every $i\in\ZZ$.
\item[(B.3)] The morphisms $\alpha^i$ with $i\le0$ and $\beta^i$ with $i>0$ induce isomorphisms in $H^0(\cc)\iso\Da(\ca)$
\[
\coprod_{i\le0}\sh[-i]{B^i}\isomor B^-,\qquad B^+\isomor\prod_{i>0}\sh[-i]{B^i}.
\]
\end{enumerate}

Given objects $B_k=(B_k^-,B_k^+,B_k^i,\alpha_k^i,\beta_k^i)_{i\in\ZZ}$, for $k=1,2$, we define
\[
\Hom_{\cb}(B_1,B_2):=\Hom_{\cc}(B_1^-\oplus B_1^+,B_2^-\oplus B_2^+).
\]
It is straightforward to check that $\cb$ (with the obvious composition of morphisms) is a dg category and that the map defined on objects by $(B^-,B^+,B^i,\alpha^i,\beta^i)_{i\in\ZZ}\mapsto B^-\oplus B^+$ extends to a fully faithful dg functor $\bfun\colon\cb\to\cc$. 

One can consider the following variant which will be used in the proof of \autoref{thm:main2}. Let $\ca$ be an abelian category with a Serre subcategory $\ce\subseteq\ca$. Consider then the full triangulated subcategory $\Da_\ce(\ca)$ of $\Da(\ca)$ consisting of all complexes with cohomology in $\ce$. There is a natural exact functor
\[
\pi\colon\Da(\ce)\lto\Da_\ce(\ca)
\]
which is in general not an equivalence.

Given a dg enhancement $(\cc,\fE)$ of $\Da_\ce(\ca)$, we define a dg category $\wh\cb^?_{\cc,\fE}(\ca,\ce)$, for which we use the shorthand $\widehat\cb$, in a fashion which is very similar to $\cb$.

In particular, an object $B=(B^-,B^+,B^i,\alpha^i,\beta^i)_{i\in\ZZ}$ of $\widehat\cb$ is given by objects $B^-$, $B^+$ and $B^i$ of $\cc$ together with morphisms $\sh[-i]{B^i}\mor{\alpha^i}B^?\mor{\beta^i}\sh[-i]{B^i}$ of $Z^0(\cc)$ (where $?=+$ if $i>0$ and $?=-$ if $i\le0$) such that (B.1) is satisfied, while (B.2) and (B.3) are replaced respectively by the following:
\begin{enumerate}
\item[(B.2')] $\fE(B^i)\iso\pi(Q_i)$, for some $Q_i\in\Ob(\ce)$ and for every $i\in\ZZ$.
\item[(B.3')] The morphisms $\alpha^i$ with $i\le0$ and $\beta^i$ with $i>0$ induce isomorphisms in $\Da_\ce(\ca)$
\[
\pi\left(\coprod_{i\le0}\sh[-i]{Q^i}\right)\isomor\fE(B^-),\qquad \fE(B^+)\isomor\pi\left(\prod_{i>0}\sh[-i]{Q^i}\right).
\]
\end{enumerate}
As for the Hom-spaces, given objects $B_k=(B_k^-,B_k^+,B_k^i,\alpha_k^i,\beta_k^i)_{i\in\ZZ}$, for $k=1,2$, we keep the same definition and set
\[
\Hom_{\widehat\cb}(B_1,B_2):=\Hom_{\cc}(B_1^-\oplus B_1^+,B_2^-\oplus B_2^+).
\]

Recall the h-flat resolution $\hff{\fI}_\cc\colon\hff{\cc}\to\cc$
of Proposition~\ref{prop:genDr}. 

\begin{prop}\label{prop:invhf}
If $(\cc,\fE)$ is a dg enhancement of $\Da_\ce(\ca)$, then $\hff{\fI}_\cc$ induces a quasi-equivalence
\[
\hff{\wh\fI}_\cc\colon\widehat\cb^?_{\hff{\cc},\fE\comp H^0(\hff{\fI}_\cc)}(\ca,\ce)\lto\widehat\cb^?_{\cc,\fE}(\ca,\ce),
\]
which is also surjective on objects.
\end{prop}

\begin{proof}
The result follows from a direct check using the definitions and we only briefly outline it here. Since the dg functor $\hff{\fI}_\cc$ is the identity on objects, we define $\hff{\wh\fI}_\cc$ by sending an object $(B^-,B^+,B^i,\widetilde\alpha^i,\widetilde\beta^i)_{i\in\ZZ}$ of $\hff{\wh\cb}:=\widehat\cb^?_{\hff{\cc},\fE\comp H^0(\hff{\fI}_\cc)}(\ca,\ce)$ to $(B^-,B^+,B^i,\hff{\fI}(\widetilde\alpha^i),\hff{\fI}(\widetilde\beta^i))_{i\in\ZZ}$ in $\wh\cb:=\widehat\cb^?_{\cc,\fE}(\ca,\ce)$. The fact that morphisms in $\hff{\wh\cb}$ (resp.\ $\wh\cb$) are defined as in $\hff{\cc}$ (resp.\ $\cc$) clearly implies that $\hff{\wh\fI}_\cc$ extends to a dg functor if we define it on morphisms like $\hff{\fI}_\cc$. It is also obvious that $\hff{\wh\fI}_\cc$ is quasi-fully faithful because $\hff{\fI}_\cc$ is, and it remains to show that $\hff{\wh\fI}_\cc$ is surjective on objects. Given $(B^-,B^+,B^i,\alpha^i,\beta^i)_{i\in\ZZ}\in\Ob(\wh\cb)$, the closed degree-$0$ morphisms $\alpha^i$ and $\beta^i$ can be seen as morphisms in the dg category 
$\cc_1$---with the notation used in \autoref{subsec:hprojhflat} to define $\hff{\cc}$, we may view $\alpha^i$ and $\beta^i$
as belonging to the bases $D^1_\cc(B^i[-i],B^?)$ (respectively 
$D^1_\cc(B^?,B^i[-i])$) freely generating $\Hom_{\cc_1}(B^i[-i],B^?)$ (respectively 
$\Hom_{\cc_1}(B^?,B^i[-i])$) as modules over $\kk$---and
the relations $\beta^i\comp\alpha^i=\id$ and 
$\beta^j\comp\alpha^i=0$, for $i\neq j$, hold in $\cc_1$
by definition, and therefore also in $\hff{\cc}=\colim\cc_i$. 
Thus we obtain an object in $\hff{\wh\cb}$ which is mapped to $(B^-,B^+,B^i,\alpha^i,\beta^i)_{i\in\ZZ}$ by $\hff{\wh\fI}_\cc$.
\end{proof}

Back to the dg category $\cb$, in the following we will assume that:
\begin{equation}\label{eqn:Z0nullhomo}
Z^0(\cc) \text{ idempotent complete and closed under coproducts of null-homotopic objects.}
\end{equation}
Notice that, by \autoref{wic}, this condition can be always achieved, up to replacing $\cc$ with $\Pretr{\cc}=\Perf{\cc}$.

\begin{prop}\label{Bimage}
If $(\cc,\fE)$ is a dg enhancement of $\Da(\ca)$ such that $\cc$ satisfies \eqref{eqn:Z0nullhomo}, then the essential image of $\fE\comp H^0(\bfun)\colon H^0(\cb)\to\Da(\ca)$ is $\Ba(\ca)$.
\end{prop}

\begin{proof}
Given $B=(B^-,B^+,B^i,\alpha^i,\beta^i)_{i\in\ZZ}$ in $\cb$, it is clear from \autoref{cor:procoprod} and the definition of $\cb$ that $\fE(\bfun(B))\iso\bigoplus_{i\in\ZZ}\sh[-i]{A^i}\in\Ob(\Ba(\ca))$, where $A^i:=\fE(B^i)\in\Ob(\ca)$.

Conversely, given $A^*=\bigoplus_{i\in\ZZ}\sh[-i]{A^i}\in\Ob(\Ba(\ca))$, we can take $B^i\in\Ob(\cc)$ such that $\fE(B^i)\iso A^i$, for every $i\in\ZZ$. Let moreover $C^-,C^+\in\Ob(\cc)$ be such that $\fE(C^-)\iso\bigoplus_{i\le0}\sh[-i]{A^i}$ and $\fE(C^+)\iso\bigoplus_{i>0}\sh[-i]{A^i}$. Then there exist morphisms 
\[
\sh[-i]{B^i}\mor{\tilde{\alpha}^i}C^?\mor{\tilde{\beta}^i}\sh[-i]{B^i}
\]
of $Z^0(\cc)$ (where $?=+$ if $i>0$ and $?=-$ if $i\le0$) such that, if $i,j$ are both $>0$ or both $\le0$, then the image in $H^0(\cc)$ of $\tilde{\beta}^j\comp\tilde{\alpha}^i$ is $\id_{\sh[-i]{B^i}}$ for $i=j$ and is $0$ for $i\ne j$. Now we set
\[
B^-:=C^-\oplus\coprod_{i\le0}\cone{\id_{\sh[-i]{B^i}}},\qquad B^+:=C^+\oplus\coprod_{i>0}\cone{\id_{\sh[-i]{B^i}}}.
\]
For every $j\in\ZZ$ we also define $\alpha^j\colon\sh[-j]{B^j}\to B^?$ as the morphism induced by $\tilde{\alpha}^j\colon\sh[-j]{B^j}\to C^?$ and by the natural morphism $\sh[-j]{B^j}\to\cone{\id_{\sh[-j]{B^j}}}$. On the other hand, we define $\beta^j\colon B^?\to\sh[-j]{B^j}$ as the morphism induced by $\tilde{\beta}^j\colon C^?\to\sh[-j]{B^j}$ and by morphisms $\cone{\id_{\sh[-i]{B^i}}}\to\sh[-j]{B^j}$ in $Z^0(\cc)$ (whose existence is ensured by \autoref{coneprop}) with the property that the composition with the natural morphism $\sh[-i]{B^i}\to\cone{\id_{\sh[-i]{B^i}}}$ is $\id_{\sh[-i]{B^i}}-\tilde{\beta}^i\comp\tilde{\alpha}^i$ for $i=j$ and is $-\tilde{\beta}^j\comp\tilde{\alpha}^i$ for $i\ne j$. It is then clear that $B:=(B^-,B^+,B^i,\alpha^i,\beta^i)_{i\in\ZZ}$ is an object of $\cb$ which satisfies $\fE(\bfun(B))\iso A^*$.
\end{proof}

\begin{remark}\label{CperfB}
By \autoref{Bimage} the dg category $\cb$ is an enhancement of $\Ba(\ca)$. Moreover, in view of \autoref{cor:genD}, and taking into account that $H^0(\cc)\iso\Da(\ca)$ is idempotent complete (see \autoref{wic}), we have $\Perf{\cb}\iso\cc$ in $\Hqe$ (using \cite[Proposition 1.16]{LO}). Observe that, similarly, $\Perf{\cv}\iso\dgCa(\ca)$ in $\Hqe$. Hence $\Perf{\cb}$ (resp.\ $\Perf{\cv}$) is an enhancement of $\Da(\ca)$ (resp.\ $\Ka(\ca)$).
\end{remark}

\subsection{A technical interlude}\label{subsec:variantB}

We first describe a variant of $\cb$. We start by introducing new $\Hom$ spaces between objects of $\cb$, depending on an integer $n>0$. 

Given $B=(B^-,B^+,B^i,\alpha^i,\beta^i)_{i\in\ZZ}$ in $\cb$, the two compositions of morphisms (induced by the $\alpha^i$ and the $\beta^i$)
\[
\bigoplus_{i=-n}^0\sh[-i]{B^i}\lto B^-\lto\bigoplus_{i=-n}^0\sh[-i]{B^i}\qquad\qquad
\bigoplus_{i=1}^n\sh[-i]{B^i}\lto B^+\lto\bigoplus_{i=1}^n\sh[-i]{B^i}
\]
are obviously the identities. As $Z^0(\cc)$ is idempotent complete, this implies that there exist objects $\Bsum^{-n-1}$ and $\Bsum^{n+1}$ of $\cc$ such that (setting also $\Bsum^i:=B^i$ for $-n\le i\le n$) 
\[
B^-\iso\bigoplus_{i=-n-1}^0\sh[-i]{\Bsum^i}\qquad\qquad
B^+\iso\bigoplus_{i=1}^{n+1}\sh[-i]{\Bsum^i}
\]
in $Z^0(\cc)$. Notice that, if $n<m$, then
\[
\sh[n+1]{\Bsum^{-n-1}}\iso\bigoplus_{i=-m-1}^{-n-1}\sh[-i]{\Bsum[m]^i}\qquad\qquad
\sh[-n-1]{\Bsum^{n+1}}\iso\bigoplus_{i=n+1}^{m+1}\sh[-i]{\Bsum[m]^i}
\]
in $Z^0(\cc)$. Now, given two objects $B_1$ and $B_2$ of $\cb$, there is an isomorphism of complexes
\[
\Hom_{\cb}(B_1,B_2)\iso\prod_{i=-n-1}^{n+1}\prod_{j=-n-1}^{n+1}\Hom_{\cc}(\sh[-i]{\Bsum_1^i},\sh[-j]{\Bsum_2^j}),
\]
and we define
\[
\Hom_n(B_1,B_2):=\prod_{i=-n-1}^{n+1}\prod_{j=-n-1}^{n+1}\Hom_n^{i,j}(B_1,B_2),
\]
where (for $-n-1\le i,j\le n+1$)
\[
\Hom_n^{i,j}(B_1,B_2):=
\begin{cases}
\Hom_{\cc}(\sh[-i]{\Bsum_1^i},\sh[-j]{\Bsum_2^j}) & \text{if $i=-n-1$ or $j=n+1$} \\
\Hom_{\cc}(\sh[-i]{\Bsum_1^i},\sh[-j]{\Bsum_2^j})^{\le j-i} & \text{if $i\ge-n$ and $j\le n$.}
\end{cases}
\]
Obviously we can identify $\Hom_n(B_1,B_2)$ with a subcomplex of $\Hom_{\cb}(B_1,B_2)$.

\begin{remark}\label{nocat}
One might hope to obtain a dg category $\cb_n$ with the same objects as $\cb$ and such that $\Hom_{\cb_n}(B_1,B_2):=\Hom_n(B_1,B_2)$, for every $B_1,B_2\in\Ob(\cb)$. Unfortunately, given also $B_3\in\Ob(\cb)$, the composition map (which is a morphism of complexes)
\begin{equation}\label{eq:comp}
\Hom_{\cb}(B_1,B_2)\otimes_{\kk}\Hom_{\cb}(B_2,B_3)\lto\Hom_{\cb}(B_1,B_3)
\end{equation}
does not send $\Hom_n(B_1,B_2)\otimes_{\kk}\Hom_n(B_2,B_3)$ to $\Hom_n(B_1,B_3)$, in general. More precisely, $\Hom_n^{h,i}(B_1,B_2)\otimes_{\kk}\Hom_n^{i,j}(B_2,B_3)$ need not be sent to $\Hom_n^{h,j}(B_1,B_3)$ if $-n\le h\le n+1$, $-n-1\le j\le n$ and $i=-n-1$ or $i=n+1$.
\end{remark}

\begin{remark}\label{msubn}
If $n<m$, we can regard $\Hom_m(B_1,B_2)$ as a subcomplex of $\Hom_n(B_1,B_2)$. Indeed, it is clear that, for $\-m-1\le i,j\le m+1$, $\Hom_{\cc}(\sh[-i]{\Bsum[m]_1^i},\sh[-j]{\Bsum[m]_2^j})$ can be viewed as a subcomplex of $\Hom_{\cc}(\sh[-i']{\Bsum_1^{i'}},\sh[-j']{\Bsum_2^{j'}})$, where we define
\[
i':=
\begin{cases}
-n-1 & \text{if $-m-1\le i\le-n-1$} \\
i & \text{if $-n-1\le i\le n+1$} \\
n+1 & \text{if $n+1\le i\le m+1$}
\end{cases}
\]
and similarly for $j'$. Since $j-i\le j'-i'$ when $i'\ge-n$ (which implies $i'\le i$) and $j'\le n$ (which implies $j'\ge j$), it follows that in any case $\Hom_m^{i,j}(B_1,B_2)$ can be viewed as a subcomplex of $\Hom_n^{i',j'}(B_1,B_2)$.
\end{remark}

By the above there is a natural morphism $\holim\Hom_n(B_1,B_2)\to\Hom_{\cb}(B_1,B_2)$ in $\D(\Mod{\kk})$. Composing it with the morphism of complexes
\[
\Hom_{\cb}(B_1,B_2)=\Hom_{\cc}(B_1^-\oplus B_1^+,B_2^-\oplus B_2^+)\lto\prod_{i,j\in\ZZ}\Hom_{\cc}(\sh[-i]{B_1^i},\sh[-j]{B_2^j})
\]
(induced by precomposition with the $\beta^j_2$ and postcomposition with the $\alpha^i_1$), we obtain a morphism in $\D(\Mod{\kk})$
\begin{equation}\label{eq:hoHom}
\holim\Hom_n(B_1,B_2)\lto\prod_{i,j\in\ZZ}\Hom_{\cc}(\sh[-i]{B_1^i},\sh[-j]{B_2^j}).
\end{equation}

\begin{lem}\label{holim}
For every $B_1,B_2\in\Ob(\cb)$ \eqref{eq:hoHom} factors (uniquely) through a morphism
\[
\holim\Hom_n(B_1,B_2)\lto\prod_{i,j\in\ZZ}\Hom_{\cc}(\sh[-i]{B^i_1},\sh[-j]{B^j_2})^{\le j-i},
\]
which is an isomorphism in $\D(\Mod{\kk})$. Moreover, this isomorphism is compatible with truncations, meaning that for every $k\in\ZZ$ the induced morphism
\[
\holim\Hom_n(B_1,B_2)^{\le k}\lto\prod_{i,j\in\ZZ}\Hom_{\cc}(\sh[-i]{B^i_1},\sh[-j]{B^j_2})^{\le\min(k,j-i)}
\]
is also an isomorphism in $\D(\Mod{\kk})$.
\end{lem}

\begin{proof}
Let us fix an integer $l$. Notice first that, since $H^l(\Hom_{\cc}(C_1,C_2))\iso\Hom_{H^0(\cc)}(C_1,\sh[l]{C_2})$ for every $C_1,C_2\in\Ob(\cc)$, we have
\[
H^l(\Hom_{\cc}(C_1,C_2)^{\le k})\iso
\begin{cases}
\Hom_{H^0(\cc)}(C_1,\sh[l]{C_2}) & \text{if $l\le k$} \\
0 & \text{if $l>k$}
\end{cases}
\]
for every $k\in\ZZ$. In particular, we obtain
\[
H^l(\Hom_{\cc}(\sh[-i]{B^i_1},\sh[-j]{B^j_2})^{\le j-i})\iso
\begin{cases}
\Hom_{H^0(\cc)}(B_1^i,\sh[l+i-j]{B_2^j}) & \text{if $l\le j-i$} \\
0 & \text{if $l>j-i$}
\end{cases}
\]
for every $i,j\in\ZZ$. Taking into account that $\Hom_{H^0(\cc)}(B^i_1,\sh[k]{B^j_2})=0$ for $k<0$ (because $\fE(B_1^i),\fE(B_2^j)\in\Ob(\ca)$), this implies that
\[
H^l(\Hom_{\cc}(\sh[-i]{B_1^i},\sh[-j]{B_2^j})^{\le j-i})\iso
\begin{cases}
\Hom_{H^0(\cc)}(B_1^i,B_2^j) & \text{if $l=j-i$} \\
0 & \text{if $l\ne j-i$.}
\end{cases}
\]
As $H^l\colon\D(\Mod{\kk})\to\Mod{\kk}$ commutes with products, it follows that
\begin{equation}\label{Hl}
H^l\left(\prod_{i,j\in\ZZ}\Hom_{\cc}(\sh[-i]{B_1^i},\sh[-j]{B_2^j})^{\le j-i}\right)\iso\prod_{i\in\ZZ}\Hom_{H^0(\cc)}(B_1^i,B_2^{i+l}).
\end{equation}
On the other hand, given integers $n,i,j$ with $n>0$ and $-n-1\le i,j\le n+1$, we have
\[
H^l(\Hom_n^{i,j}(B_1,B_2))\iso
\begin{cases}
\Hom_{H^0(\cc)}(\sh[-i]{\Bsum_1^i},\sh[l-j]{\Bsum_2^j}) & \text{if $l\le j-i$ or $i=-n-1$ or $j=n+1$} \\
0 & \text{if $l>j-i$, $i\ge-n$ and $j\le n$.}
\end{cases}
\]
We shall henceforth assume $n>-l/2-1$. Then
\[
H^l(\Hom_n^{n+1,-n-1}(B_1,B_2))=0
\]
(since $i=n+1\ge-n$, $j=-n-1\le n$ and $l>j-i=-2n-2$). Therefore
\begin{multline}\label{Hln}
H^l(\Hom_n(B_1,B_2))\iso\prod_{i=-n-1}^{n+1}\prod_{j=-n-1}^{n+1}H^l(\Hom_n^{i,j}(B_1,B_2)) \\
\iso\prod_{i=-n-1}^n\prod_{j=-n}^{n+1}H^l(\Hom_n^{i,j}(B_1,B_2))\oplus\prod_{j=-n}^{n+1}H^l(\Hom_n^{n+1,j}(B_1,B_2))\oplus\prod_{i=-n-1}^nH^l(\Hom_n^{i,-n-1}(B_1,B_2)).
\end{multline}
Now, observe that $\sh[n+1]{\Bsum_1^{-n-1}}\iso\coprod_{i<n}\sh[-i]{B_1^i}$ and $\sh[-n-1]{\Bsum_2^{n+1}}\iso\prod_{i>n}\sh[-i]{B_2^i}$ in $H^0(\cc)$. From this we deduce that (for $-n-1\le j\le n+1$)
\begin{multline*}
H^l(\Hom_n^{-n-1,j}(B_1,B_2))\iso\Hom_{H^0(\cc)}(\sh[n+1]{\Bsum_1^{-n-1}},\sh[l-j]{\Bsum_2^j}) \\
\iso\Hom_{H^0(\cc)}\left(\coprod_{i<-n}\sh[-i]{B_1^i},\sh[l-j]{\Bsum_2^j}\right)\iso\prod_{i<-n}\Hom_{H^0(\cc)}(\sh[-i]{B_1^i},\sh[l-j]{\Bsum_2^j})
\end{multline*}
and similarly (for $-n-1\le i\le n+1$)
\begin{multline*}
H^l(\Hom_n^{i,n+1}(B_1,B_2))\iso\Hom_{H^0(\cc)}(\sh[-i]{\Bsum_1^i},\sh[l-n-1]{\Bsum_2^{n+1}}) \\
\iso\Hom_{H^0(\cc)}\left(\sh[-i]{\Bsum_1^i},\prod_{j>n}\sh[l-j]{B_2^j}\right)\iso\prod_{j>n}\Hom_{H^0(\cc)}(\sh[-i]{\Bsum_1^i},\sh[l-j]{B_2^j}).
\end{multline*}
Thus the first summand in \eqref{Hln} can be written as
\[
\prod_{i=-n-1}^n\prod_{j=-n}^{n+1}H^l(\Hom_n^{i,j}(B_1,B_2))\iso\prod_{i\le n}\prod_{j\ge-n}\Hom_{n,l}(B_1^i,B_2^j),
\]
where (for $i\le n$ and $j\ge-n$)
\[
\Hom_{n,l}(B_1^i,B_2^j):=\begin{cases}
\Hom_{H^0(\cc)}(B_1^i,\sh[l+i-j]{B_2^j}) & \text{if $l\le j-i$ or $i<-n$ or $j>n$} \\
0 & \text{if $l>j-i$, $i\ge-n$ and $j\le n$.}
\end{cases}
\]
Using again the fact that $\Hom_{H^0(\cc)}(B_1^i,\sh[l+i-j]{B_2^j})=0$ if $l<j-i$, we see that, for fixed $i,j\in\ZZ$,
\[
\Hom_{n,l}(B_1^i,B_2^j)\iso
\begin{cases}
\Hom_{H^0(\cc)}(B_1^i,B_2^j) & \text{if $l=j-i$} \\
0 & \text{if $l\ne j-i$.}
\end{cases}
\]
when $n\ge\abs{i},\abs{j}$.

Reasoning as before, the last two summands in \eqref{Hln} can be written as
\begin{gather*}
\prod_{j=-n}^{n+1}H^l(\Hom_n^{n+1,j}(B_1,B_2))\iso\prod_{j\ge-n}\Hom_{n,l}(B_1,B_2^j), \\
\prod_{i=-n-1}^nH^l(\Hom_n^{i,-n-1}(B_1,B_2))\iso\prod_{i\le n}\Hom_{n,l}(B_1^i,B_2)
\end{gather*}
where (for $j\ge-n$ and $i\le n$)
\begin{gather*}
\Hom_{n,l}(B_1,B_2^j):=\begin{cases}
\Hom_{H^0(\cc)}(\Bsum_1^{n+1},\sh[l+n+1-j]{B_2^j}) & \text{if $l<j-n$ or $j>n$} \\
0 & \text{if $l\ge j-n$ and $j\le n$,}
\end{cases} \\
\Hom_{n,l}(B_1^i,B_2):=\begin{cases}
\Hom_{H^0(\cc)}(B_1^i,\sh[l+i+n+1]{\Bsum_2^{-n-1}}) & \text{if $l<-n-i$ or $i<-n$} \\
0 & \text{if $l\ge-n-i$ and $i\ge-n$.}
\end{cases}
\end{gather*}
Then, given $j\in\ZZ$ (respectively, $i\in\ZZ$), $\Hom_{n,l}(B_1,B_2^j)=0$ (respectively, $\Hom_{n,l}(B_1^i,B_2)=0$) when $n\ge\abs{j},j-l$ (respectively $n\ge\abs{i},-i-l$).

Summing up, we have shown that $H^l(\Hom_n(B_1,B_2))$ is a product of terms, each of which stabilizes as $n$ goes to infinity. Moreover, the only surviving terms in the limit are $\Hom_{H^0(\cc)}(B_1^i,B_2^{i+l})$, which are precisely those appearing in the right-hand side of \eqref{Hl}. From this it follows that
\[
H^l(\holim\Hom_n(B_1,B_2))\iso\lim_nH^l(\Hom_n(B_1,B_2))
\]
and that the right-hand side can be identified with the left-hand side of \eqref{Hl}. This clearly proves the first statement. As for the second, it is enough to note that, by the same argument,
\[
\holim\Hom_n(B_1,B_2)^{\le k}\iso(\holim\Hom_n(B_1,B_2))^{\le k}.
\]
This concludes the proof.
\end{proof}

Now we are going to see how the problem with compositions mentioned in \autoref{nocat} can be overcome.

\begin{lem}\label{compn}
Given $B_1,B_2,B_3\in\Ob(\cB)$ and $n,k,l\in\ZZ$ with $n\ge k,l,1$, the composition map \eqref{eq:comp} restricts to a morphism of complexes
\[
\m_n^{k,l}\colon\Hom_{2n}(B_1,B_2)^{\le k}\otimes_{\kk}\Hom_{2n}(B_2,B_3)^{\le l}\lto\Hom_n(B_1,B_3)^{\le k+l}.
\]
\end{lem}

\begin{proof}
Given $f\in\Hom_{2n}(B_1,B_2)^{\le k}\subseteq\Hom_{\cb}(B_1,B_2)$ and $g\in\Hom_{2n}(B_2,B_3)^{\le l}\subseteq\Hom_{\cb}(B_2,B_3)$, we have to prove that $g\comp f\in\Hom_{\cb}(B_1,B_3)$ actually belongs to $\Hom_n(B_1,B_3)^{\le k+l}$. As obviously $g\comp f\in\Hom_{\cb}(B_1,B_3)^{\le k+l}$, it is enough to show that $g\comp f\in\Hom_n(B_1,B_3)$. We can clearly assume that there exist $-2n-1\le h,i,j\le2n+1$ such that
\begin{gather*}
f\in\Hom_{2n}^{h,i}(B_1,B_2)^{\le k}\subseteq\Hom_{\cc}(\sh[-h]{\Bsum[2n]_1^h},\sh[-i]{\Bsum[2n]_2^i}), \\
g\in\Hom_{2n}^{i,j}(B_2,B_3)^{\le l}\subseteq\Hom_{\cc}(\sh[-i]{\Bsum[2n]_2^i},\sh[-j]{\Bsum[2n]_2^j}).
\end{gather*}
Using the notation of \autoref{msubn} (with $m=2n$), we know that $g\comp f\in\Hom_{\cc}(\sh[-h']{\Bsum_1^{h'}},\sh[-j']{\Bsum_3^{j'}})$. Thus, in order to conclude that $g\comp f\in\Hom_n^{h',j'}(B_1,B_3)$, it remains to check that $g\comp f\in\Hom_{\cb}(B_1,B_3)^{\le j'-h'}$ if $h'\ge-n$ and $j'\le n$. Now, this is certainly true if $\abs{i}\le2n$, because in that case $f\in\Hom_{\cb}(B_1,B_2)^{\le i-h}$, $g\in\Hom_{\cb}(B_2,B_3)^{\le j-i}$ and $i-h+j-i=j-h\le j'-h'$. On the other hand, if $i=-2n-1$, then $f\in\Hom_{\cb}(B_1,B_2)^{\le-2n-1-h}$, $g\in\Hom_{\cb}(B_2,B_3)^{\le l}$ and $-2n-1-h+l\le-2n-1-h'+n\le j'-h'$. Similarly, if $i=2n+1$, then $f\in\Hom_{\cb}(B_1,B_2)^{\le k}$, $g\in\Hom_{\cb}(B_2,B_3)^{\le j-2n-1}$ and $k+j-2n-1\le n+j'-2n-1\le j'-h'$.
\end{proof}

For every $n>0$ the natural projection
\[
\begin{split}
\Hom_n(B_1,B_2)= & \prod_{i=-n-1}^{n+1}\prod_{j=-n-1}^{n+1}\Hom_n^{i,j}(B_1,B_2) \\
\lto & \prod_{i=-n}^n\prod_{j=-n}^n\Hom_n^{i,j}(B_1,B_2)=\prod_{i=-n}^n\prod_{j=-n}^n\Hom_{\cc}(\sh[-i]{B_1^i},\sh[-j]{B_2^j})^{\le j-i}
\end{split}
\]
induces, passing to truncations, a morphism of complexes
\[
\Hom_n(B_1,B_2)^{\le k}\lto\prod_{i=-n}^n\prod_{j=-n}^n\Hom_{\cc}(\sh[-i]{B_1^i},\sh[-j]{B_2^j})^{\le\min(k,j-i)},
\]
for every $k\in\ZZ$. Composing it with the product of the natural maps
\begin{equation}\label{eq:tr}
\Hom_{\cc}(\sh[-i]{B_1^i},\sh[-j]{B_2^j})^{\le\min(k,j-i)}\lto\bigl(\Hom_{\cc}(\sh[-i]{B_1^i},\sh[-j]{B_2^j})^{\le\min(k,j-i)}\bigr)^{\ge j-i},
\end{equation}
we obtain a morphism of complexes
\[
\p_n^k\colon\Hom_n(B_1,B_2)^{\le k}\lto\Homr_n^k(B_1,B_2):=\prod_{i=-n}^n\prod_{j=-n}^n\bigl(\Hom_{\cc}(\sh[-i]{B_1^i},\sh[-j]{B_2^j})^{\le\min(k,j-i)}\bigr)^{\ge j-i}.
\]
Observe that, for every $k,l\in\ZZ$, there are also morphisms of complexes
\[
\mr_n^{k,l}\colon\Homr_{2n}^k(B_1,B_2)\otimes_{\kk}\Homr_{2n}^l(B_2,B_3)\lto\Homr_{2n}^{k+l}(B_1,B_3)\lto\Homr_n^{k+l}(B_1,B_3),
\]
where the first map is induced by composition in $\cc$, and the second one is the natural projection.

\begin{lem}\label{compat}
The maps $\p_n^k$ are compatible with compositions, meaning that, for every $n,k,l\in\ZZ$ with $n\ge k,l,1$,
\[
\xymatrix{
\Hom_{2n}(B_1,B_2)^{\le k}\otimes_{\kk}\Hom_{2n}(B_2,B_3)^{\le l} \ar[rr]^-{\m_n^{k,l}} \ar[d]_{\p_{2n}^k\otimes p_{2n}^l} & & \Hom_n(B_1,B_3)^{\le k+l} \ar[d]^{\p_n^{k+l}} \\
\Homr_{2n}^k(B_1,B_2)\otimes_{\kk}\Homr_{2n}^l(B_2,B_3) \ar[rr]_-{\mr_n^{k,l}} & & \Homr_n^{k+l}(B_1,B_3)
}
\]
is a commutative diagram of complexes.
\end{lem}

\begin{proof}
As in the proof of \autoref{compn}, let $f\in\Hom_{2n}^{h,i}(B_1,B_2)^{\le k}$ and $g\in\Hom_{2n}^{i,j}(B_2,B_3)^{\le l}$ for some $-2n-1\le h,i,j\le2n+1$: we need to prove that
\begin{equation}\label{eq:comm}
\p_n^{k+l}(\m_n^{k,l}(f\otimes g))=\mr_n^{k,l}(\p_{2n}^k(f)\otimes\p_{2n}^l(g)).
\end{equation}
We can restrict to the case $-n\le h,j\le n$, since otherwise both sides of \eqref{eq:comm} are evidently $0$. We claim that the same is true if $\abs{i}=2n+1$. In fact, if $i=-2n-1$, then $f\in\Hom_{\cb}(B_1,B_2)^{\le-2n-1-h}$, $g\in\Hom_{\cb}(B_2,B_3)^{\le l}$ and $-2n-1-h+l\le-2n-1-h+n<j-h$. Similarly, if $i=2n+1$, then $f\in\Hom_{\cb}(B_1,B_2)^{\le k}$, $g\in\Hom_{\cb}(B_2,B_3)^{\le j-2n-1}$ and $k+j-2n-1\le n+j-2n-1<j-h$. Finally, if $\abs{i}\le2n$, we can assume that $f$ and $g$ are homogeneous, say of degrees $d$ and $e$. We must have $d\le i-h$ and $e\le j-i$ (whence $d+e\le j-h$), and again \eqref{eq:comm} becomes $0=0$ unless $d+e=j-h$. Clearly $d+e=j-h$ implies $d=i-h$ and $e=j-i$, in which case it is straightforward to see that \eqref{eq:comm} is satisfied.
\end{proof}

\subsection{Homotopy limits of sequences}
\label{subsec:hocolimdgcats}

As the name suggests, homotopy limits are unique only up to homotopy. And 
there are multiple ways to make enhanced versions of them---we already met
this  in \autoref{subsec:modelpullbacks}, where the special
case of homotopy pullbacks was discussed in some detail. 
We remind the reader: in \autoref{subsec:modelpullbacks} the approach
was to impose a model structure on some ambient category, and with respect to 
this model structure do some fibrant replacement.
In this section
we want to lay the groundwork for the way we will treat homotopy limits
of countable sequences, and the method will be different. It will be based on 
a (dual) version of Milnor's mapping telescope.

We start in somewhat greater generality, the sequences will come later. 
Let $A_1,A_2,A_3$ be cochain complexes
of $\kk$-modules, and let $\mu:A_1\otimes A_2\to A_3$ be a cochain map, which we should think of as the composition.
Suppose further that, for $i\in\{1,2,3\}$, we are given cochain maps 
$\phi_i:A_i\to A_i$
such that the square below commutes
\[\xymatrix{
A_1\otimes A_2\ar[r]^-\mu \ar[d]_{\phi_1\otimes\phi_2} & A_3\ar[d]^{\phi_3} \\
A_1\otimes A_2\ar[r]^-\mu & A_3.
}\]
Now for $i\in\{1,2,3\}$ we define 
\[
\widetilde A_i:=\sh[-1]{\cone{A_i\mor{\id-\phi_i}A_i}}.
\]
And the composition map
$\widetilde \mu:\widetilde A_1\otimes\widetilde  A_2\to\widetilde  A_3$
is set to be the composite
\[\xymatrix{
\sh[-1]{\cone{A_1\mor{\id-\phi_1}A_1}}
\otimes\sh[-1]{\cone{A_2\mor{\id-\phi_2}A_2}}\ar[d]^-{\text{truncation}} \\
\sh[-1]{\cone{A_1\otimes A_2\lmor{\Psi}(A_1\otimes A_2)\oplus(A_1\otimes A_2)}}\ar[d]^{\Theta}\\
\sh[-1]{\cone{A_3\mor{\id-\phi_3}A_3}}.
}\]
The truncation is the obvious map; the tensor product of two mapping
cones is the total complex of a complex of three terms, and we truncate
the term on the right. The map $\Psi$ is also the obvious, meaning what is left 
from the tensor product of two mapping cones after truncation: we take the 
morphisms $(\id-\phi_1)\otimes\id\colon A_1\otimes A_2\to A_1\otimes A_2$ and 
$\id\otimes(\id-\phi_2)\colon A_1\otimes A_2\to A_1\otimes A_2$ and combine
them to form a single map $\Psi\colon A_1\otimes A_2\to (A_1\otimes A_2)\oplus
(A_1\otimes A_2)$. And finally the map $\Theta$ is obtained as the map
deduced from the commutative square below by taking the mapping cones
of the horizontal maps
\[
\xymatrix{
A_1\otimes A_2\ar[r]^-{\Psi}\ar[d]_\mu & 
(A_1\otimes A_2)\oplus (A_1\otimes A_2)\ar[d]^-{(\mu,\mu\comp(\phi_1\otimes\id))}
\\
A_3\ar[r]^-{\id-\phi_3} &A_3.
}
\]

\begin{remark}\label{rmk:holimdg}
Suppose now that $A_1$, $A_2$ and $A_3$ are inverse sequences of 
cochain complexes of $\kk$-modules. That is: for any integer $n>0$ and for $i\in\{1,2,3\}$ we are given a cochain complex $A_{i,n}$, these come with multiplication maps
$\mu_n:A_{1,n}\otimes A_{2,n}\to A_{3,n}$ and with sequence maps
$\phi_{i,n}\colon A_{i,n+1}\to A_{i,n}$, and for each $n$ the square
below commutes
\[
\xymatrix@C+20pt{
A_{1,n+1}\otimes A_{2,n+1}\ar[r]^-{\mu_{n+1}} \ar[d]_{\phi_{1,n}\otimes\phi_{2,n}} 
& A_{3,n+1}\ar[d]^{\phi_{3,n}} \\
A_{1,n}\otimes A_{2,n}\ar[r]^-{\mu_n} & A_{3,n}
}
\]
For $i\in\{1,2,3\}$ define $\widehat A_i:=\prod_{n>0}A_{i,n}$.
The multiplication map 
$\widehat\mu\colon\widehat A_1\otimes \widehat A_2\to\widehat A_3$ is the composite
\[
\xymatrix@C+20pt{
\displaystyle\left(\prod_{n>0} A_{1,n}\right)\otimes\left(\prod_{n>0}A_{2,n}\right)
\ar[r] &
\displaystyle\prod_{n>0}\left(A_{1,n}\otimes A_{2,n}\right) \ar[r]^-{\prod_{n>0}\mu_n} &
\displaystyle\prod_{n>0}A_{3,n}.
}
\]
If we let $\phi_i\colon \widehat A_i\to\widehat A_i$ be the composite
\[
\xymatrix@C+20pt{
\displaystyle\prod_{n>0}A_{i,n}\ar[r]^-{\text{projection}} &
\displaystyle\prod_{n>0}A_{i,n+1} \ar[r]^-{\prod_{n>0}\phi_{1,n}} &
\displaystyle\prod_{n>0}A_{i,n},
}
\]
then the square
\[
\xymatrix{
\widehat A_1\otimes \widehat A_2\ar[r]^-{\widehat \mu} \ar[d]_{\phi_1\otimes\phi_2} & \widehat A_3\ar[d]^{\phi_3} \\
\widehat A_1\otimes \widehat A_2\ar[r]^-{\widehat \mu} & \widehat A_3
}
\]
commutes. Then, setting
\[
\widetilde A_i:=\sh[-1]{\cone{\widehat A_i\mor{\id-\phi_i}\widehat A_i}},
\]
the discussion preceding the remark showed us how to construct
the composition 
$\widetilde \mu:\widetilde A_1\otimes\widetilde  A_2\to\widetilde  A_3$.
\end{remark}

Using this, we can give the following.

\begin{definition}\label{def:differentialring}
The \emph{homotopy limit of a sequence $A=\{A_n\}$ of 
complexes of $\kk$-modules} is the complex
\[
\holim A_n:=\sh[-1]{\cone{\prod_{n>0}A_n\mor{\id-\phi}\prod_{n>0}A_n}}.
\]
\end{definition}

\autoref{rmk:holimdg} also showed us how, given three inverse 
sequences of complexes of $\kk$-modules $A_1$, $A_2$ and $A_3$ and
compatible multiplications $\mu_n:A_{1,n}\otimes A_{2,n}\to A_{3,n}$, we can assemble them to a multiplication map 
\[
\holim\mu_n\colon\holim A_{1,n}\otimes\holim A_{2,n}\longrightarrow
\holim A_{3,n}\ .
\]

\subsection{Relating $\cv$ and $\cb$}
\label{subsec:keyquasifunctor}

We have been assembling a sequence of technical lemmas, and the time has come 
to use them. In this Section we will prove

\begin{prop}\label{qfunVB}
There exists a morphism $\qf\colon\cv\to\cb$ in $\Hqe$ such that the exact functor $H^0(\qf)\colon H^0(\cv)\iso\Va(\ca)\to H^0(\cb)\iso\Ba(\ca)$ (see \autoref{CperfB}) can be identified with the natural functor $\Va(\ca)\to\Ba(\ca)$.
\end{prop}

The proof will occupy the remainder of this Section. We will begin 
with the dg category 
$\cb$, and gradually produce a zigzag of dg functors that compose (in the category
$\Hqe$) to our desired map $\qf\colon\cv\to\cb$.

\medskip

\noindent\emph{Step 1.}
The dg category $\cb'$ has the same objects as $\cb$, and the dg functor
$\cb\to\cb'$ is the identity on objects. 
The Hom-complexes in the dg category $\cb'$, as well as the 
dg functor $\cb\to\cb'$, are specified by giving the cochain map
$\Hom_\cb(B_1,B_2)\to\Hom_{\cb'}(B_1,B_2)$ for every pair
of objects $B_1,B_2\in\cb$. We declare this to be the natural
cochain map
\[\xymatrix@C+20pt{
\Hom_\cb(B_1,B_2)\ar[r] & \holim \Hom_\cb(B_1,B_2)
}\]
where on the right we mean the homotopy limit, in the sense of 
\autoref{subsec:hocolimdgcats}, of the inverse sequence
\[\xymatrix{
\cdots\ar[r]&\Hom_\cb(B_1,B_2)\ar[r]^-\id &\Hom_\cb(B_1,B_2)\ar[r]^-\id &\Hom_\cb(B_1,B_2)
}\]
The composition law in the category $\cb'$, giving the map
\[\xymatrix@C+20pt{
\Hom_{\cb'}(B_1,B_2)\otimes\Hom_{\cb'}(B_2,B_3)\ar[r] &
\Hom_{\cb'}(B_1,B_3)
}\]
is as in \autoref{rmk:holimdg} and \autoref{def:differentialring}.

It is obvious that the dg functor $\cb\to\cb'$ is a quasi-equivalence.

\medskip

\noindent\emph{Step 2.} In this step we will produce a dg functor $\cb''\to\cb'$,
with $\cb'$ as in Step 1. Let us start with the following useful definition.

\begin{definition}\label{def:oartiallyorderedset}
Let $\cs$ be the set of functions $f\colon\NN\to\NN\cup\{0\}$ satisfying
\begin{enumerate}
\item[{\rm (i)}] $f$ is non-decreasing, meaning $f(n)\leq f(n+1)$ for all $n\in\NN$.
\item[{\rm (ii)}] $f(n)\to\infty$ as $n\to\infty$.
\end{enumerate}
We turn $\cs$ into a partially ordered set by setting $f\leq f'$ if $f(n)\leq f'(n)$ for all
$n\in\NN$.
\end{definition}

Back in \autoref{subsec:variantB} we introduced the subcomplexes
$\Hom_n(B_1,B_2)\subset\Hom_\cb(B_1,B_2)$ for $B_1,B_2$ objects of $\cb$ and 
for $n\in\NN$. We extend
this definition now, allowing $n=0$, 
by declaring $\Hom_0(B_1,B_2):=\Hom_\cb(B_1,B_2)$.

Given a pair of objects $B_1,B_2\in\cb$, a function $f\in\cs$ and an integer
$k>0$, we can combine the information recalled in the discussion just prior
to Step 2 to form the inverse sequence
\[\xymatrix{
\cdots\ar[r]&\Hom_{f(3)}(B_1,B_2)^{\leq k}\ar[r] &\Hom_{f(2)}(B_1,B_2)^{\leq k}\ar[r] &\Hom_{f(1)}(B_1,B_2)^{\leq k}
}\]
If $f\leq g$ are elements of $\cs$ then there is a natural map of inverse sequences
\[\xymatrix{
\cdots\ar[r]&\Hom_{g(3)}(B_1,B_2)^{\leq k}\ar[r]\ar[d] &\Hom_{g(2)}(B_1,B_2)^{\leq k}\ar[r]\ar[d] &\Hom_{g(1)}(B_1,B_2)^{\leq k}\ar[d] \\
\cdots\ar[r]&\Hom_{f(3)}(B_1,B_2)^{\leq k}\ar[r] &\Hom_{f(2)}(B_1,B_2)^{\leq k}\ar[r] &\Hom_{f(1)}(B_1,B_2)^{\leq k}
}\]
which allows us to view the construction, for fixed $B_1,B_2,k$, as a functor
from $\cs\opp$ to the category of inverse sequences of cochain complexes. We can now form the dg category
$\cb''$; the objects are identical to those of $\cb$. For a pair of objects
$B_1,B_2\in\cb$ we declare
\[
\Hom_{\cb''}(B_1,B_2):=
\colim_{f\in\cs\opp,k\to\infty}
\left(\holim_{n\to\infty}\Hom_{f(n)}(B_1,B_2)^{\leq k}\right)
\]
This means that, for fixed $f\in\cs$ and $k\in\NN$, we take the homotopy inverse 
limit of the
sequence depicted above. And as this is contravariantly functorial in $f\in\cs$ 
and covariantly
functorial in $k\in\NN$, we can form the (ordinary) colimit of these cochain 
complexes. 
Moreover, for each $f\in\cs$ and $k\in\NN$, there is an obvious map of inverse systems
\[\xymatrix{
\cdots\ar[r]&\Hom_{f(3)}(B_1,B_2)^{\leq k}\ar[r]\ar[d] &\Hom_{f(2)}(B_1,B_2)^{\leq k}\ar[r]\ar[d] &\Hom_{f(1)}(B_1,B_2)^{\leq k}\ar[d] \\
\cdots\ar[r]&\Hom_\cb(B_1,B_2)\ar[r]^-\id &\Hom_\cb(B_1,B_2)\ar[r]^-\id &\Hom_\cb(B_1,B_2)
}\]
and, taking homotopy inverse limits, we deduce a map 
\[\xymatrix{
\holim_{n\to\infty}\Hom_{f(n)}(B_1,B_2)^{\leq k} \ar[r] &
\holim_{n\to\infty}\Hom_\cb(B_1,B_2)\ar@{=}[r] &
\Hom_{\cb'}(B_1,B_2)
}\]
with $\cb'$ the dg category of Step 1. And as this map is 
compatible with increasing $k\in\NN$ and decreasing $f\in\cs$, it gives
rise to a cochain map 
\[\xymatrix{
\Hom_{\cb''}(B_1,B_2)\ar@{=}[r] &
\displaystyle
\colim_{f\in\cs\opp,k\to\infty}
\left(
\holim_{n\to\infty}\Hom_{f(n)}(B_1,B_2)^{\leq k}\right) \ar[r] &
\Hom_{\cb'}(B_1,B_2)
}\]
This defines for us what the dg functor $\cb''\to\cb'$ does on Hom-complexes.

It remains to deal with composition. Suppose $f,g\in\cs$ and 
$k,l\in\NN$ are given. We form $h\in\cs$ as follows
\[
h(n):=
\begin{cases}
0 & \text{unless $\min(f(n),g(n))>\max(2k,2l)$} \\
\left\lfloor\frac{\min(f(n),g(n))}2\right\rfloor & \text{otherwise} \\
\end{cases}
\]
where the symbol $\lfloor x\rfloor$ means the integer part of $x$; that is
the function $\lfloor -\rfloor$ takes a real number $x$ to
the largest integer not larger than $x$.
With this choice we have that either $h(n)=0$, or else $h(n)\geq\max(k,l,1)$ and 
$f(n),g(n)$ are both $\geq 2h(n)$.  \autoref{compn} tells us that the 
composition 
\[\xymatrix{
\Hom_{f(n)}(B_1,B_2)^{\leq k}\otimes_\kk
\Hom_{g(n)}(B_2,B_3)^{\leq l}\ar[r] &
\Hom_{h(n)}(B_1,B_3)^{\leq k+l}
}\]
is well-defined, and \autoref{rmk:holimdg} and 
\autoref{def:differentialring} allows us to pass
to the homotopy inverse limits, producing a map
\[
\xymatrix@C+20pt{
\makebox[.5\textwidth]{$\left(\holim_{n\to\infty}\Hom_{f(n)}(B_1,B_2)^{\leq k}\right)
\otimes_\kk
\left(\holim_{n\to\infty}\Hom_{g(n)}(B_2,B_3)^{\leq l}\right)$} \ar[d] \\
\holim_{n\to\infty}\Hom_{h(n)}(B_1,B_3)^{\leq k+l}
}
\]
Now passing to the colimit over $f,g\in\cs$ and $k,l\in\NN$, we produce
the composition map
\[
\xymatrix@C+20pt{
\Hom_{\cb''}(B_1,B_2)\otimes\Hom_{\cb''}(B_2,B_3)\ar[r] &
\Hom_{\cb''}(B_1,B_3).
}
\]

\medskip

\noindent\emph{Step 3.}
In this step we will produce a dg functor $\cb''\to\overline\cb$,
with $\cb''$ as in Step 2. On objects this dg functor is
the identity. We need to explain, for every pair of objects $B_1,B_2\in\cb$, the map
$\Hom_{\cb''}(B_1,B_2)\to\Hom_{\overline\cb}(B_1,B_2)$. 

The idea is simple enough. In the paragraphs preceding 
\autoref{compat} me introduced the maps $p_n^k:\Hom_n(B_1,B_2)^{\leq k}\to
\Homr_n(B_1,B_2)^{\leq k}$. Back then we assumed $n\geq1$. In
Step 2 we extended the definition of $\Hom_n(B_1,B_2)^{\leq k}$
to $n=0$ by setting $\Hom_0(B_1,B_2)^{\leq k}=\Hom_\cb(B_1,B_2)^{\leq k}$,
and now we extend the definition of $\Homr_n(B_1,B_2)^{\leq k}$ to $n=0$
by setting $\Homr_0(B_1,B_2)^{\leq k}=0$.
Now let $B_1,B_2\in\cb$, $f\in\cs$ and $k\in\NN$ be given; the maps
$p_{f(i)}^k$ provide morphisms of inverse sequences
\[\xymatrix{
\cdots\ar[r]&\Hom_{f(3)}(B_1,B_2)^{\leq k}\ar[r]\ar[d]^{p_{f(3)}^k} &\Hom_{f(2)}(B_1,B_2)^{\leq k}\ar[r]\ar[d]^{p_{f(2)}^k} &\Hom_{f(1)}(B_1,B_2)^{\leq k}\ar[d]^{p_{f(1)}^k} \\
\cdots\ar[r]&\Homr_{f(3)}(B_1,B_2)^{\leq k}\ar[r] &\Homr_{f(2)}(B_1,B_2)^{\leq k}\ar[r] &\Homr_{f(1)}(B_1,B_2)^{\leq k}
}\]
And the map $\Hom_{\cb''}(B_1,B_2)\to\Hom_{\overline\cb}(B_1,B_2)$ is obtained
by taking first homotopy inverse limits and then colimits, as in
\[
\xymatrix{
\displaystyle
\colim_{f\in\cs\opp,k\to\infty}\left(\holim_{n\to\infty}\Hom_{f(n)}(B_1,B_2)^{\leq k}\right) 
\ar[d]^{\colim_{f\in\cs\opp,k\to\infty}\left(\holim_{n\to\infty}p_{f(n)}^k\right)} \\
\displaystyle
\colim_{f\in\cs\opp,k\to\infty}\left(\holim_{n\to\infty}\Homr_{f(n)}(B_1,B_2)^{\leq k}
\right)
}
\]
The composition law on the dg category $\overline\cb$, as well as the fact that
the map $\Hom_{\cb''}(B_1,B_2)\to\Hom_{\overline\cb}(B_1,B_2)$ respects
composition, can be deduced by combining \autoref{compat} with
\autoref{rmk:holimdg} and
\autoref{def:differentialring}.

From \autoref{holim} we learn that the dg functor $\cb''\to\overline\cb$ is a
quasi-equivalence.

\medskip

\noindent\emph{Step 4.} In this step we will produce a dg functor $\wt B\to\overline B$. 
As in the previous
steps the functor is the identity on objects. We define
$\Hom_{\wt\cb}(B_1,B_2)$ by the formula
\[
\Hom_{\wt\cb}(B_1,B_2):=\prod_{i=-\infty}^\infty\prod_{j=-\infty}^\infty\bigl(\Hom_{\cc}(\sh[-i]{B_1^i},\sh[-j]{B_2^j})^{\le j-i}\bigr)^{\ge j-i}
\]
and we need to explain the map $\Hom_{\wt\cb}(B_1,B_2)\to\Hom_{\overline\cb}(B_1,B_2)$.

Now recall: by definition (see just before \autoref{compat})
\[
\Homr_{n}(B_1,B_2)^{\leq k}:=\left(\prod_{i=-n}^n\prod_{j=-n}^n\bigl(\Hom_{\cc}(\sh[-i]{B_1^i},\sh[-j]{B_2^j})^{\le j-i}\bigr)^{\ge j-i}\right)^{\leq k}
\]
Hence there is an obvious map 
\[
\xymatrix{
\Hom_{\wt\cb}(B_1,B_2)^{\leq k}\ar[r] &
\Homr_{n}(B_1,B_2)^{\leq k}
}
\]
which is just the functor $(-)^{\leq k}$ applied to the projection from the large product to the smaller one. For any $f\in\cs$ this induces a map
\[
\xymatrix{
\Hom_{\wt\cb}(B_1,B_2)^{\leq k}\ar[r] &
\holim_{n\to\infty}\Homr_{f(n)}(B_1,B_2)^{\leq k}
}
\]
Now taking the colimit over $f\in\cs$ and as $k\to\infty$ produces a map
\[
\xymatrix{
\Hom_{\wt\cb}(B_1,B_2)\ar[r] &
\displaystyle\colim_{f\in\cs\opp, k\to\infty}\left(
\holim_{n\to\infty}\Homr_{f(n)}(B_1,B_2)^{\leq k}\right)
}
\]
which is our definition of the morphism
$\Hom_{\wt\cb}(B_1,B_2)\to\Hom_{\overline\cb}(B_1,B_2)$.

It is easy to check that the map is compatible with composition, hence defines a dg functor $\wt\cb\to\overline\cb$. And it is obvious that this map is a quasi-equivalence.

\medskip

This concludes the proof of \autoref{qfunVB}, since $\wt\cb$ is manifestly quasi-equivalent to $\cv$ under the natural dg functor $\wt\cb\to\cv$. For later use, let us summarize here the sequence of dg categories and dg functors constructed along the proof:
\begin{equation}\label{eqn:biqzigzag}
\xymatrix@C+20pt{
&\widetilde\cB\ar[dl]_-{\wt\fF}\ar[dr]_-{\overline\fF}&&\cB''\ar[dl]_-{\fF''}\ar[dr]_-{\fF'}&&\cB\ar[dl]_-{\fF}\\
\cV&&\overline\cB&&\cB'&
}
\end{equation}
where all dg functors but $\fF'\colon\cB''\to\cB'$ are quasi-equivalences.

\section{Uniqueness for $\Da(\ca)$}\label{sect:D(A)}

This section is completely devoted to the proof of \autoref{thm:main1} (1). It requires some technical observations which are contained in \autoref{subsec:prelres}. Some straightforward but interesting applications are discussed in \autoref{subsec:appl1}.

\subsection{A brief summary of the setting}\label{subsec:summarysetting}

Let $\ca$ be an abelian category and let $(\cc,\fE)$ be an enhancement of $\Da(\ca)$. In view of \autoref{ex:enhancementsabs}, \autoref{thm:main1} (1) will be proved once we show that there is an isomorphism between $\cc$ and $\Ddga(\ca)$ in $\Hqe$.

Let us first of all define a morphism between these dg categories. One may construct the dg categories $\cv=\cv^?(\ca)$ and $\cb=\cb^?(\ca)$ as in \autoref{sect:enB} which, in view of \autoref{CperfB}, come with isomorphisms in $\Hqe$
\[
f^K\colon\Perf{\cv}\equiva\dgCa(\ca),\qquad f^D\colon\Perf{\cb}\equiva\cc.
\]
Furthermore, by \autoref{qfunVB}, there is a morphism $u\in\Hom_\Hqe(\cv,\cb)$ such that $H^0(u)$ is the natural functor from $\Va(\ca)$ to $\Ba(\ca)$. Consider then the morphism in $\Hqe$ 
\[
f:=\Ind(u)\colon\Perf{\cv}\to\Perf{\cb}.
\]
As a conclusion, we have the exact functors $\fF_1,\fF_2\colon\Ka(\ca)\to\Da(\ca)$ defined as follows:
\begin{equation}\label{eqn;:F1F2}
\fF_1:=\fE\comp H^0\left(f^D\comp f\comp (f^K)^{-1}\right)\qquad \fF_2:=\fQ,
\end{equation}
where $\fQ$ is the quotient functor in \eqref{eqn:quotfun}. Set, for later use,
\begin{equation}\label{eqn:qf}
g:=f^D\comp f\comp (f^K)^{-1}\colon\dgCa(\ca)\to\cc.
\end{equation}

The following is clear from the definitions.

\begin{lem}\label{lem:fun1}
In the setting above, there is a natural isomorphism
\[
\theta\colon\fF_1\rest{\Va(\ca)}\equiva\fF_2\rest{\Va(\ca)}
\]
of exact functors.
\end{lem}

\subsection{Some preliminary results}\label{subsec:prelres}

We discuss a general result which applies nicely to the setting in the previous section.

\begin{prop}\label{prop:fun2}
Assume that $\ca$ is an abelian category and that $\fG_1,\fG_2:\Ka(\ca)\to\Da(\ca)$ are exact functors such that
\begin{enumerate}
\item[{\rm (i)}] There is a natural isomorphism $\theta\colon\fG_1\rest{\Va(\ca)}\equiva\fG_2\rest{\Va(\ca)}$;
\item[{\rm (ii)}] Suppose $a\leq b$ are integers. If $V^*\in\Ob\left(\Va(\ca)\right)$ is such that $V^i=0$ for all $i\notin[a,b]$, then $\fG_1(V^*)\iso \fG_2(V^*)\in\Ob\left(\D(\ca)^{\leq b}\cap\D(\ca)^{\geq a}\right)$. 
\end{enumerate}
Then, for every $A^*\in\Ob\left(\Ka(\ca)\right)$, there exists an isomorphism
$\wt\theta_{A^*}\colon\fG_1(A^*)\equiva\fG_2(A^*)$ such that the following square commutes in $\Da(\ca)$
\[
\xymatrix@C+30pt{
	\fG_1(V^*)\ar[r]^-{\fG_1(h)}\ar[d]_{\theta_{V^*}} & \fG_1(A^*) \ar[d]^{\wt\theta_{A^*}}\\
	\fG_2(V^*)\ar[r]^-{\fG_2(h)}                & \fG_2(A^*)
}
\]
for every $V^*\in\Ob\left(\Va(\ca)\right)$ and every morphism $h:V^*\to A^*$ of $\Ka(\ca)$.
\end{prop}

\begin{proof}
Given $A^*\in\Ka(\ca)$, we denote by $K^i$ the kernel of the differential $d^i\colon A^i\to A^{i+1}$ and by $\alpha^i\colon A^{i-1}\to K^i$ the natural factorization of the differential $d^{i-1}\colon A^{i-1}\to A^i$.

Consider the commutative square
\[
\xymatrix@C+10pt{
	\fG_1\left(\bigoplus_{i\in\ZZ}\sh[-i]{A^{i-1}}\right) 
	\ar[rr]^-{\fG_1(\bigoplus_{i\in\ZZ}\sh[-i]{\alpha^i})}\ar[d]_{\theta} & & 
	\fG_1\left(\bigoplus_{i\in\ZZ}\sh[-i]{K^i}\right)\ar[d]^\theta
	\\
	\fG_2\left(\bigoplus_{i\in\ZZ}\sh[-i]{A^{i-1}}\right)
	\ar[rr]^-{\fG_2(\bigoplus_{i\in\ZZ}\sh[-i]{\alpha^i})} & &
	\fG_2\left(\bigoplus_{i\in\ZZ}\sh[-i]{K^i}\right)
}
\]
where the vertical maps are isomorphisms by (i). Of course, it could be split into the direct sum of two similar commutative diagrams, where the direct sums are indexed, respectively, over $i\ge1$ and $i\le0$. By (ii), \autoref{prop:coprodfun} and \autoref{prop:prodfun}, we obtain two commutative diagrams
\[
\xymatrix@C+10pt{
	\prod_{i\ge1}\fG_1(\sh[-i]{A^{i-1}}) 
	\ar[rr]^-{\prod_{i\ge1}\fG_1(\sh[-i]{\alpha^i})}\ar[d]_{\prod_{i\ge1}\theta} & & 
	\prod_{i\ge1}\fG_1(\sh[-i]{K^i})\ar[d]^{\prod_{i\ge1}\theta}
	\\
	\prod_{i\ge1}\fG_2(\sh[-i]{A^{i-1}})
	\ar[rr]^-{\prod_{i\ge1}\fG_2(\sh[-i]{\alpha^i})} & &
	\prod_{i\ge1}\fG_2(\sh[-i]{K^i})
}
\]
and
\[
\xymatrix@C+10pt{
	\coprod_{i\le0}\fG_1(\sh[-i]{A^{i-1}}) 
	\ar[rr]^-{\coprod_{i\le0}\fG_1(\sh[-i]{\alpha^i})}\ar[d]_{\coprod_{i\le0}\theta} & & 
	\coprod_{i\le0}\fG_1(\sh[-i]{K^i})\ar[d]^{\coprod_{i\le0}\theta}
	\\
	\coprod_{i\le0}\fG_2(\sh[-i]{A^{i-1}})
	\ar[rr]^-{\coprod_{i\le0}\fG_2(\sh[-i]{\alpha^i})} & &
	\coprod_{i\le0}\fG_2(\sh[-i]{K^i}).
}
\]
For each $i\in\ZZ$, $\theta$ induces by (i) an isomorphism of distinguished triangles
\[
\xymatrix@C+30pt{
	\fG_1(\sh[-i]{A^{i-1}}) \ar[r]^-{\fG_1(\sh[-i]{\alpha^i}))}\ar[d]_{\theta} & 
	\fG_1(\sh[-i]{K^i})\ar[d]^\theta \ar[r]^-{\fG_1(\sh[-i]{\ph^i})} &
	\fG_1\left(\sh[-i]{\cone{\alpha^i}}\right)\ar[d]^{\theta'_i}&
	\\
	\fG_2(\sh[-i]{A^{i-1}})\ar[r]^-{\fG_2(\sh[-i]{\alpha^i}))} 
	& \fG_2(\sh[-i]{K^i})\ar[r]^-{\fG_2(\sh[-i]{\varphi^i})} &
	\fG_2\left(\sh[-i]{\cone{\alpha^i}}\right),
}
\]
where $\varphi^i$ is the morphism of \eqref{eqn:phi}. By taking products over the integers $i\geq 1$ and coproducts over $i\leq 0$, this produces an isomorphism of distinguished triangles
\begin{equation}\label{eqn:commsq1}
\xymatrix@C+30pt{
	\fG_1\left(\bigoplus_{i\in\ZZ}\sh[-i]{A^{i-1}}\right) \ar[r]^-{\fG_1(\bigoplus_{i\in\ZZ}\sh[-i]{\alpha^i}))}\ar[d]_{\theta}  &
	\fG_1\left(\bigoplus_{i\in\ZZ}\sh[-i]{K^i}\right)\ar[d]^\theta \ar[r]^-{\fG_1(\varphi)} & 
	\fG_1\left(\bigoplus_{i\in\ZZ}\sh[-i]{\cone{\alpha^i}}\right)\ar[d]^{\theta'}
	\\
	\fG_2\left(\bigoplus_{i\in\ZZ}\sh[-i]{A^{i-1}}\right)\ar[r]^-{\fG_2(\bigoplus_{i\in\ZZ}\sh[-i]{\alpha^i}))} 
	&  \fG_2\left(\bigoplus_{i\in\ZZ}\sh[-i]{K^i}\right)\ar[r]^-{\fG_2(\varphi)} & 
	\fG_2\left(\bigoplus_{i\in\ZZ}\sh[-i]{\cone{\alpha^i}}\right),
}
\end{equation}
where $\varphi:=\bigoplus_{i\in\ZZ}\sh[-i]{\varphi^i}$ and $\theta':=\bigoplus_{i\in\ZZ}\theta'_i$. Hence the rightmost square in \eqref{eqn:commsq1} commutes.

For $i\in\ZZ$ consider now the inclusion $\rho^i\colon K^{i-1}\mono A^{i-1}$. The composite $K^{i-1}\mor{\rho^i} A^{i-1}\mor{\alpha^i} K^i$ vanishes in $\ca$. Thus $\rho^i$ factors in $\Ka(\ca)$ as $K^{i-1}\mor{\psi^{i-1}}\sh[-1]{\cone{\alpha^i}}\to A^{i-1}$, where $\psi^j$ is the morphism of \eqref{eqn:psi}. From this we deduce, for each $i\in\ZZ$, a diagram
\[
\xymatrix@C+30pt{
	\fG_1(\sh[-i+1]{K^{i-1}})\ar[d]^\theta \ar[r]^-{\fG_1(\sh[-i+1]{\psi^{i-1}})} &
	\fG_1(\sh[-i]{\cone{\alpha^i}})\ar[r]\ar[d]^{\theta'_i}&
	\fG_1(\sh[-i+1]{A^{i-1}}) \ar[d]^{\theta}
	\\
	\fG_2(\sh[-i+1]{K^{i-1}})\ar[r]^-{\fG_2(\sh[-i+1]{\psi^{i-1}})} &
	\fG_2(\sh[-i]{\cone{\alpha^i}})\ar[r]&
	\fG_2(\sh[-i+1]{A^{i-1}}). 
}
\]
Note that if we delete the middle column the resulting square commutes because
of the naturality of the isomorphism $\theta$.
If we delete the left column the resulting square commutes
by the definition of $\theta'_i$. It follows that the difference
between the composites in the square on the left is annihilated
by the map $\fG_2(\sh[-i]{\cone{\alpha^i}})\to\fG_2(\sh[-i+1]{A^{i-1}})$, and hence must factor through
$\fG_2(\sh[-i]{K^i})$. But by (ii) there can be no non-zero map
$\fG_1(\sh[-i+1]{K^{i-1}})\to\fG_2(\sh[-i]{K^i})$, and hence the square on the left must also commute. Taking the coproduct over $i\leq 0$ and the product over $i\geq 1$ and then assembling, we deduce a commutative
square
\[
\xymatrix@C+30pt{
	\fG_1\left(\bigoplus_{i\in\ZZ}\sh[-i+1]{K^{i-1}}\right)\ar[d]_-{\theta} \ar[r]^-{\fG_1(\psi)} &
	\fG_1\left(\bigoplus_{i\in\ZZ}\sh[-i]{\cone{\alpha^i}}\right)\ar[d]^{\theta'}\\
	\fG_2\left(\bigoplus_{i\in\ZZ}\sh[-i+1]{K^{i-1}}\right))\ar[r]^-{\fG_2(\psi)} &
	\fG_2\left(\bigoplus_{i\in\ZZ}\sh[-i]{\cone{\alpha^i}}\right)
}
\]
where $\psi:=\bigoplus_{i\in\ZZ}\sh[-i]{\psi^i}$.

By putting together this commutative square and rightmost commutative square in \eqref{eqn:commsq1}, we obtain the commutative square
\[
\xymatrix@C+30pt{
	\fG_1\left(\bigoplus_{i\in\ZZ}\sh[-i+1]{K^{i-1}}\right)\ar[d]_-{\theta} \ar[r]^-{\fG_1(\varphi+\psi)} &
	\fG_1\left(\bigoplus_{i\in\ZZ}\sh[-i]{\cone{\alpha^i}}\right)\ar[d]^{\theta'}\\
	\fG_2\left(\bigoplus_{i\in\ZZ}\sh[-i+1]{K^{i-1}}\right)\ar[r]^-{\fG_2(\varphi+\psi)} &
	\fG_2\left(\bigoplus_{i\in\ZZ}\sh[-i]{\cone{\alpha^i}}\right)
}
\]
which can be completed to an isomorphism of distinguished triangles
\[
\xymatrix@C+30pt{
	\fG_1\left(\bigoplus_{i\in\ZZ}\sh[-i+1]{K^{i-1}}\right)\ar[d]_-{\theta} \ar[r]^-{\fG_1(\varphi+\psi)} &
	\fG_1\left(\bigoplus_{i\in\ZZ}\sh[-i]{\cone{\alpha^i}}\right)\ar[d]^{\theta'}\ar[r]^-{\fG_1(\sigma)} & \fG_1(A^*)\ar[d]^-{\wt\theta_{A^*}}\\
	\fG_2\left(\bigoplus_{i\in\ZZ}\sh[-i+1]{K^{i-1}}\right)\ar[r]^-{\fG_2(\varphi+\psi)} &
	\fG_2\left(\bigoplus_{i\in\ZZ}\sh[-i]{\cone{\alpha^i}}\right)\ar[r]^-{\fG_2(\sigma)} & \fG_2(A^*)
}
\]
Here we use the fact that, as we observed in the proof of \autoref{prop:genK},
the mapping cone of $\ph+\psi$ is isomorphic in $\K(\ca)$ to $A^*$.

Clearly $\wt\theta_{A^*}$ is an isomorphism and thus, to complete the proof it remains to prove that the square in the statement commutes.

To this end observe that both squares in the diagram
\begin{equation}\label{eqn:commsq3}
\xymatrix@C+40pt{
	\fG_1\left(\bigoplus_{i\in\ZZ}\sh[-i+1]{K^{i-1}}\right)\ar[d]_-{\theta} \ar[r]^-{\fG_1(\varphi)} &
	\fG_1\left(\bigoplus_{i\in\ZZ}\sh[-i]{\cone{\alpha^i}}\right)\ar[d]^{\theta'}\ar[r]^-{\fG_1(\sigma)} & \fG_1(A^*)\ar[d]^-{\wt\theta_{A^*}}\\
	\fG_2\left(\bigoplus_{i\in\ZZ}\sh[-i+1]{K^{i-1}}\right)\ar[r]^-{\fG_2(\varphi)} &
	\fG_2\left(\bigoplus_{i\in\ZZ}\sh[-i]{\cone{\alpha^i}}\right)\ar[r]^-{\fG_2(\sigma)} & \fG_2(A^*)
}
\end{equation}
commute by construction, and thus the outside square commutes as well. Moreover, any map $V^*\to A^*$, where $V^*\in\Va(\ca)$, must factor as $V^*\mor{\beta}\bigoplus_{i\in\ZZ}\sh[-i]{K^i}\mor{\sigma\comp\varphi} A^*$, and in the diagram
\begin{equation}\label{eqn:commsq4}
\xymatrix@C+40pt{
	\fG_1\left(V^*\right)\ar[d]_-{\theta}\ar[r]^-{\fG_1(\beta)} &  \fG_1\left(\bigoplus_{i\in\ZZ}\sh[-i]{K^i}\right)\ar[d]^-{\theta}\ar[r]^-{\fG_1(\sigma\comp\varphi)} &\fG_1(A^*)\ar[d]^{\wt\theta_{A^*}}\\
	\fG_2\left(V^*\right)\ar[r]^-{\fG_2(\beta)} &  \fG_2\left(\bigoplus_{i\in\ZZ}\sh[-i]{K^i}\right)\ar[r]^-{\fG_2(\sigma\comp\varphi)} &\fG_2(A^*)
}
\end{equation}
both squares commute: the left hand square by the naturality of $\theta$ and the right hand square (which is the outside square in \eqref{eqn:commsq3}) by the observation above. Thus we deduce the commutativity of the outside square in \eqref{eqn:commsq4}, completing the proof.
\end{proof}

We now want to apply this result in the setting of \autoref{subsec:summarysetting}.

\begin{cor}\label{cor:comsqimp}
Let $\ca$ be an abelian category and let $\fF_1,\fF_2:\Ka(\ca)\to\Da(\ca)$ be as in \eqref{eqn;:F1F2}. Then for every $A^*\in\Ob\left(\Ka(\ca)\right)$ there is an isomorphism $\wt\theta_{A^*}\colon\fF_1(A^*)\equiva\fF_2(A^*)$ such that the following square commutes in $\Da(\ca)$
\[
\xymatrix@C+30pt{
\fF_1(V^*)\ar[r]^-{\fF_1(h)}\ar[d]_-{\theta_{V^*}} &\fF_1(A^*) \ar[d]^-{\wt\theta_{A^*}}\\
\fF_2(V^*)\ar[r]^-{\fF_2(h)} & \fF_2(A^*)
}
\]
for every $V^*\in\Ob\left(\Va(\ca)\right)$ and every morphism $h\colon V^*\to A^*$ of $\Ka(\ca)$.
\end{cor}

\begin{proof}
The result is a direct consequence of \autoref{prop:fun2}. Indeed, assumption (i) is satisfied thanks to \autoref{lem:fun1}. On the other hand, assumption (ii) clearly holds since $\fF_2=\fQ$. 
\end{proof}

\subsection{Proof of \autoref{thm:main1} (1)}\label{subsec:proofThm1}

This part of the argument is very similar to \cite[Section 6]{LO}. The key differences are our approach to generation for $\Da(\ca)$ and all of its consequences in the previous section. We are in the setting of \autoref{subsec:summarysetting}.

Suppose that $L^*\in\Acya{\ca}$ is an acyclic object. By \autoref{cor:comsqimp}, $\fF_1(L^*)\iso\fF_2(L^*)$ in $\Da(\ca)$. Since $\fF_2(L^*)$ is zero as $\fF_2=\fQ$ is the Verdier quotient map, we have $\fF_1(L^*)\iso 0$, whence the functor $\fF_1$ must factor through $\fQ$. Let us write this as
\[
\Ka(\ca)\mor{\fQ}\Da(\ca)=\Ka(\ca)/\Acya(\ca)\mor{\fF_1'}\Da(\ca).
\] 
As an immediate consequence, the morphism $g$ in \eqref{eqn:qf} must factor through the Drinfeld quotient as follows:
\[
\dgCa(\ca)\lto\Ddga(\ca)=\dgCa(\ca)/\dgAcya(\ca)\mor{g'}\cc.
\]
Note that, by construction, we have
\begin{equation}\label{eqn:F1'}
\fF'_1=\fE\comp H^0(g').
\end{equation}
The proof will be complete once we show that $H^0(g')$ is an equivalence. As $\fE$ is an exact equivalence, we are reduced to showing that $\fF'_1$ is an equivalence.

Let us prove that $\fF'_1$ is fully faithful. By \autoref{cor:genD}, it suffices to show that, for any pair of
objects $V_1,V_2\in\Ob(\Va(\ca))$,  the functor $\fF'_1$ induces
an isomorphism
\[
\Hom_{\Ka(\ca)/\Acya(\ca)}(V_1,V_2)\equiva\Hom_{\Da(\ca)}(\fF_1(V_1),\fF_1(V_2))
\]

To prove the injectivity choose in the category $\Ka(\ca)/\Acya(\ca)$ 
a morphism $h:V_1\to V_2$ mapping to zero under $\fF'_1$. We may represent $h$ in $\Ka(\ca)$ as
\[
\xymatrix@C+20pt@R-20pt{
V_1\ar[dr]^{a} & & V_2\ar[dl]_-{b}\\
& A\ar[dl] & \\
L & & 
}
\]
where $V_2\mor{b} A\to L$ is a distinguished triangle with $L\in\Acya(\ca)$, and where $h=b^{-1}\comp a$ in $\Ka(\ca)/\Acya(\ca)$. But then $0\iso \fF'_1(h)=\fF'_1(b)^{-1}\comp\fF'_1(a)$, and we deduce that 
$\fF'_1(a)=\fF_1(a)\iso 0$. 
By \autoref{cor:comsqimp}, there is an isomorphism $\wt\theta_A\colon\fF_1(A)\equiva\fF_2(A)$ rendering commutative the square
\[
\xymatrix@C+30pt{
	\fF_1(V)\ar[r]^-{\fF_1(a)}\ar[d]_{\theta_V} & \fF_1(A) \ar[d]^{\wt\theta_A}\\
	\fF_2(V)\ar[r]^-{F_2(a)} & \fF_2(A).
}
\]
The vertical maps are isomorphisms, hence the vanishing of $\fF_1(a)$ implies that $\fF_2(a)\iso 0$. Hence $a$ vanishes in $\Ka(\ca)/\Acya(\ca)$ as $\fF_2=\fQ$ is the Verdier quotient. Thus $h=0$.

As for the surjectivity, let $h\colon\fF_1(V_1)\to\fF_1(V_2)$ be a morphism in the category $\Da(\ca)$. By \autoref{lem:fun1}, $\theta$ induces isomorphisms $\fF_1(V_i)\iso\fF_2(V_i)$, for $i=1,2$. Hence there is a morphism $h'\colon\fF_2(V_1)\to\fF_2(V_2)$ making the following diagram commutative in $\Da(\ca)$
\[
\xymatrix@C+30pt{
	\fF_1(V_1)\ar[r]^-{h}\ar[d]_{\theta_{V_1}} & \fF_1(V_2) \ar[d]^{\theta_{V_2}}\\
	\fF_2(V_1)\ar[r]^-{h'} & \fF_2(V_2).
}
\]
We can represent $h'$ in $\Ka(\ca)$ as
\[
\xymatrix@C+20pt@R-20pt{
	V_1\ar[dr]^{a'} & & V_2\ar[dl]_-{b'}\\
	& A'\ar[dl] & \\
	L' & & 
}
\]
where $V_2\mor{b'} A'\to L'$ is a distinguished triangle with $L'\in\Acya(\ca)$, and where $h'=(b')^{-1}\comp a'$ in $\Da(\ca)$. By \autoref{cor:comsqimp}, there is an isomorphism $\wt\theta_A\colon\fF_1(A)\equiva\fF_2(A)$ making the diagram
\[
\xymatrix@C+30pt{
	\fF_1(V_1)\ar[r]^-{F_1(a')}\ar[d]_{\theta_{V_1}} & \fF_1(A) \ar[d]^{\wt\theta_A}&
	\fF_1(V_2)\ar[l]_-{\fF_1(b')}\ar[d]^{\theta_{V_2}}\\
	\fF_2(V_1)\ar[r]^-{\fF_2(a')} & \fF_2(A)& \fF_2(V_2).\ar[l]_-{\fF_2(b')}   
}
\]
commutative. From this we deduce that
\[
h=\theta_{V_2}^{-1}\comp h'\comp\theta_{V_1}=\theta_{V_2}^{-1}\comp \fF_2(b')^{-1}\comp\fF_2(a')\comp\theta_{V_1}=\fF_1(b')^{-1}\comp\fF_1(a').
\]
Hence, by definition, $\fF'_1(h')=h$.

As for the essential surjectivity of $\fF'_1$, observe that, since it is fully faithful, its essential image is a full and thick (because the source category $\Da(\ca)$ is idempotent complete) triangulated subcategory of $\Da(\ca)$ containing $\fF_1(\Va(\ca))$. But $\fF_1(\Va(\ca))=\fF_2(\Va(\ca))=\Ba(\ca)$ by \autoref{lem:fun1} and so, by \autoref{cor:genD}, the essential image of $\fF'_1$ coincides with $\Da(\ca)$.

This concludes the proof of \autoref{thm:main1} (1) as we proved that $(g')^{-1}$ is an isomorphism between $\cc$ and $\Ddga(\ca)=\dgCa(\ca)/\dgAcya(\ca)$ in $\Hqe$.

\begin{remark}\label{rmk:semistrong}
It should be noted that our proof shows more: $\Da(\ca)$ has a semi-strongly unique dg enhancement, for any abelian category $\ca$. Indeed as we observed above, by \autoref{cor:comsqimp}, given any object $A\in\Ob(\Da(\ca))=\Ob(\Ka(\ca))$, there is an isomorphism
$\fF_1(A)\iso\fF_2(A)$ in $\Da(\ca)$. But $\fF_1(A)\iso\fF'_1(A)=\fE\comp H^0(g')(A)$ by \eqref{eqn:F1'}. Moreover $\fF_2(A)=A$ by definition. Hence $H^0(g')(A)\iso\fE^{-1}(A)$, for all $A\in\Ob(\Ddga(\ca))$.
\end{remark}

\begin{remark}\label{rmk:recast}
For later use, we can reinterpret the previous proof in terms of the zigzag diagram \eqref{eqn:biqzigzag} which describes the morphism $u\colon\cv\to\cb$ in $\Hqe$. First of all, observe that the morphism $f=\Ind(u)$ in $\Hqe$ defined in \autoref{subsec:summarysetting} is represented by the zigzag diagram
\begin{equation}\label{eqn:biqzigzag2}
\xymatrix@C-10pt{
&\Perf{\widetilde\cB}\ar[dl]_-{\Ind(\wt\fF)}\ar[dr]_-{\Ind(\overline\fF)}&&\Perf{\cB''}\ar[dl]_-{\Ind(\fF'')}\ar[dr]_-{\Ind(\fF')}&&\Perf{\cB}\ar[dl]_-{\Ind(\fF)}\\
\Perf{\cV}&&\Perf{\overline\cB}&&\Perf{\cB'}&
}
\end{equation}
where all dg functors but $\Ind(\fF')$ are quasi-equivalences. By \autoref{prop:genDr}, up to taking functorial h-flat resolutions, we can assume without loss of generality that all dg categories in the above diagram are h-flat.

Denoting by $\cn$ the full dg subcategory of $\Perf{\cv}$ (which is an enhancement of $\Ka(\ca)$ by \autoref{CperfB}) defining an enhancement of $\Acya(\ca)$, our proof above shows that there are full dg subcategories $\wt\cs\subseteq\Perf{\wt\cb}$, $\overline\cs\subseteq\Perf{\overline\cb}$, $\cs''\subseteq\Perf{\cb''}$, $\cs'\subseteq\Perf{\cb'}$ and $\cs\subseteq\Perf{\cb}$ such that
\begin{enumerate}
\item[(1)] we have quasi-equivalences $\wt\cs\iso\cn$, $\wt\cs\iso\overline\cs$, $\cs''\iso\overline\cs$ and $\cs\iso\cs'$ which are induced by the corresponding dg functors in \eqref{eqn:biqzigzag2};
\item[(2)] $\cs'$ is the smallest full dg subcategory of $\Perf{\cb'}$ containing $\Ind(\fF')(\cs'')$;
\item[(3)] we have quasi-equivalences $\Perf{\wt\cb}/\wt\cs\iso\Perf{\cv}/\cn$, $\Perf{\wt\cb}/\wt\cs\iso\Perf{\overline\cb}/\overline\cs$, $\Perf{\cb''}/\cs''\iso\Perf{\overline\cb}/\overline\cs$, $\Perf{\cb''}/\cs''\iso\Perf{\cb'}/\cs'$ and $\Perf{\cb}/\cs\iso\Perf{\cb'}/\cs'$ induced by the dg functors \eqref{eqn:biqzigzag2} in view of \autoref{rmk:univpropquot}.
\end{enumerate}
In particular, $H^0(\cs')\iso H^0(\cs)\iso 0$ and we get the commutative diagram of zigzags
\begin{equation}\label{eqn:biqzigzag3}
\xymatrix@C-12pt{
&\Perf{\widetilde\cB}\ar[dl]\ar[dr]\ar[d]&&\Perf{\cB''}\ar[dl]\ar[dr]\ar[d]&&\Perf{\cB}\ar[dl]\ar[d]\\
\Perf{\cV}\ar[d]&\Perf{\wt\cb}/\wt\cs\ar[dl]\ar[dr]&\Perf{\overline\cB}\ar[d]&\Perf{\cb''}/\cs''\ar[dl]\ar[dr]&\Perf{\cB'}\ar[d]&\Perf{\cb}/\cs\ar[dl]\\
\Perf{\cV}/\cn&&\Perf{\overline\cB}/\overline\cs&&\Perf{\cB'}/\cs'&
}
\end{equation}
where the dg functors in the bottom zigzag are quasi-equivalences, the vertical maps are Drinfeld dg quotient functors such that $\Perf{\cb'}\to\Perf{\cb'}/\cs'$ and $\Perf{\cb}\to\Perf{\cb}/\cs$ are quasi-equivalences.
\end{remark}

\subsection{Applications}\label{subsec:appl1}

The first easy observation is that, by \autoref{rmk:equivKD} (and \autoref{rmk:uniqueoppeq} (i)), \autoref{thm:main1} (1) immediately implies the following.

\begin{cor}\label{cor:uniqKA}
If $\ca$ is an abelian category, then $\Ka(\ca)$ has a unique enhancement, for $?=b,+,-,\emptyset$.
\end{cor}

In this light, it would be interesting to investigate further the following problem.

\begin{qn}\label{qn:uniqAcy}
Let $\ca$ be an abelian category. Does $\Acya(\ca)$ have a unique enhancement, for $?=b,+,-,\emptyset$?
\end{qn}

Next, if $\cg$ is a Grothendieck category, then we learnt in \autoref{subsec:wellgentriacat} that, for $\alpha$ a sufficiently large cardinal, then $\D(\cg)^\alpha\iso\D(\cg^\alpha)$. Hence we get:

\begin{cor}\label{cor:Calpha}
If $\cg$ is a Grothendieck category and $\alpha$ is a sufficiently large cardinal, then $\D(\cg)^\alpha$ has a unique enhancement.
\end{cor}

When $\cg=\Qcoh(X)$ and $X$ is an algebraic stack, then this answers the second part of \cite[Question 4.7]{CS4} in the positive. Partial results are in \cite{A1}.

Continuing the discussion in the geometric setting, let us recall that while quasi-coherent sheaves are defined on any algebraic stack, coherent sheaves seem to be well defined only on locally noetherian algebraic stacks (see, for example, \cite[Chapter 15]{LMB}) or, of course, schemes. But when they are defined, they form an abelian category and thus we can deduce the following.

\begin{cor}\label{cor:cohquasicoh}
If $X$ is an algebraic stack, then $\Da(\Qcoh(X))$ has a unique enhancement, for $?=b,+,-,\emptyset$. If $X$ is a scheme or a locally noetherian algebraic stack, then the same is true for $\Da(\Coh(X))$. 
\end{cor}

Note that here, apart from \autoref{thm:main1} (1), we are implicitly using \autoref{prop:changeuniv}, as $\Qcoh(X)$ is not in general small in the given universe. This concludes the proof of the problem posed in \cite[Question 4.7]{CS4}.

We conclude this section with a question which we believe to be very natural as it concerns a potentially very interesting generalization of \autoref{thm:main1} (1).

\begin{qn}\label{qn:genthm1}
Let $\ca$ be an abelian category and let $\cb$ be a Serre subcategory of $\ca$. Does the full subcategory $\Da_{\cb}(\ca)$ of $\Da(\ca)$, consisting of complexes with cohomology in $\cb$, have a unique dg enhancement for $?=b,+,-,\emptyset$?
\end{qn}

A positive answer to \autoref{qn:genthm1} would immediately imply the uniqueness of dg enhancements for $\Dqa(X)$ in \autoref{thm:main2}. Unfortunately, the techniques developed in this paper seem insufficient to address \autoref{qn:genthm1}.

\subsection{Aside: the ``realization'' functor of Be{\u\i}linson, Bernstein and Deligne}\label{subsect:realization}

Recall that a \emph{t-structure} on a triangulated category $\ct$ is a pair
of full subcategories $(\ct^{\geq 0},\ct^{\leq 0})$ such that
\begin{enumerate}
	\item[(i)] If $X\in\Ob(\ct^{\leq 0})$ and $Y\in\Ob(\ct^{\geq 0})$, then $\Hom_\ct(X,\sh[-1]{Y})=0$;
	\item[(ii)] $\sh[1]{\ct^{\leq 0}}\subseteq\ct^{\leq 0}$ and $\sh[-1]{\ct^{\geq 0}}\subseteq\ct^{\geq 0}$ (we set $\ct^{\leq n}:=\sh[-n]{\ct^{\leq 0}}$ and $\ct^{\geq n}:=\sh[-n]{\ct^{\geq 0}}$);
	\item[(iii)] For any object $X\in\Ob(\ct)$ there is a distinguished triangle
	$$X^{\leq 0}\lto X\lto X^{\geq 1},$$ where $X^{\leq 0}\in\Ob(\ct^{\leq 0})$ and $X^{\geq 1}\in\Ob(\ct^{\geq 1})$.
\end{enumerate}
The \emph{heart} of a t-structure $(\ct^{\geq 0},\ct^{\leq 0})$ on $\ct$ is the abelian category $\ct^\heartsuit:=\ct^{\geq 0}\cap\ct^{\leq 0}$. Given an object $X\in\ct$ and an integer $i$, we can use the distinguished triangle in (iii) to define the object $X^{\geq i}\in\Ob(\ct^{\geq i})$ as $\sh[1-i]{(\sh[i-1]{X}^{\ge1})}$. A t-structure on $\ct$ is \emph{non-degenerate} if the intersections of all the $\ct^{\leq n}$ and of all the $\ct^{\geq n}$ are trivial.

In \cite[Proposition~3.1.10]{BeiBerDel82}, Be{\u\i}linson, Bernstein and Deligne proved a very interesting result asserting that if $\ct$ is a triangulated category with a t-structure with heart 
$\ct^\heartsuit$, under a hypothesis on the existence of filtered derived categories, 
the natural inclusion $\ct^\heartsuit\subset\ct$ extends to an exact
functor
\[
\real\colon\D^b(\ct^\heartsuit)\lto\ct
\]
respecting t-structures. The original
proof was relatively complicated, and there has been some literature on this
since. But the point here is that the result becomes straightforward
by our techniques, which allow for generalizations to $\Da(\ct^\heartsuit)$.

As above let $\ct$ be a triangulated category with a t-structure, and assume
we are given an enhancement $(\cc,\fE)$ of $\ct$. Let $\Va(\ct^\heartsuit)$ be as in 
the opening paragraphs of \autoref{subsec:cat} and let $\cv^?(\ct^\heartsuit)$ be its enhancement defined at the very beginning of \autoref{subsec:defB}. And finally assume that there is
a morphism $f\in\Hom_\Hqe(\cv^?(\ct^\heartsuit),\cc)$ such that the functor $\fF:=\fE\comp H^0(f)\colon\Va(\ct^\heartsuit)\to\ct$ satisfies the following properties, for every countable collection $\{T^i\st i\in\ZZ\}$ of objects
of $\ct^\heartsuit$.
\begin{enumerate}
\item[(1)] $\fF$ takes a finite sum $\oplus_{i=m}^n T^i[-i]$ in $\Va(\ct^\heartsuit)$
to the object $\oplus_{i=m}^n T^i[-i]\in\ct$, and on morphisms it is the obvious functor;
\item[(2)] If $?=-,\emptyset$, then $\fF\big(\oplus_{i=-\infty}^n T^i[-i]\big)\in\ct^{\leq n}$;
\item[(3)] If $?=+,\emptyset$, then $\fF\big(\oplus_{i=n}^\infty T^i[-i]\big)\in\ct^{\geq n}$.
\end{enumerate}

Consider the morphism $\widetilde f:=\Ind(f)\rest{\Perf{\cv^?(\ct^\heartsuit)}}\in\Hom_{\Hqe}(\Perf{\cv^?(\ct^\heartsuit)},\Perf{\cc})$. Assuming that $\ct$ is idempotent complete, it induces an exact functor
\[
\widetilde\fF:=\fE\comp H^0(\widetilde f)\colon\Ka(\ct^\heartsuit)\lto\ct,
\]
and it is an easy exercise to show that the $\ct^\heartsuit$-cohomology of $\widetilde\fF(E)$ vanishes for every $E\in\Ob(\Acya(\ct^\heartsuit))$. Assuming also that the t-structure is nondegenerate, we deduce that  $\widetilde\fF(\Acya(\ct^\heartsuit))=0$, and the functor
$\widetilde\fF$ must factor through
\[
\real\colon\Da(\ct^\heartsuit)\lto\ct.
\]

Thus the problem reduces to finding a morphism $f\in\Hom_\Hqe(\cv^?(\ct^\heartsuit),\cc)$ with the required properties, and \autoref{sect:enB} is all about methods to do this. The special case $?=b$ is trivial, since the limit arguments disappear.
To the best of our knowledge there is only one other article
in the literature which shows how to construct the realization functor on 
unbounded or half-bounded derived categories, namely
Virili~\cite[Sections 5 and 6]{Virili18}. But the hypotheses Virili places 
on the categories
$\ct$ and $\ct^\heartsuit$ are much more restrictive than ours.

\section{The unseparated and completed derived categories}\label{sect:big}

In this section we prove \autoref{thm:main1} (2). In \autoref{subsect:sepdercat} we prove the uniqueness of enhancements for the unseparated derived category of a Grothendieck category, while the proof of the uniqueness of enhancements
for the completed derived category of a Grothendieck category is carried out in \autoref{subsec:completeddercat}. This is preceded by a discussion about basic properties of completed derived categories in \autoref{subsec:completedderivedcategories}. We end this section with some speculations about uniqueness of dg enhancements for admissible subcategories.

\subsection{The unseparated derived category}\label{subsect:sepdercat}

The uniqueness of dg enhancements for $\Dc(\cg)$, for $\cg$ a Grothendieck category, can be deduced from a general criterion proved in \cite{CSUn2}. 

To state it precisely, let us recall the following general definition.

\begin{definition}\label{abc}
	Let $\ca$ be a small $\kk$-linear category considered as a dg category concentrated in degree $0$ and let $\ct$ be a triangulated category with small coproducts. An exact functor $\fF\colon\dgD(\ca)\to\ct$ is \emph{right vanishing} if it preserves small coproducts and there exists a full subcategory $\cR$ of $\ct$ with the following properties:
	\begin{enumerate}
		\item[(R1)]\label{Rcoprod} $\cR$ is closed under small coproducts;
		\item[(R2)]\label{Rext} $\cR$ is closed under extensions (meaning that, if $X\to Y\to Z$ is a distinguished triangle in $\ct$ with $X,Z\in\cR$, then $Y\in\cR$, as well); 
		\item[(R3)]\label{Rrep} $\sh[k]{\fF(\Yon(A))}\in\cR$ for every $A\in\ca$ and every integer $k<0$;
		\item[(R4)]\label{Rort} $\Hom_\ct\left(\fF(\Yon(A)),R\right)=0$ for every $A\in\ca$ and every $R\in\cR$.
	\end{enumerate}
\end{definition}

The following is the revised version of \cite[Theorem C]{CSUn1} which appeared in the preprint version \cite{CSUn2}. It extends \cite[Theorem 2.7]{LO}. The key notion is well generation for triangulated categories (see \autoref{subsec:wellgentriacat}).

\begin{thm}\label{thm:crit1}
	Let $\ca$ be a small $\kk$-linear category considered as a dg category concentrated in degree $0$ and let $\cl$ be a localizing subcategory of $\dgD(\ca)$ such that:
	\begin{itemize}
		\item[{\rm (a)}] The quotient $\dgD(\ca)/\cl$ is a well generated triangulated category;
		\item[{\rm (b)}] The quotient functor $\fQ\colon\dgD(\ca)\to\dgD(\ca)/\cl$ is right vanishing.
	\end{itemize}
	Then $\dgD(\ca)/\cl$ has a unique enhancement.
\end{thm}

Let us move to the first part of \autoref{thm:main1} (2). By \autoref{thm:sepderwellgen}, $\Dc(\cg)$ is well generated and so, in order to apply \autoref{thm:crit1}, we just need to show that $\Dc(\cg)$ can be written as a quotient $\dgD(\ca)/\cl$ with $\ca$ concentrated in degree $0$ and that the quotient functor is right vanishing.

For the given Grothendieck category $\cg$, take a sufficiently large cardinal $\alpha$ such that $\cg^\alpha$ is abelian (see \autoref{thm:alphacompobj}). 

\begin{remark}\label{rmk:smallGroth}
Even though, by the discussion in \autoref{subsect:dgenuniq} the choice of the universe can be made harmless, it is easy to see that in this case $\cg^\alpha$ is small, since $\cg$ is generated by one object.
\end{remark}

By the discussion at the beginning of \cite[Section 5]{K0}, we have an equivalence
\begin{equation}\label{eqn:eq1}
\fF\colon\cg\to\Lex_{\alpha}((\cg^\alpha)\opp,\Mod{\kk})
\end{equation}
given by the assignment $G\mapsto\Hom_\cg(\farg,G)|_{\cg^\alpha}$ (see \autoref{ex:Grothcat} (ii)). On the other hand, by \cite[Proposition 5.4]{K0}, the natural inclusion
$\Lex_{\alpha}((\cg^\alpha)\opp,\Mod{\kk})\mono\Mod{\cg^\alpha}$
has a left adjoint
\begin{equation}\label{eqn:eq2}
\fQ'\colon\Mod{\cg^\alpha}\to\Lex_{\alpha}((\cg^\alpha)\opp,\Mod{\kk})
\end{equation}
which is an exact functor sending $\Yon[\cg^\alpha](G)$ to itself. Hence, by \eqref{eqn:eq1} and \eqref{eqn:eq2}, we have
\begin{equation}\label{eqn:eq3}
\fF^{-1}(\fQ'(\Yon[\cg^\alpha](G)))\iso G,
\end{equation}
for all $G\in\cg^\alpha$.

On the other hand, consider the commutative diagram of natural inclusions
\[
\xymatrix{
\K(\cg)\ar@{^{(}->}[rr]^-{\fI_1}&&\K(\Mod{\cg^\alpha})\\
\Dc(\cg)\ar@{^{(}->}[rr]^-{\fI_4})\ar@{^{(}->}[u]^-{\fI_2}&&\D(\Mod{\cg^\alpha}).\ar@{^{(}->}[u]_-{\fI_3}
}
\]
By passing to the (left) adjoints we get the corresponding commutative diagram
\begin{equation*}\label{eqn:comm1}
	\xymatrix{
		\K(\Mod{\cg^\alpha})\ar[rr]^-{\fQ_1}\ar[d]_-{\fQ_3}&&\K(\cg)\ar[d]^-{\fQ_2}\\
		\D(\Mod{\cg^\alpha})\ar[rr]^-{\fQ_4}&&\Dc(\cg).
	}
\end{equation*}
If $\fK\colon\cg^\alpha\to\K(\cg)$ is the natural functor, it can be completed to the following diagram:
\begin{equation}\label{eqn:comm2}
\xymatrix{
\cg^\alpha\ar@/^2.0pc/[drrrr]^-{\fK}\ar[drr]^-{\Yon[\cg^\alpha]}\ar@/_2.0pc/[ddrr]_-{\Yon[\cg^\alpha]}&&&&\\
	&&\K(\Mod{\cg^\alpha})\ar[rr]^-{\fQ_1}\ar[d]_-{\fQ_3}&&\K(\cg)\ar[d]^-{\fQ_2}\\
	&&\D(\Mod{\cg^\alpha})\ar[rr]^-{\fQ_4}&&\Dc(\cg).
}
\end{equation}
It is easy to see that the diagram is commutative (up to isomorphism). Indeed, the lower triangle is commutative by definition while the upper one, involving $\fK$ commutes by \eqref{eqn:eq3}.

Let us now observe that the triangulated category $\Dc(\cg)$ fits in the framework of \autoref{thm:crit1}. Indeed, by \cite[Theorem 5.12]{K0}, the functor $\fQ_4$ is a localization with localizing subcategory $\cs\subseteq\D(\Mod{\cg^\alpha})$. Thus $\fQ_4$ induces an equivalence
\[
\fF_1\colon\D(\Mod{\cg^\alpha})/\cs\to\Dc(\cg)
\]
Moreover, by the second main result in \cite{K0}, the triangulated category $\Dc(\cg)$ (and thus the localization $\D(\Mod{\cg^\alpha})/\cs$) is well generated.

Therefore, in order to apply \autoref{thm:crit1} and conclude that $\Dc(\cg)$ has a unique enhancement, we just need to show that the quotient functor
\[
\overline{\fQ}\colon\D(\Mod{\cg^\alpha})\to\D(\Mod{\cg^\alpha})/\cs
\]
is right vanishing. The set $\cR$ of objects of $\D(\Mod{\cg^\alpha})/\cs$ with respect to which this is checked consists of the preimage under $\fF_1$ of the objects which are isomorphic to complexes of injective objects in $\cg$ concentrated in positive degrees.

Now, it is clear that (R1)--(R3) are satisfied by definition. To prove (R4), consider any $G\in\cg^\alpha$ and $R\in\cR$. Then
\begin{multline*}
\Hom_{\D(\Mod{\cg^\alpha})/\cs}\left(\overline{\fQ}(\Yon[\cg^\alpha](G)),R\right)\iso\Hom_{\Dc(\cg)}\left(\fF_1(\overline{\fQ}(\Yon[\cg^\alpha](G))),\fF_1(R)\right)\\
\iso\Hom_{\Dc(\cg)}\left(\fQ_4(\Yon[\cg^\alpha](G)),\fF_1(R)\right)\iso\Hom_{\Dc(\cg)}\left(\fQ_2(\fK(G)),\fF_1(R)\right)
\end{multline*}
where the second isomorphism is by the definition of $\fF_1$ and the third one is due to the commutativity of \eqref{eqn:comm2}. But since the functor $\fQ_2$ consists of taking injective resolutions, the complex $\fQ_2(\fK(G))$ has trivial cohomology in positive degrees. Hence
\[
\Hom_{\Dc(\cg)}\left(\fQ_2(\fK(G)),\fF_1(R)\right)=0
\]
by the definition of $\cR$.

\subsection{The completed derived category: basic properties}\label{subsec:completedderivedcategories}

In this section we discuss enhancements of the completed derived category. This is interesting to analyze, partly because this is a triangulated category whose definition starts from a dg category---which, of course, turns out to be one of its dg enhancements. In the presentation we follow \cite[Section 1.2.1]{LurHA}. We will not introduce the language of $\infty$-categories as we will only be using it superficially in this paper.

Let $\cg$ be a Grothendieck category and fix the enhancement of $\D(\cg)$ given by the dg category $\Ddg(\cg)$ which is the Drinfeld quotient $\dgC(\ca)/\dgAcy(\ca)$.

For $i\in\ZZ$, we denote by $\Ddg(\cg)^{\geq i}$ the full dg subcategory of $\Ddg(\cg)$ such that $H^0(\Ddg(\cg)^{\geq i})=\D(\cg)^{\geq i}$. Once we interpret them as $\infty$-categories, one obtains functors
\[
\tau^{\geq i}\colon\Ddg(\cg)\lto\Ddg(\cg)^{\geq i}.
\]
From this one produces a sequence of $\infty$-categories and functors
\[
\cdots\mor{\tau^{\geq -2}}\Ddg(\cg)^{\geq -2}\mor{\tau^{\geq -1}}\Ddg(\cg)^{\geq -1}\mor{\tau^{\geq 0}}\Ddg(\cg)^{\geq 0}\mor{\tau^{\geq 1}}\Ddg(\cg)^{\geq 1}\mor{\tau^{\geq 2}}\cdots
\]
Denote by $\Dhdg(\cg)$ its homotopy limit. 

By \cite[Proposition 1.2.1.17]{LurHA}, $\Dhdg(\cg)$ is a stable $\infty$-category and thus, by the main result in \cite{Coh}, it is naturally quasi-equivalent to a pretriangulated dg category. So, without loss of generality, we can assume in this paper that $\Dhdg(\cg)$ is actually a pretriangulated dg category, thus motivating the notation.

\begin{remark}\label{rmk:compltrunc}
	By definition, $\Dhdg(\cg)$ can alternatively be defined as the homotopy limit of the following sequence
	\begin{equation}\label{eqn:leftcomplD+}
		\cdots\mor{\tau^{\geq -2}}\Ddg(\cg)^{\geq -2}\mor{\tau^{\geq -1}}\Ddg(\cg)^{\geq -1}\mor{\tau^{\geq 0}}\Ddg(\cg)^{\geq 0}
	\end{equation}
	More precisely, one defines the left completion of $\Du(\cg)$ by using \eqref{eqn:leftcomplD+}. By \cite[Remark 1.2.1.18]{LurHA} such a completion is naturally quasi-equivalent to $\Dhdg(\cg)$. It follows that the objects of $\Dhdg(\cg)$ can be identified with homotopy limits of objects in the full dg subcategories $\Ddg(\cg)^{\geq i}$, for $i\leq 0$.
\end{remark}

\begin{definition}\label{def:completeddercat}
	The \emph{(left) completed derived category} of a Grothendieck category $\cg$ is the triangulated category $\Dh(\cg):=H^0(\Dhdg(\cg))$.
\end{definition}

By construction $\Dh(\cg)$ has a dg enhancement. The other properties of such a triangulated category which are relevant in this paper are summarized in the following proposition (see \autoref{subsect:realization} for the definition and basic properties of t-structures). 

\begin{prop}[\cite{LurHA}, Proposition 1.2.1.17]\label{prop:compldercat}
	Let $\cg$ be a Grothendieck category.
	\begin{enumerate}
		\item[{\rm (1)}] There is a natural exact functor $\D(\cg)\to\Dh(\cg)$;
		\item[{\rm (2)}] The category $\Dh(\cg)$ has a natural t-structure and the functor in {\rm (1)} identifies its heart with $\cg$ and it induces an equivalence $\D(\cg)^{\geq 0}\iso\Dh(\cg)^{\geq 0}$;
		\item[{\rm (3)}] The category $\Dh(\cg)$ is left complete (i.e.\ for any $X\in\Dh(\cg)$ the natural morphism $X\to\holim X^{\geq i}$ is an isomorphism).
	\end{enumerate}
\end{prop}

Part (3) in the proposition above, which essentially follows from the definition, motivates the choice of the name for $\Dh(\cg)$. As a consequence of (2) we have a natural equivalence
\begin{equation}\label{eqn:D+}
	\Dh(\cg)^+\iso\Du(\cg),
\end{equation}
where $\Dh(\cg)^+$ denotes the full subcategory of $\Dh(\cg)$ consisting of all the objects $X\in\Ob(\Dh(\cg))$ such that $X=X^{\geq i}$, for $i\ll 0$. We denote by $\Dhdg(\cg)^+$ the full dg subcategory of $\Dhdg(\cg)$ such that $H^0(\Dhdg(\cg)^+)=\Dh(\cg)^+$.

\begin{remark}\label{rmk:nonleftcomplete}
It was proved in \cite{NNonLeftCompl} that there exist Grothendieck categories $\cg$ such that $\D(\cg)$ is not left complete in the sense of \autoref{prop:compldercat} (3). The category $\Dh(\cg)$ provides a natural way to complete $\D(\cg)$ in view of \autoref{prop:compldercat} (1).
\end{remark}

\subsection{The completed derived category: uniqueness of enhancements}\label{subsec:completeddercat}

Let us prove that $\Dh(\cg)$ has a unique dg enhancement, when $\cg$ is a Grothendieck category. To do this we use \autoref{rmk:uniqueoppeq} (ii) and show that $\Dh(\cg){\opp}$ has a unique dg enhancement. Hence we assume that there is an exact equivalence
\[
\fF\colon\Dh(\cg){\opp}\equiva H^0(\cc),
\]
where $\cc$ is a pretriangulated dg category.

\begin{remark}\label{rmk:opptstr}
By \autoref{prop:compldercat} (2), $\Dh(\cg)$ has a natural t-structure. Hence $\Dh(\cg){\opp}$ and $H^0(\cc)$ are endowed with a t-structure $(H^0(\cc)^{\geq 0},H^0(\cc)^{\leq 0})$ as well. Of course $\fF$ yields equivalences $H^0(\cc)^{\geq 0}\iso\Dh(\cg)^{\leq 0}$ and $H^0(\cc)^{\leq 0}\iso\Dh(\cg)^{\geq 0}$. Moreover, as by \autoref{prop:compldercat} (3) $\Dh(\cg)$ is left complete, by passing to the opposite category, we have that the natural morphism
\[
\hocolim X^{\leq i}\lto X
\]
is an isomorphism, for any $X\in\Ob(H^0(\cc))$. More generally, in view of \autoref{rmk:compltrunc}, the objects of $\Dh(\cg){\opp}$ (resp.\ $H^0(\cc)$) are identified with all the homotopy colimits of sequences of objects in $(\Dh(\cg){\opp})^{\leq i}$ (resp.\ $H^0(\cc)^{\leq i}$) with eventually stable cohomology.
\end{remark}

Clearly, there exists a full pretriangulated dg subcategory $\cc^+\subseteq\cc$ such that $\fF^+:=\fF\rest{(\Dh(\cg)^+){\opp}}$ is an exact equivalence
\[
\fF^+\colon(\Dh(\cg)^+){\opp}\equiva H^0(\cc^+).
\]
In the following let $\cd$ be either $\Dhdg(\cg){\opp}$ or $\cc$ and let $\cd^+$ be either $(\Dhdg(\cg)^+){\opp}$ or $\cc^+$. We denote by $\RYon\colon\cd\to\hproj{\cd^+}$ the morphism in $\Hqe$ defined as the composition
\[
\cd\mor{\dgYon[\cd]}\dgMod{\cd}\mor{\Res(\fI)}\dgMod{\cd^+}\lmor{\fP}\dgMod{\cd^+}/\dgAc{\cd^+}\mor{(\fP\comp\fJ)^{-1}}\hproj{\cd^+},
\]
where $\fI\colon\cd^+\mono\cd$ and $\fJ\colon\hproj{\cd^+}\mono\dgMod{\cd^+}$ are the natural inclusions, and $\fP$ is the quotient dg functor. Let moreover $\cd^{\leq i}$ be the full dg subcategory of $\cd$ such that $H^0(\cd^{\leq i})=H^0(\cd)^{\leq i}$ (we implicitly refer to the t-structures in \autoref{rmk:opptstr}).

\begin{lem}\label{lem:restrquasifullyfaith}
The essential image of the exact functor $H^0(\RYon)$ consists of objects isomorphic to homotopy colimits in $H^0(\hproj{\cd^+})$ of $\dgYon[\cd^+](X_i)$, where $X_i\in\Ob(\cd^{\leq i})$, for $i\geq 0$. Moreover, $\RYon$ is quasi-fully faithful.
\end{lem}

\begin{proof}
Let $X\in\Ob(\cd)$. By \autoref{rmk:opptstr}, we have a (closed, degree $0$) morphism in $\dgMod{\cd^+}$
\[
\varphi_X\colon\hocolim\Hom_{\cd}\left(\farg,X^{\leq i}\right)\lto\Hom_{\cd}\left(\farg,X\right),
\]
which we claim to be a quasi-isomorphism. In order to prove this, it is enough to show that the induced morphism $\widetilde\varphi_X(Y)\colon\hocolim\Hom_{H^0(\cd)}\left(Y,X^{\leq i}\right)\to\Hom_{H^0(\cd)}\left(Y,X\right)$ is an isomorphism, for all $Y\in\Ob(H^0(\cd^+))$.  But in the t-structure on $H^0(\cd)$ from \autoref{rmk:opptstr}, $Y$ has bounded above non-trivial cohomologies. That means that the natural maps
\[
\Hom_{H^0(\cd)}\left(\farg,X^{\leq i}\right)\lto\Hom_{H^0(\cd)}\left(\farg,X^{\leq i+1}\right)
\]
are isomorphisms for $i$ sufficiently large. Hence we have natural isomorphisms
\[
\hocolim\Hom_{H^0(\cd)}\left(Y,X^{\leq i}\right)\iso\Hom_{H^0(\cd)}\left(Y,\hocolim X^{\leq i}\right)\iso\Hom_{H^0(\cd)}\left(Y,X\right)
\]
and the composition of these maps is precisely $\widetilde\varphi_X(Y)$. This shows that any object in the essential image of $H^0(\RYon)$ is isomorphic to the homotopy colimit of objects in $\cd^{\leq i}$, for $i\geq 0$. To show the other inclusion, note that the same argument proves that the dg module $\hocolim\Hom_{\cd}(\farg,X_i)$, where $X_i\in\Ob(\cd^{\leq i})$, is quasi-isomorphic to $\Hom_\cd(\farg,X)$, where $X:=\hocolim X_i$ is in $\cd$ by \autoref{rmk:opptstr}.

Let us now prove that $\RYon$ is quasi-fully faithful. For $X_1,X_2\in\Ob(H^0(\cd))$, we have the quasi-isomorphisms
\begin{multline*}
\Hom_{\hproj{\cd^+}}\left(\RYon(X_1),\RYon(X_2)\right)\iso \Hom_{\hproj{\cd^+}}\left(\hocolim\Hom_{\cd}\left(\farg,X_1^{\leq i}\right),\RYon(X_2)\right)\\
\iso\holim\Hom_{\hproj{\cd^+}}\left(\Hom_{\cd}\left(\farg,X_1^{\leq i}\right),\RYon(X_2)\right)\iso\holim\Hom_{\cd}\left(X_1^{\leq i},X_2\right)\\
\iso\Hom_{\cd}\left(\hocolim X_1^{\leq i},X_2\right)\iso\Hom_{\cd}\left(X_1,X_2\right),
\end{multline*}
where the first quasi-isomorphism is a consequence of the fact that $\varphi_{X_1}$ is a quasi-isomorphism. The third one follows from the Yoneda lemma while the last one is again \autoref{rmk:opptstr}. For the second and the fourth quasi-isomorphisms we use the fact that the $\Hom$-functor turns colimits in the first argument into limits. Therefore, $\RYon$ is quasi-fully faithful.  
\end{proof}

By \eqref{eqn:D+}, \autoref{thm:main1} (1) and again \autoref{rmk:uniqueoppeq} (ii), there is a morphism
\[
g^+\in\Hom_\Hqe\left((\Dhdg(\cg)^+){\opp},\cc^+\right)
\]
such that
\begin{enumerate}
\item[(a)] $\fG^+:=H^0(g^+)$ is an exact equivalence;
\item[(b)] $\fG^+(X)\iso\fF^+(X)$, for all $X\in\Ob((\Dh(\cg)^+){\opp})$ (see \autoref{rmk:semistrong}).
\end{enumerate}
The morphism $g:=\Ind(g^+)\in\Hom_\Hqe(\hproj{(\Dhdg(\cg)^+){\opp}},\hproj{\cc^+})$ is such that $\fG:=H^0(g)$ is an exact equivalence by (a). Thus we get the diagram
\begin{equation}\label{eqn:commHqe}
\xymatrix{
\Dhdg(\cg)\opp\ar[dr]_-{\RYon[\Dhdg(\cg)\opp]\;\;} & (\Dhdg(\cg)^+){\opp}\ar@{_{(}->}[l]\ar[rr]^-{g^+}\ar@{^{(}->}[d]^-{\dgYon[(\Dhdg(\cg)^+){\opp}]} & & \cc^+\ar@{^{(}->}[r]\ar@{^{(}->}[d]_-{\dgYon[\cc^+]} & \cc\ar[dl]^-{\RYon[\cc]} \\
 & \hproj{(\Dhdg(\cg)^+){\opp}}\ar[rr]^-{g} & & \hproj{\cc^+} &
}
\end{equation}
which is commutative in $\Hqe$. Denote by $g'$ the composition $g\comp\RYon[\Dhdg(\cg)\opp]$. It is clear from \autoref{lem:restrquasifullyfaith} that $H^0(g')$ is fully faithful.

Let us prove that the essential image of $H^0(g')$ coincides with the essential image of $H^0(\RYon[\cc])$. Let $X\in\Ob(\Dh(\cg))$. Then, by \autoref{rmk:opptstr}, \eqref{eqn:commHqe} and the fact that $\fG$ is an exact equivalence we get
\begin{multline*}
H^0(g')(X)\iso\fG\left(H^0\left(\RYon[(\Dhdg(\cg)^+){\opp}]\right)(\hocolim X^{\leq i})\right)\iso\fG\left(\hocolim H^0\left(\RYon[(\Dhdg(\cg)^+){\opp}]\right)(X^{\leq i})\right)\\
\iso\hocolim\fG\left(H^0\left(\RYon[(\Dhdg(\cg)^+){\opp}]\right)(X^{\leq i})\right)\iso\hocolim H^0\left(\dgYon[\cc^+]\right)\left(\fG^+(X^{\leq i})\right).
\end{multline*}
By (b), $\fG^+(X^{\leq i})\iso\fF^+(X^{\leq i})\in\Ob(\cc^{\leq i})$, for all $i\geq 0$. By \autoref{lem:restrquasifullyfaith}, we get that $H^0(g')(X)$ is in the essential image of $H^0(\RYon[\cc])$. The same argument shows the other inclusion.

In conclusion, let $g''$ denote the inverse of $\RYon[\cc]$ in $\Hqe$, when we interpret $\RYon[\cc]$ as a morphism between $\cc$ and its essential image. The above argument shows that $g''\comp g'$ is an isomorphism in $\Hom_\Hqe\left(\Dh(\cg){\opp},\cc\right)$ as we want.

\subsection{Two applications: a shortcut and a negative result}\label{subsect:applicationcounter}

Let us first note that the methods used in the proof in \autoref{subsect:sepdercat} (or rather their specialization to the subcategory of compact objects) can be use to give a straightforward proof of the uniqueness of dg enhancements of $\Db(\ca)$, where $\ca$ is a small abelian category.

Indeed, given such an $\ca$, one takes the Grothendieck category $\cg:=\IndC(\ca)$ (see \autoref{ex:Grothcat} (iii)). In that case, by \autoref{ex:sepdercomp} and \cite[Proposition 4.7]{K0}, we have a sequence of equivalences
\begin{equation}\label{eqn:Db1}
\Db(\ca)\iso\K(\Inj(\IndC(\ca)))^c\iso\ic{\left(\K(\Inj(\ca))^c/\cl'\right)},
\end{equation}
where, by abuse of notation, we write $\Inj(\ca)$ instead of $\Inj(\Mod{\ca})$, and where $\cl'\subseteq\K(\Inj(\ca))^c$ is a full triangulated subcategory. Let $\cl$ be the smallest localizing subcategory of $\K(\Inj(\ca))$ containing $\cl'$. By the main result in \cite{N1}, we deduce from this the exact equivalence
\begin{equation}\label{eqn:Db2}
\Db(\ca)\iso\left(\K(\Inj(\ca))/\cl\right)^c.
\end{equation}

Since $\K(\Inj(\ca))\iso\D(\Mod{\ca})\iso\dgD(\ca)$ (see, for example, \cite[Lemma 4.8]{K0}), the uniqueness of dg enhancements for $\Db(\ca)$ is a consequence of the following result.

\begin{thm}[\cite{LO}, Theorem 2]\label{thm:LO}
	Let $\ca$ be a small category and let $\cl$ be a localizing subcategory of $\dgD(\ca)$ such that:
	\begin{itemize}
		\item[{\rm (a)}] $\cl^c=\cl\cap\dgD(\ca)^c$ and $\cl^c$ satisfies {\rm (G1)} in $\cl$;
		\item[{\rm (b)}] $\Hom_{\dgD(\ca)/\cl}\left(\fQ(\Yon(A_1)),\sh[i]{\fQ(\Yon(A_2))}\right)=0$, for all $A_1,A_2\in\ca$ and all integers $i<0$.
	\end{itemize}
	Then $(\dgD(\ca)/\cl)^c$ has a unique enhancement.
\end{thm}

Indeed, in our situation, $\cl^c=\cl'$ and assumption (b) is clear because, given $A\in\ca$, the object $\Yon(A)\in\dgD(\ca)$ is mapped to $A$ by the composition of the equivalences \eqref{eqn:Db1} and \eqref{eqn:Db2} (see the discussion in \cite[Section 6.1]{CSUn1}).

\medskip

Let us now pass to the negative examples of this section. The relation between the uniqueness of dg enhancements of triangulated categories and of their admissible subcategories turns out to be a delicate and intriguing problem. In \cite{CS4} the following natural question was raised:

\begin{qn}\label{qn:uniqloc}
Let $\cs$ be a full triangulated subcategory of a triangulated category $\ct$. Does $\cs$ have unique enhancement if $\ct$ does? What if $\cs$ is a localizing (in case $\ct$ has small coproducts) or an admissible subcategory of $\ct$?
\end{qn}

Recall that a full triangulated subcategory $\cs$ of a triangulated category $\ct$ is \emph{admissible} if the inclusion functor $\cs\mono\ct$ has left and right adjoints.

As an application of \autoref{thm:main1} (2) we deduce the following result.

\begin{cor}\label{cor:negative}
Assume that $\kk=\ZZ$. Then there exists a triangulated category $\ct$ with small coproducts and a localizing and admissible subcategory $\cs$ of $\ct$ such that $\ct$ has a unique enhancement, while $\cs$ does not.
\end{cor}

\begin{proof}
For $i=1,2$, consider the triangulated categories $\ct_i$ with localizing and admissible subcategories $\cs_i\mono\ct_i$ in \autoref{ex:recoll}. For the convenience of the reader, let us recall that, in the example, we consider rings $R_1:=\ZZ/p^2\ZZ$ and $R_2:=\FF_p[\ep]$ (where $\ep^2=0$) and, for $i=1,2$, we set $\ct_i:=\Dc(\Mod{R_i})$. Moreover, we denoted by $\cs_i$ the full subcategory of $\ct_i$ consisting of acyclic complexes. By \autoref{thm:main1}, $\ct_i$ has a unique dg enhancement, whereas $\cs_i$ does not, as it is explained in \cite[Section 3.3]{CS4}. Indeed, while $\cs_1$ and $\cs_2$ are both equivalent to the category $\Mod{\FF_p}$ endowed with the triangulated structure defined by $\sh{}=\id$ (and, necessarily, distinguished triangles given by triangles inducing long exact sequences), it follows from \cite{DS,Sc} that the natural enhancements of $\cs_1$ and $\cs_2$ are not isomorphic in $\Hqe$.
\end{proof}

\begin{remark}
(i) Keeping the notation of \autoref{ex:recoll}, the triangulated categories $\cs_1$ and $\cs_2$ are $\FF_p$-linear, and the same is true for the natural enhancement of $\cs_2$, but not for that of $\cs_1$. Thus it remains an open problem whether a similar counterexample can be found when $\kk$ is a field.

(ii) In \autoref{ex:recoll}, the fact that $\D(\Mod{R_i})$ is compactly generated implies that the right adjoint to the quotient functor $\K(\Inj(\Mod{R_i}))\to \D(\Mod{R_i})$ has itself a right adjoint. Hence, $\D(\Mod{R_i})$ is an admissible localizing subcategory of $\K(\Inj(\Mod{R_i}))$ and, moreover, both triangulated categories $\K(\Inj(\Mod{R_i}))$ and $\D(\Mod{R_i})$ have unique enhancements (by \autoref{thm:main1}). This, together with \autoref{cor:negative}, shows that the same triangulated category $\K(\Inj(\Mod{R_i}))$ may have a unique enhancement and, at the same time, contain admissible (localizing) subcategories answering both positively (e.g.\ $\D(\Mod{R_i})$) and negatively (e.g.\ $\cs_i$) to \autoref{qn:uniqloc}. This clarifies that a complete answer to the question above cannot simply rely on the properties of the ambient category.
\end{remark}

\section{Homotopy pullbacks and enhancements}\label{sect:uniquepullbacks}

In this part we investigate the relation of dg enhancement with homotopy pullbacks; this will provide us with the formal machinery needed to prove \autoref{thm:main2}.
The framework we will be working with is set up in some generality in 
\autoref{subsec:settingpullbacks}, with the examples hinting at 
the geometric relevance of the formalism. 
In \autoref{subsec:usefcrit} we state and prove the main ingredient in the 
proof of \autoref{thm:main2}, while \autoref{subsec:geomapplpb} 
provides technical refinements.

\subsection{The setting}\label{subsec:settingpullbacks}

Let us summarize the abstract setting where we aim to state and prove the general criterion for the uniqueness of dg enhancements in presence of homotopy pullbacks.

\begin{setup}\label{setting1}
Let $\cc$ be a pretriangulated dg category such that $H^0(\cc)$ is idempotent complete and let $\cd_i\subseteq\cc$ be a full dg subcategory, for $i=1,2$ such that
\begin{enumerate}
\item[{\rm (i)}] $\cd_i$ is closed under shifts, for $i=1,2$;
\item[{\rm (ii)}] $\Hom_{H^0(\cc)}\left(H^0(\cd_i),H^0(\cd_j)\right)=0$, for $i\neq j\in\{1,2\}$.
\end{enumerate}
Assume further that there is a commutative diagram in $\dgCat$
\[
\xymatrix{
\cc\ar[r]^-{\fQ_1}\ar[d]_-{\fQ_2}&\cc_{\cd_1}\ar[d]^-{\overline\fQ_1}\\
\cc_{\cd_2}\ar[r]^-{\overline\fQ_2}&\cc_{\cd_1,\cd_2}
}
\]
of pretriangulated and idempotent complete dg categories and such that
\begin{enumerate}
\item[{\rm (iii)}] $H^0(\fQ_i)$ is the idempotent completion of the Verdier quotient functor that sends to zero the thick subcategory $\overline\cd_i\subseteq H^0(\cc)$ generated by $H^0(\cd_i)$, for $i=1,2$;
\item[{\rm (iv)}] $H^0(\overline\fQ_1)\comp H^0(\fQ_1)(=H^0(\overline\fQ_2)\comp H^0(\fQ_2))$ is the idempotent completion of the Verdier quotient functor that sends to zero the thick subcategory $\overline\cd_{1,2}\subseteq H^0(\cc)$ generated by $H^0(\cd_1)\cup H^0(\cd_2)$.
\end{enumerate}
\end{setup}

For the non-expert reader, let us recall that given an exact functor $\fF\colon\ct_1\to\ct_2$ and a thick subcategory $\cs\subseteq\ct_1$, we say that $\fF$ is the idempotent completion of the Verdier quotient that sends $\cs$ to zero if we have a factorization
\[
\xymatrix{
\ct_1\ar[r]^-{\fF}\ar[dr]_-{\fQ}&\ct_2\\
&\ct_1/\cs\ar@{^(->}[u]_{\fI},
}
\]
where $\fQ$ is the Verdier quotient functor and $\fI$ is the idempotent completion of $\ct_1/\cs$.

Let us now discuss the two geometric examples that are of interest here.

\begin{ex}[Quasi-coherent sheaves]\label{ex:DQCoh}
Let $X$ be a quasi-compact and quasi-separated scheme. Let $\Dqa(X)$ be the full triangulated subcategory of $\Da(\Mod{\so_X})$ consisting of complexes with quasi-coherent cohomologies, for $?=b,+,-,\emptyset$. Let $U_1,U_2\subseteq X$ be quasi-compact open subsets such that $X=U_1\cup U_2$. Denote by $\iota_i\colon U_i\mono X$ and $\iota_{1,2}\colon U_1\cap U_2\mono X$ be the open immersions.

Let $(\cc,\fF)$ be a dg enhancement of $\Dqa(X)$ and assume that $\cc$ is h-flat. Set
\begin{gather*}
\cd_i:=\left\{C\in\cc\st\iota_i^*\comp\fF(C)\iso 0\right\},\\
\cd_{1,2}:=\left\{C\in\cc\st\iota_{1,2}^*\comp\fF(C)\iso 0\right\},
\end{gather*}
for $i=1,2$. Assumption (i) in Setup \ref{setting1} is then verified. Moreover, $\overline\cd_i=H^0(\cd_i)\iso\Da_{X\setminus U_i}(X)$ and $\overline\cd_{1,2}=H^0(\cd_{1,2})\iso\Da_{X\setminus (U_1\cap U_2)}(X)$, where $\Da_Z(X)$ denotes the full (triangulated) subcategory of $\Dqa(X)$ consisting of complexes with cohomology supported on the closed subset $Z\subseteq X$. Given this identification, the fact that $X=U_1\cup U_2$ clearly implies that $\Hom_{H^0(\cc)}(H^0(\cd_i),H^0(\cd_j)=0$, when $i\neq j\in\{1,2\}$. Thus (ii) in Setup \ref{setting1} holds true.

Consider the (idempotent completion of the) Drinfeld quotients $\cc_{\cd_i}:=\cc/\cd_i$ and $\cc_{\cd_1,\cd_2}:=\cc/\cd_{1,2}$. Following the discussion in \autoref{subsec:hprojhflat}, we get dg quotient functors $\fQ_i\colon\cc\to\cc_{\cd_i}$ and $\fQ_{1,2}\colon\cc\to\cc_{\cd_1,\cd_2}$. Moreover, $\cc_{\cd_1,\cd_2}$ is, at the same time, the Drinfeld quotient of $\cc_{\cd_1}$ and of $\cc_{\cd_2}$ by $\fQ_1(\cd_2)$ and $\fQ_2(\cd_1)$ respectively. Hence we also get the dg functors $\overline\fQ_i\colon\cc_{\cd_i}\to\cc_{\cd_1,\cd_2}$, for $i=1,2$. Since $\cc$ is h-flat, the results in \autoref{subsec:hprojhflat} show that 
\begin{gather*}
H^0(\cc_{\cd_i})\iso H^0(\cc)/H^0(\cd_i)\iso\Dqa(U_i), \\
H^0(\cc_{\cd_1,\cd_2})\iso H^0(\cc)/H^0(\cd_{1,2})\iso\Dqa(U_1\cap U_2),
\end{gather*}
for $i=1,2$. Moreover $H^0(\fQ_i)$ and $H^0(\overline\fQ_1)\comp H^0(\fQ_1)$ are the corresponding Verdier quotient functors. Hence assumptions (iii) and (iv) in Setup \ref{setting1} are satisfied as well.
\end{ex}

\begin{ex}[Perfect complexes]\label{ex:Perf}
As in the previous example, let $X$ be a quasi-compact and quasi-separated scheme. Let $\Dp(X)$ be the category of perfect complexes on $X$ (i.e.\ complexes in $\Dq(X)$ which are locally quasi-isomorphic to bounded complexes of locally free sheaves of finite type). By \cite[Theorem 3.1.1]{BB}, $\Dp(X)$ can alternatively be described as the category of compact objects $\Dq(X)^c$.

If $(\cc,\fF)$ is a dg enhancement of $\Dp(X)$ with $\cc$ h-flat, then the same construction as in \autoref{ex:DQCoh} yields the dg categories $\cd_i$, $\cd_{1,2}$, $\cc_{\cd_i}$ and $\cc_{\cd_1,\cd_2}$ with dg functors $\fQ_i$ and $\overline\fQ_i$. The fact that the assumptions of Setup \ref{setting1} are satisfied can be checked as in \autoref{ex:DQCoh}. One has to be careful only about (iii) and (iv) as, for localization theory for compact objects in \cite{N1,TT}, the functor $H^0(\fQ_i)\colon H^0(\cc)\to H^0(\cc_{\cd_i})\iso\Dp(U_i)$ is the composition
\[
H^0(\cc)\lto H^0(\cc)/H^0(\cd_i)\mono H^0(\cc_{\cd_i}),
\]
where the first functor is the Verdier quotient functor while the latter inclusion identifies $H^0(\cc_{\cd_i})$ to the idempotent completion of $H^0(\cc)/H^0(\cd_i)$. The same is true for the composition $H^0(\overline\fQ_1)\comp H^0(\fQ_1)$.
\end{ex}

\subsection{A useful criterion}\label{subsec:usefcrit}

This section is devoted to the proof of the following result which will be crucial for our geometric applications.

\begin{thm}\label{thm:critpb}
In Setup \ref{setting1}, the natural dg functor from $\cc$ to the homotopy pullback $\cc_{\cd_1}\times^h_{\cc_{\cd_1,\cd_2}}\cc_{\cd_2}$ of the diagram
\[
\cc_{\cd_1}\mor{\overline\fQ_1}\cc_{\cd_1,\cd_2}\xleftarrow{\overline\fQ_2}\cc_{\cd_2}
\]
is a quasi-equivalence.
\end{thm}

\begin{proof}
To simplify the notation, we set $\cp:=\cc_{\cd_1}\times^h_{\cc_{\cd_1,\cd_2}}\cc_{\cd_2}$ and we use its explicit construction discussed in \autoref{subsec:modelpullbacks}. We denote by $\fF\colon\cc\to\cp$ the natural dg functor explicitly described in \autoref{rmk:hompullback2}. We want to prove that it is a quasi-equivalence.

Let us begin by proving that the functor $\fF$ is quasi-fully-faithful. Since the assumptions of \autoref{prop:prelcolim} are satisfied, any object $C\in\Ob(\cc)$ sits in a commutative diagram
\[
\xymatrix{
	& D_2\ar[d]\ar@{=}[r] & D_2\ar[d]   \\
	D_1\ar[r]\ar@{=}[d] & C\ar[r]\ar[d] &C_{D_1}\ar[d] \\
	D_1\ar[r] & C_{D_2}\ar[r] &C_{D_1,D_2}
}
\]
where rows and columns yield distinguished triangles in $H^0(\hproj{\cc})$, $D_i\in\Ob(\hproj{\cd_i})$ and the complexes
\begin{equation}\label{eqn:compb1}
\Hom_{\hproj{\cc}}\left(\hproj{\cd_i},C_{D_i}\right)\qquad\Hom_{\hproj{\cc}}\left(\hproj{\cd_i},C_{D_1,D_2}\right)
\end{equation}
are acyclic, for $i=1,2$.
If follows first of all that the square
\begin{equation}\label{eqn:compb2}
\xymatrix@R-19pt{
	\Hom_{\hproj{\cc}}\left(C',C\right)\ar[r]\ar[dd] &
	\Hom_{\hproj{\cc}}\left(C',C_{D_1}\right)\ar[dd] \\
	 &  \\
	\Hom_{\hproj{\cc}}\left(C',C_{D_2}\right)\ar[r] &
	\Hom_{\hproj{\cc}}\left(C',C_{D_1,D_2}\right)
}
\end{equation}
is a homotopy cartesian square in $\dgMod{\kk}$ (here we use the same terminology as in \cite[Section 1.4]{N2}), for all $C'\in\Ob(\cc)$. Now \eqref{eqn:compb1} and assumptions (iii) and (iv) in Setup \ref{setting1} imply that the natural maps
\[
\begin{split}
\Hom_{\hproj{\cc}}\left(C',C_{D_i}\right)&\lto\Hom_{\hproj{\cc_{\cd_i}}}\left(\Ind(\fQ_i)(C'),\Ind(\fQ_i)(C_{D_i})\right)\\
\Hom_{\hproj{\cc}}\left(C',C_{D_1,D_2}\right)&\lto\Hom_{\hproj{\cc_{\cd_1,\cd_2}}}\left(\Ind(\overline\fQ_1\comp\fQ_1)(C'),\Ind(\overline\fQ_1\comp\fQ_1)(C_{D_1,D_2})\right)
\end{split}
\]
induced, respectively, by $\Ind(\fQ_i)$, for $i=1,2$, and by $\Ind(\overline\fQ_1\comp\fQ_1)(=\Ind(\overline\fQ_2\comp\fQ_2))$ are all quasi-isomorphisms in $\dgMod{\kk}$. On the other hand, the images under $\Ind(\fQ_i)$ and $\Ind(\overline\fQ_i\comp\fQ_i)$ of the natural morphisms $C\to C_{D_i}$ and $C\to C_{D_1,D_2}$ in $\hproj{\cc}$ become, by construction, quasi-isomorphisms in $\hproj{\cc_{\cd_i}}$ and $\hproj{\cc_{\cd_1,\cd_2}}$, respectively.

Hence \eqref{eqn:compb2} is quasi-isomorphic in $\dgMod{\kk}$ to the square
\begin{equation}\label{eqn:compb4}
\xymatrix@R-19pt{
\Hom_{\cc}\left(C',C\right)\ar[r]\ar[dd] &\Hom_{\cc_{\cd_1}}\left(C',C\right)\ar[dd] \\
& \\
\Hom_{\cc_{\cd_2}}\left(C',C\right)\ar[r] & \Hom_{\cc_{\cd_1,\cd_2}}\left(C',C\right)
}
\end{equation}
where the arrows are induced by the quotient dg functors $\fQ_i$ and $\overline\fQ_i$, for $i=1,2$.

In conclusion, \eqref{eqn:compb4} must be homotopy cartesian in $\dgMod{\kk}$ and this implies that the dg functor $\fF$ induces a quasi-isomorphism
\[
\Hom_\cc(C',C)\lto\Hom_\cp(\fF(C'),\fF(C)),
\]
for all $C',C\in\Ob(\cc)$, meaning that $\fF$ is quasi-fully faithful.

It remains to show that $H^0(\fF)$ is essentially surjective. To this end let us recall that an object in $\cp$ is a triple $(C_1,C_2,f)$, where $C_i\in\Ob(\cc_{\cd_i})$ and $f\colon\overline\fQ_1(C_1)\to\overline\fQ_2(C_2)$ is a homotopy equivalence in $\cc_{\cd_1,\cd_2}$. Hence, we need to show that, given such a triple $(C_1,C_2,f)$, there is $C\in\Ob(\cc)$ and a homotopy equivalence $\fF(C)\iso(C_1,C_2,f)$ in $\cp$.

First we note that, by assumption (iii) in Setup \ref{setting1}, we may choose $\widetilde C_i\in\Ob(\cc)$ such that $C_i$ is a direct factor $\fQ_i(\widetilde C_i)$ in $H^0(\cc_{\cd_i})$, for $i=1,2$. We also have triangles in $\hproj{\cc}$
\begin{equation}\label{eqn:cofseq1}
D_i\lto\widetilde C_i\lto\widetilde C_{D_i}
\end{equation}
which become distinghished in $H^0(\hproj{\cc})$, such that
$\Hom_{\hproj{\cc}}\left(\hproj{\cd_i},\widetilde C_{D_i}\right)$ is acyclic and with $D_i\in\hproj{\cd_i}$, for $i=1,2$ (see \autoref{prop:prelcolim}). If we apply $\Ind(\fQ_i)$ to \eqref{eqn:cofseq1}, we get quasi-isomorphisms in $\hproj{\cc_{\cd_i}}$
\begin{equation}\label{eqn:isopb1}
\fQ_i(\widetilde C_i)\equiva\Ind(\fQ_i)(\widetilde C_{D_i}),
\end{equation}
for $i=1,2$. As the map induced by $\Ind(\fQ_i)$
\[
\Hom_{\hproj{\cc}}\left(\widetilde C_{D_i},\widetilde C_{D_i}\right)\lto \Hom_{\hproj{\cc_{\cd_i}}}\left(\fQ_i(\widetilde C_{D_i}),\fQ_i(\widetilde C_{D_i})\right)
\]
is a quasi-isomorphism, for $i=1,2$, the idempotent $e'_i$ realizing $C_i$ as a direct summand of $\fQ_i(\widetilde C_i)\iso\Ind(\fQ_i)(\widetilde C_{D_i})$ in $H^0(\hproj{\cc_{\cd_i}})$ can be lifted to a closed degree-$0$ morphism $e_i\colon\widetilde C_{D_i}\lto\widetilde C_{D_i}$ in $\hproj{\cc}$. Therefore, we can set
\[
\widehat C_i:=\hocolim\left(\widetilde C_{D_i}\mor{e_i}\widetilde C_{D_i}\mor{e_i}\dots\right)\in\Ob(\hproj{\cc}),
\]
for $i=1,2$. We claim that the complex $\Hom_{\hproj{\cc}}\left(\hproj{\cd_i},\widehat C_i\right)$ is acyclic. Indeed, for all $D\in\cd_i$, we have natural isomorphisms
\[
\Hom_{H^0(\hproj{\cc})}(D,\widehat C_i)\iso\colim\Hom_{H^0(\hproj{\cc})}(D, C_{D_i})
\]
and $\Hom_{H^0(\hproj{\cc})}(D, C_{D_i})$ is trivial by the definition of $C_{D_i}$. For the isomorphism, we used that $D\in\cd_i\subseteq\cc$ is a compact object in $H^0(\hproj{\cc})$. Since $H^0(\hproj{\cd_i})$ is compactly generated by the objects in $\cd_i$, the claim follows.

Moreover, there is an isomorphism in $H^0(\hproj{\cc_{\cd_i}})$
\begin{equation}\label{eqn:cofseq2}
\Ind(\fQ_i)(\widehat C_i)\iso C_i.
\end{equation}
Indeed, using the fact that $\Ind(\fQ_i)$ commutes with coproducts in $Z^0(\hproj{\cc})$, we get
\begin{multline*}
\Ind(\fQ_i)(\widehat C_i)\iso\hocolim\left(\Ind(\fQ_i)(\widetilde C_{D_i})\to\Ind(\fQ_i)(\widetilde C_{D_i})\to\dots\right)\\
\iso\hocolim\left(\fQ_i(\widetilde C_i)\mor{e'_i}\fQ_i(\widetilde C_i)\mor{e'_i}\dots\right)\iso C_i,
\end{multline*}
For the penultimate isomorphism we used \eqref{eqn:isopb1}.

Now we can go further and invoke \autoref{prop:prelcolim} again in order to get a triangle
\begin{equation}\label{eqn:cofseq3}
D\lto\widehat C_2\lto C_D,
\end{equation}
with $D\in\Ob(\hproj{\cd_1})$ and such that it becomes distinguished in $H^0(\hproj{\cc})$ and the complex $\Hom_{\hproj{\cc}}\left(\hproj{\cd_i},C_D\right)$ is acyclic, for $i=1,2$.

Hence we have a sequence of homotopy equivalences
\begin{equation}\label{eqn:cofseq5}
\overline\fQ_1\left(\Ind(\fQ_1)(\widehat C_1)\right)\iso\overline\fQ_1\left(C_1\right)\iso\overline\fQ_2\left(C_2\right)\iso\overline\fQ_2\left(\Ind(\fQ_2)(\widehat C_2)\right)\iso\overline\fQ_2\left(\Ind(\fQ_2)(C_D)\right).
\end{equation}
The first and the third ones are those in \eqref{eqn:cofseq2}. The second one is just $f$ in the given triple $(C_1,C_2,f)$. For the last one, observe that
\[
\fH:=H^0(\overline\fQ_2)\comp H^0(\Ind(\fQ_2))=H^0(\overline\fQ_1)\comp H^0(\Ind(\fQ_1)).
\]
Thus, if we apply $\fH$ to the distinguished triangle \eqref{eqn:cofseq3} in $H^0(\hproj{\cc})$, we get an isomorphism $\fH(\widehat C_2)\iso\fH(C_D)$ because $\fH(D)=H^0(\overline\fQ_1)\left(H^0(\Ind(\fQ_1))(D)\right)=0$.

In conclusion, we have morphisms
\begin{equation}\label{eqn:cofseq4}
\widehat C_1\mor{f_1} C_D\xleftarrow{f_2}\widehat C_2
\end{equation}
in $Z^0(\hproj{\cc})$. We claim that this can be completed to a square
\[
\xymatrix{
	C\ar[r]^-{g_2}\ar[d]_-{g_1} & \widehat C_2\ar[d]^-{f_2}\\
	\widehat C_1\ar[r]^-{f_1} & C_D
}
\]
in $Z^0(\hproj{\cc})$ which becomes commutative in $H^0(\hproj{\cc})$. Indeed, one proceeds as in \cite[Section 1.4]{N2} and consider the morphism $\varphi:=f_1+f_2\colon\widehat C_1\oplus\widehat C_2\to C_D$ in $Z^0(\hproj{\cc})$. We set $C:=\sh[-1]{\cone{\varphi}}$ so that we get a triangle
\begin{equation}\label{eqn:pbtria1}
\xymatrix{
C\ar[r]^-{\psi}&\widehat C_1\oplus\widehat C_2\ar[r]^-{\varphi}&C_D
}
\end{equation}
which becomes distinguished in $H^0(\hproj{\cc})$. One sets $g_i$ to be the composition of $\psi$ with the natural projection $\widehat C_1\oplus\widehat C_2\to\widehat C_i$. It is clear from the defining triangle \eqref{eqn:pbtria1} that there are isomorphisms in $H^0(\hproj{\cc})$
\begin{equation}\label{eqn:pbtria2}
\Ind(\fQ_i)(C)\iso\Ind(\fQ_i)(\widehat C_i)\iso C_i,
\end{equation}
where the last one is \eqref{eqn:cofseq2}.

The object $C$ is actually compact in $H^0(\hproj{\cc})$ (and thus it is contained in $H^0(\cc)\iso H^0(\Perf{\cc})$). Indeed, by \eqref{eqn:pbtria2}, $\Ind(\fQ_i)(C)\in\cc_{\cd_i}$ is compact. By (ii) in Setup \ref{setting1} we immediately have that 
\[
\Hom_{H^0(\hproj{\cc})}\left(H^0(\hproj{\cd_i}),H^0(\hproj{\cd_j})\right)=0,
\]
for $i\neq j\in\{1,2\}$, and thus $H^0(\hproj{\cd_1})$ and $H^0(\hproj{\cd_2})$ are localizing completely orthogonal full triangulated subcategories of $H^0(\hproj{\cc})$. Thus they provide Bousfield localizations. Since $H^0(\hproj{\cd_i})$ is compactly generated by the objects in $\cd_i$ which are compact in $H^0(\hproj{\cc})$, \cite[Corollary 5.12]{R} implies that $C$ is compact.

In conclusion, $C\in\cc$, $\fQ_i(C)\iso C_i$ by \eqref{eqn:pbtria2} and then, by \eqref{eqn:cofseq5}, we conclude that there is a quasi-isomorphism $\fF(C)\iso(C_1,C_2,f)$. 
\end{proof}

\subsection{Refining the geometric application}\label{subsec:geomapplpb}

We want to extend further the discussion in \autoref{ex:DQCoh} and \autoref{ex:Perf}. Assume $X$ to be a quasi-compact and quasi-separated scheme, and let $U_1,\dots, U_n\subset X$ be a finite collection of quasi-compact open subsets such that $X=U_1\cup\dots\cup U_n$.

For $I\subseteq N:=\{1,\dots,n\}$ we set $U_I:=\cap_{i\in I} U_i$ and we denote by $\iota_I\colon U_I\mono X$ the corresponding open immersion. Clearly $U_\emptyset=X$ and $\iota_\emptyset=\id$.

\begin{remark}\label{rmk:geompb}
It is clear that $\cup_{i\in I}U_i$ is quasi-compact and quasi-separated for every $I\subseteq N$. Also the open subsets
$U_I$ must be quasi-compact, because they are the intersections of the quasi-compact open subsets $U_i$ in the quasi-separated $X$.
Now assume that the open subsets $U_i\subset X$ are all affine; in this case for $I\neq\emptyset$ the open subset $U_I$ is also separated, being an open subset of an affine (hence separated) scheme. Therefore, by \cite[Corollary 5.5]{BN}, $\Dqa(U_I)\iso\Da(U_I)$, for $I\ne\emptyset$. Here and in what follows we use the shorthand $\Da(U_I):=\Da(\Qcoh(U_I))$.
\end{remark}

Let $(\cc,\fE)$ be a dg enhancement either of $\Dqa(X)$ or of $\Dp(X)$ and let us assume that $\cc$ is h-flat. We denote by $\cd_I$ the full dg subcategory of $\cc$ defined by
\[
\Ob(\cd_I):=\left\{C\in\cc\st\iota_I^*\comp\fE(C)\iso 0\right\}
\]
and set $\cc_I:=\Perf{\cc/\cd_I}$. We have an exact equivalence given by the composition
\[
\fE_I\colon H^0(\cc_I)\equiva H^0(\cc)/H^0(\cd_i)\equiva\Dqa(U_I)\qquad (\text{resp. } \fE_I\colon H^0(\cc_I)\iso\Dp(U_I)).
\]
Moreover the dg quotient functors $\fQ_I\colon\cc\to\cc_I$ are such that the diagram of triangulated categories and exact functors
\begin{equation}\label{eqn:commgeom1}
\xymatrix{
H^0(\cc)\ar[r]^-{H^0(\fQ_I)}\ar[d]_-{\fE}& H^0(\cc_I)\ar[d]^-{\fE_I}\\
\Dqa(X)\ar[r]^-{\iota_I^*}&\Da(U_I).
}
\end{equation}
is commutative.

By construction, whenever $I\subseteq I'\subseteq N$, there is a natural dg functor $\cc_I\to\cc_{I'}$ and these functors compose nicely. In particular, as in \autoref{rmk:genpb}, we can form the homotopy limit
\[
\hl{\cc}:=\holim_{\emptyset\ne I\subseteq N}\cc_I.
\]

\begin{remark}\label{rmk:comphlim}
In analogy with the usual homotopy limit of categories we can inductively compute $\hl{\cc}$ up to quasi-equivalence as follows due to the universal property in \autoref{rmk:genpb}.

Let $N':=\{1,\dots,n-1\}$ and consider the homotopy limits
\[
\cc':=\holim_{\emptyset\ne I\subseteq N'}\cc_I \qquad \cc'':=\holim_{\{n\}\subsetneq I\subseteq N}\cc_I.
\]
By construction we have natural dg functors
\[
\overline\fQ'_1\colon\cc'\lto\cc''\qquad\overline\fQ'_2\colon\cc_{\{n\}}\lto\cc''.
\]
Then $\hl{\cc}$ is the homotopy pullback of the diagram
\[
\cc'\mor{\overline\fQ_1}\cc''\xleftarrow{\overline\fQ_2}\cc_{\{n\}}.
\]
\end{remark}

For $n=2$ we have the following easy result.

\begin{lem}\label{lem:pbgeom}
Let $X$ be a quasi-compact and quasi-separated scheme and let $U_1,U_2\subseteq X$ be quasi-compact open subschemes such that $X=U_1\cup U_2$. Assume that $\cc$ is an h-flat dg enhancement either of $\Dqa(X)$ or of $\Dp(X)$, for $?=b,+,-,\emptyset$. Then there is a quasi-equivalence $\cc\to\hl{\cc}$.
\end{lem}

\begin{proof}
By \autoref{rmk:comphlim}, $\hl{\cc}$ is just the pullback of the diagram
\[
\cc_{\{1\}}\lto\cc_{\{1,2\}}\longleftarrow\cc_{\{2\}}.
\]
Since, by \autoref{ex:DQCoh} and \autoref{ex:Perf}, the assumptions in Setup \ref{setting1} are verified (with $\cc_{\{i\}}=\cc_{\cd_i}$, for $i=1,2$, and $\cc_{\{1,2\}}=\cc_{\cd_1,\cd_2}$), the result follows from \autoref{thm:critpb}.
\end{proof}

The following is the natural generalization of \autoref{lem:pbgeom}.

\begin{prop}\label{prop:enpb}
Let $X$ be a quasi-compact and quasi-separated scheme and let $\cc$ be an h-flat dg enhancement either of $\Dqa(X)$ or of $\Dp(X)$, for $?=b,+,-,\emptyset$. Then there is a quasi-equivalence $\cc\to\hl{\cc}$.
\end{prop}

\begin{proof}
The argument is by induction on $n$, the number of quasi-compact open subsets in the covering $\{U_1,\dots,U_n\}$ of $X$. Clearly, if $n=1$, there is nothing to prove. Thus we can assume $n\geq 2$ and set 
\[
V_1:=U_1\cup\dots\cup U_{n-1}\qquad V_2:=U_n.
\]
Clearly $V_i$ is quasi-compact for $i=1,2$. Define $\cc_{\cd_i}$, $\cc_{\cd_1,\cd_2}$, $\fQ_i$ and $\overline\fQ_i$ as in \autoref{ex:DQCoh} or \autoref{ex:Perf}. In particular, $\cc_{\cd_i}$ and $\cc_{\cd_1,\cd_2}$ are dg enhancements of $\Dqa(V_i)$ and $\Dqa(V_1\cap V_2)$ (or of $\Dp(V_i)$ and $\Dp(V_1\cap V_2)$), for $i=1,2$.

Consider now the homotopy limits $\cc'$ and $\cc''$ in \autoref{rmk:comphlim}. By induction, based on \autoref{rmk:comphlim} and \autoref{lem:pbgeom}, it is easy to see that $\cc'$ is a dg enhancement of $\Dqa(V_1)$ and $\cc''$ is a dg enhancement of $\Dqa(V_1\cap V_2)$ (resp.\ $\Dp(V_1)$ and $\Dp(V_1\cap V_2)$). The induction gives yet a more precise information: there are quasi-equivalences $\fF'\colon\cc_{\cd_1}\to\cc'$ and $\fF''\colon\cc_{\cd_1,\cd_2}\to\cc''$ making the diagrams
\begin{equation}\label{eqn:commsqhl}
\xymatrix{
	\cc_{\cd_1}\ar[r]^-{\overline\fQ_1}\ar[d]_-{\fF'} & 	 \cc_{\cd_1,\cd_2}\ar[d]^-{\fF''} & \cc_{\cd_2}\ar[l]_-{\overline\fQ_2}\ar@{=}[d]\\
	\cc'\ar[r]^-{\overline\fQ'_1} & \cc'' &\cc_{\{n\}}\ar[l]_-{\overline\fQ'_2}
}
\end{equation}
commutative in $\dgCat$. Note that the induction applies to $\cc''$ as well, since $V_1\cap V_2=(U_1\cap U_n)\cup\dots\cup(U_{n-1}\cap U_n)$ is union of $n-1$ quasi-compact open subsets.

By \autoref{lem:pbgeom}, there is a quasi-equivalence $\fF_1\colon\cc\to\cc_{\cd_1}\times^h_{\cc_{\cd_1,\cd_2}}\cc_{\cd_2}$. By \eqref{eqn:commsqhl} and \cite[Proposition 13.3.9]{Hir}, we get a quasi-equivalence
\[
\fF_2\colon\cc_{\cd_1}\times^h_{\cc_{\cd_1,\cd_2}}\cc_{\cd_2}\lto\cc'\times^h_{\cc''}\cc_{\{n\}}
\]
where the latter dg category if $\hl{\cc}$ by \autoref{rmk:comphlim}. Hence $\fF:=\fF_2\comp\fF_1\colon\cc\to\hl{\cc}$ is the quasi-equivalence we are looking for.
\end{proof}

In \autoref{prop:enpb} (and in \autoref{lem:pbgeom}) we always assumed that the dg enhancement $\cc$ is h-flat. If not, we can take the quasi-equivalence $\hff{\fI}_\cc\colon\hff{\cc}\to\cc$ in \autoref{prop:genDr}. By \autoref{prop:enpb}, we get a quasi-equivalence
\[
\fF\colon\hff{\cc}\lto\hl{\left(\hff{\cc}\right)}.
\]
Then, the composition $[\fF]\comp[\hff{\fI}_\cc]^{-1}$ yields an isomorphism between $\cc$ and $\hl{\left(\hff{\cc}\right)}$ in $\Hqe$.

In conclusion, we have the following.

\begin{cor}\label{cor:enpb}
Let $X$ be a quasi-compact and quasi-separated scheme and let $\cc$ be a dg enhancement either of $\Dqa(X)$ or of $\Dp(X)$, for $?=b,+,-,\emptyset$. Then there is an isomorphism $\cc\iso\hl{\left(\hff{\cc}\right)}$ in $\Hqe$.
\end{cor}

\section{Uniqueness of enhancements for geometric categories}\label{sect:uniqgeomcat}

In this section we prove \autoref{thm:main2}; more precisely 
the proof is in \autoref{subsec:Dq} and \autoref{subsec:Dp}. In \autoref{subsect:sidermk} we show how a weaker version of \autoref{thm:main2}, for the category of perfect complexes can be obtained by `simpler' means.

\subsection{Proof of \autoref{thm:main2}: the case of $\Dqa(X)$}\label{subsec:Dq}

Write $X$ as a finite union $X=U_1\cup\dots\cup U_n$ of affine open subschemes. The argument is based on induction on $n$. 

When $X$ is affine (or, more generally, quasi-compact and separated) $\Dqa(X)\iso\Da(\Qcoh(X))$ by \cite[Corollary 5.5]{BN}, whence $\Dqa(X)$ has a unique enhancement by \autoref{thm:main1} (1).
Thus we may assume $n\geq2$. 

Suppose that $(\cc,\fE)$ is a dg enhancement of $\Dqa(X)$. Up to replacing $\cc$ with the quasi-equivalent dg category $\Perf{\cc}$, we can assume that $\cc$ satisfies \eqref{eqn:Z0nullhomo}. 
Now we construct the dg category $\widehat\cb:=\widehat\cb_{\cc,\fE}(\Mod{X},\Qcoh(X))$, together with a fully faithful dg functor $\widehat\cb\to\cc$. 
In view of \autoref{prop:invhf}, up to replacing $\cc$ with the quasi-equivalent dg category $\hff{\cc}$, we can assume that $\cc$ is h-flat. 

Now we can define the dg categories $\cc_I$ and the dg functors $\fQ_I\colon\cc\to\cc_I$ as in \autoref{subsec:geomapplpb} and we keep the same notation as in that section.
For every subset $\emptyset\neq I\subseteq N=\{1,\dots,n\}$, as $\cc_I$ has the same objects as $\cc$, we can construct a dg category $\cb_I$ with the same objects as $\widehat\cb$ and with
\[
\Hom_{\cb_I}(B_1,B_2):=\Hom_{\cc_I}(B_1^-\oplus B_1^+,B_2^-\oplus B_2^+),
\]
Moreover, by construction, the dg functor $\fQ_I$ induces a natural dg functor
\[
\widehat\fQ_I\colon\widehat\cb\lto\cb_I.
\]

\begin{remark}\label{rmk:perfBI}
Arguing as in \autoref{CperfB}, we see that the image of $H^0(\widehat\cb)$ under the composition
\[
H^0(\cc)\mor{H^0(\fQ_I)} H^0(\cc_I)\mor{\fE_I}\Dqa(U_I)\iso\Da(U_I):=\Da(\Qcoh(U_I))
\]
(see \autoref{rmk:geompb} for the last equivalence) coincides with $\Ba(\Qcoh(U_I))$. This simple observation follows from the definition of $\widehat\cb$ together with the fact that a quasi-coherent sheaf on $U_I$ extends to a quasi-coherent sheaf on $X$ and the functor
\[
H^0(\widehat\cb)\to H^0(\cc)\mor{\fE}\Dqa(X)
\]
surjects on complexes of quasi-coherent sheaves with zero differentials. Hence, $H^0(\fQ_I)(H^0(\widehat\cb))\iso H^0(\cb_I)\iso\Ba(\Qcoh(U_I))$ and the dg functor $\cb_I\to\cc_I$ induces a quasi-equivalence $\Perf{\cb_I}\iso\cc_I$, because of \autoref{cor:genD}. Moroever, if $\emptyset\neq I\subseteq I'\subseteq N$, we have, dg functors
\[
\widetilde\fQ_{I,I'}\colon\Perf{\cb_I}\lto\Perf{\cb_{I'}}
\]
corresponding to the restriction $\Da(U_I)\to\Da(U_{I'})$. In particular, these dg functors compose nicely.
\end{remark}

By the invariance of homotopy limits under quasi-equivalences (see \autoref{rmk:genpb}), we have a quasi-equivalence
\[
\holim\Perf{\cb_I}\lto\hl{\cc}.
\]
On the other hand, by \autoref{prop:enpb} (see also \autoref{cor:enpb}), we have a quasi-equivalence $\cc\lto\hl{\cc}$. Putting it all together, we get an isomorphism
\begin{equation}\label{eqn:isofin1}
\cc\iso\holim\Perf{\cb_I}
\end{equation}
in $\Hqe$.

By \autoref{rmk:recast}, given $I$, as a consequence of our proof of \autoref{thm:main1} (1) for the uniqueness of dg enhancements of $\Da(U_I)$, we get full dg subcategories $\cn_I\subseteq\Perf{\cv_I}$, $\wt\cs_I\subseteq\Perf{\wt\cb_I}$, $\overline\cs_I\subseteq\Perf{\overline\cb_I}$, $\cs''_I\subseteq\Perf{\cb''_I}$, $\cs'_I\subseteq\Perf{\cb'_I}$ and $\cs_I\subseteq\Perf{\cb_I}$ and the commutative diagram of zigzags
\begin{equation*}\label{eqn:biqzigzag4}
\xymatrix@C-15pt{
&\Perf{\widetilde\cB_I}\ar[dl]\ar[dr]\ar[d]&&\Perf{\cB''_I}\ar[dl]\ar[dr]\ar[d]&&\Perf{\cB_I}\ar[dl]\ar[d]\\
\Perf{\cV_I}\ar[d]&\Perf{\wt\cb_I}/\wt\cs_I\ar[dl]\ar[dr]&\Perf{\overline\cB_I}\ar[d]&\Perf{\cb''_I}/\cs''_I\ar[dl]\ar[dr]&\Perf{\cB'_I}\ar[d]&\Perf{\cb_I}/\cs_I\ar[dl]\\
\Perf{\cV_I}/\cn_I&&\Perf{\overline\cB_I}/\overline\cs_I&&\Perf{\cB'_I}/\cs'_I&
}
\end{equation*}
where the dg functors in the bottom zigzag are quasi-equivalences, the vertical maps are Drinfeld dg quotient functors such that $\Perf{\cb'_I}\to\Perf{\cb'_I}/\cs'_I$ and $\Perf{\cb_I}\to\Perf{\cb_I}/\cs_I$ are quasi-equivalences. Most importantly, the lower zigzag remains functorial in $I$ (by \autoref{rmk:univpropquot}).

Thus, by \autoref{rmk:genpb} we get the following chain of isomorphisms in $\Hqe$
\begin{multline*}
\holim\Perf{\cv_I}/\cn_I\iso\holim\Perf{\widetilde\cb_I}/\widetilde\cs_I\iso\holim\Perf{\overline\cb_I}/\overline\cs_I\iso\holim\Perf{\cb''_I}/\cs''_I\\
\iso\holim\Perf{\cb'_I}/\cs'_I\iso\holim\Perf{\cb_I}/\cs_I\iso\holim \Perf{\cb_I}.
\end{multline*}
Together with \eqref{eqn:isofin1}, we get an isomorphism $\holim \Perf{\cv_I}/\cn_I\iso\cc$ in $\Hqe$. But the construction of $\cn_I$ and of the dg functors $\cv_I\to\cv_{I'}$ (and thus of the induced ones $\Perf{\cv_I}/\cn_I\to\Perf{\cv_{I'}}/\cn_{I'}$) are independent of the given dg enhancement $(\cc,\fE)$. Indeed, $\cn_I$ is an enhancement of $\Acya(U_I)$ and the dg functors are induced by the open immersions $\iota_{I,I'}$ above.

In conclusion, all dg enhancements of $\Dqa(X)$ are isomorphic in $\Hqe$ to $\holim \Perf{\cv_I}/\cn_I$ and thus $\Dqa(X)$ has a unique dg enhancement.

\subsection{Proof of \autoref{thm:main2}: the case of $\Dp(X)$}\label{subsec:Dp}

The case of perfect complexes is treated with a strategy which is essentially identical to the one in \autoref{subsec:Dq}: we write $X$ as a finite union $X=U_1\cup\dots\cup U_n$ of affine open subschemes, we proceed by induction on $n$, and finish off by using \autoref{subsec:geomapplpb} in an essential way. The main difference is that, instead of following the construction of $\widetilde\cb$ and $\cv$ in \autoref{subsec:defB} and \autoref{subsec:variantB}, we will partially use the strategy used in \cite{LO} to prove \autoref{thm:LO}.

If $n=1$, then $X$ is affine and the result follows from \cite[Proposition 2.6]{LO}. Thus we can assume $n\geq 2$ and start proving the following simple result.

\begin{lem}\label{lem:genopaff}
Let $X$ be an affine scheme and $\iota\colon U\mono X$ be an open subscheme. Then $\Dq(U)$ is compactly generated by the compact object $\so_U=\iota^*\so_X$ and $\Dp(U)=\gen{\so_U} \infty$. 
\end{lem}

\begin{proof}
Let $Z:=X\setminus U$ and denote by $\D_Z(X)$ the full triangulated subcategory of $\Dq(X)$ consisting of complexes with (topological) support contained in the closed subset $Z$. Consider the short exact sequence of triangulated categories
\[
0\lto \D_Z(X)\lto\Dq(X)\lmor{\iota^*}\Dq(U)\lto 0.
\]
By \cite[Theorem 6.8]{R}, $\D_Z(X)$ is compactly generated by objects which are compact in $\Dq(X)$. Hence, by \cite[Theorem 2.1]{N1}, $\iota^*$ sends the compact generator $\so_X$ to the compact generator $\so_U$ and $\Dq(U)^c=\gen{\so_U}\infty$. Since $\Dp(U)=\Dq(U)^c$ (see \autoref{ex:Perf}), this concludes the proof.
\end{proof}

Hence, in our situation, $\Dp(U_I)=\gen{\so_{U_I}} \infty$, for all $\emptyset\neq I\subseteq N$. This implies that, given a dg enhancement $(\cc,\fE)$ of $\Dp(X)$ inducing equivalences $\fE_I\colon H^0(\cc_I)\equiva\Dp(U_I)$, we can take the dg category $\cn_I\subseteq\cc_I$ with only one object $O_I$ which is the lift of $\so_{U_I}$ along $\fE_I$. By the above result the inclusion of $\cn_I$ in $\cc_I$ induces a quasi-equivalence $\Perf{\cn_I}\iso\cc_I$. Moreover, as in the previous section, the quotient dg functors $\cc_I\to\cc_{I'}$, for $\emptyset\neq I\subseteq I'\subset N$, induce dg functors $\cn_I\to\cn_{I'}$ and thus dg functors
\[
\widetilde\fQ_{I,I'}\colon\Perf{\cn_I}\lto\Perf{\cn_{I'}}
\]
with the same properties as in \autoref{rmk:perfBI}.

On the other hand, $\so_{U_I}$ as an object of $\Dp(U_I)$ generates a dg category $\ca_I$ all sitting in degree zero. The open inclusion $\iota_{I,I'}\colon U_{I'}\mono U_I$ induce a dg functor
\[
\widetilde\fQ'_{I,I'}\colon\ca_I\lto\ca_{I'}.
\]
The argument in \cite[Section 6]{LO} applies in this case and thus:
\begin{enumerate}
\item[(a)] There is a morphism $\Perf{\ca_I}\to\Perf{\cn_I}$ in $\Hqe$ that can be represented by a roof
\[
\xymatrix@C+20pt@R-20pt{
&\Perf{\widetilde\cn_I}\ar[dl]\ar[dr] &\\
\Perf{\ca_I} & & \Perf{\cn_I}, 
}
\]
which is functorial in $I$\footnote{The argument in \cite{LO} actually provides a composition of roofs $\ca_I\to H^0(\cn_I)$ and $H^0(\cn_I)\leftrightarrow\tau_{\leq 0}(\cn_I)\to\cn_I$. But this composition can be easily rearranged in the form above.}.
\item[(b)] A full dg subcategory $\cl_I\subseteq\Perf{\ca_I}$ such that the idempotent completion
\[
\widehat\ca_I:=\ic{(\Perf{\ca_I}/\cl_I)}
\]
of the Drinfeld quotient is a dg enhancememnt of $\Dp(U_I)$. Moreover the definition of $\widehat\cA_I$ is functorial in $I$ and both its definition and of the natural dg functors $\widehat\ca_I\lto\widehat\ca_{I'}$ induced by $\widetilde\fQ'_{I,I'}$ are independent of the dg enhancement $(\cc,\fE)$.
\end{enumerate}
Here, as in the previous section, we are freely using the possibility of taking functorial h-flat resolutions.

As a result of Lemma 6.2 and Lemma 6.3 in \cite{LO}, we get a roof of quasi-equivalences
\[
\xymatrix@C+20pt@R-20pt{
&\Perf{\widetilde\cn_I}/\cl'_I\ar[dl]\ar[dr] &\\
\widehat\ca_I & & \Perf{\cn_I}, 
}
\]
where $\cl'_I$ is a full pretriangulated dg subcategory of $\Perf{\widetilde\cn_I}$ identified with $\cl_I$ under the quasi-equivalence $\Perf{\widetilde\cn_I}\to\widehat\ca_I$ and the diagram is functorial in $I$.

In conclusion, as in the previous section, we have a chain of isomorphism in $\Hqe$
\[
\cc\iso\hl{\cc}\iso\holim\Perf{\cn_I}\iso\holim\widehat\ca_I
\]
due to \autoref{prop:enpb} and the invariance of homotopy limits under quasi-equivalences. Thus $\Dp(X)$ has a unique enhancement since the latter homotopy limit does not depend on $(\cc,\fE)$. Therefore, the proof of \autoref{thm:main2} is complete.

\begin{remark}\label{rmk:compimplDq}
As we observed in \autoref{ex:Perf}, the triangulated category $\Dp(X)$ identifies with $\Dq(X)^c$, when $X$ is a quasi-compact and quasi-separated scheme. It is then easy (see, for example, the argument in \cite[Remark 5.5]{CS4}) to deduce the uniqueness of enhancement of $\Dq(X)$ from that of $\Dp(X)$. The complications in \autoref{subsec:Dq} emerge while dealing with $\Dqa(X)$, when $?=b,+,-$. 
\end{remark}

\subsection{An aside}\label{subsect:sidermk}

It is worth pointing out that a weak version of \autoref{thm:main2} for $\Dp(X)$ can be proved using techniques very similar to the ones in \cite{CSUn1}.

The simple observation is that assumption (2) in \cite[Theorem B]{CSUn1} can be replaced by the weaker (ii) below in order to get the following slightly more general result.

\begin{thm}\label{thm:newCS}
Let $\cg$ be a Grothendieck category with a small set $\ca$ of generators such that
\begin{enumerate}
\item[{\rm (i)}] $\ca$ is closed under finite coproducts;
\item[{\rm (ii)}] If $f\colon\coprod_{i\in I}C_i\to C$ (with $I$ a small set) is a morphism in $\cg$ and $A\in\Ob(\ca)$ is a subobject of $C$ such that $A\subseteq\im f$, then there exists a finite subset $I'$ of $I$ such that $A\subseteq f(\coprod_{i\in I'}C_i)$;
\item[{\rm (iii)}] If $f\colon A'\epi A$ is an epimorphism of $\cg$ with $A,A'\in\ca$, then $\ker f\in\cA$;
\item[{\rm (iv)}] For every $A\in\ca$ there exists $N(A)>0$ such that $\Hom_{\D(\cg)}\left(A,\sh[N(A)]{A'}\right)=0$ for every $A'\in\Ob(\ca)$.
\end{enumerate}
Then $\D(\cg)^c$ has a unique enhancement.
\end{thm}

Actually it is very easy to see that the argument used in \cite{CSUn1} to prove the theorem still works with the new hypothesis and proves \autoref{thm:newCS}. Indeed, in \cite[Remark 6.3]{CSUn1} it is proved, in particular, that the old assumption (2) implies (ii). Moreover, in the other parts of the proof of \cite[Theorem B]{CSUn1} where (2) was used (in the proofs of Lemma 6.4 and Proposition 6.5, where Remark 6.3 is invoked) precisely (ii) is needed.

As an application, we obtain the following improvement of \cite[Proposition 6.10]{CSUn1} in the case of schemes. Recall that a scheme $X$ \emph{has enough perfect coherent sheaves} if $\Qcoh(X)$ is generated, as a Grothendieck category, by a small set of objects in $\Coh(X)\cap\Dp(X)$.

\begin{prop}\label{prop:geocpt}
Let $X$ be a quasi-compact and semi-separated scheme with enough perfect coherent sheaves. Then $\Dp(X)$ has a unique enhancement.
\end{prop}

\begin{proof}
The argument of the proof of \cite[Proposition 6.10]{CSUn1} (with the same $\ca$) shows that conditions (i), (iii) and (iv) of \autoref{thm:main2} are satisfied. As for (ii), one is reduced to proving the following statement: if $\s{F}=\bigcup_{i\in I}\s{F}_i\in\Coh(X)$, with $\s{F}_i$ a quasi-coherent subsheaf of $\s{F}$ for every $i\in I$, then there exists a finite subset $I'$ of $I$ such that $\s{F}=\bigcup_{i\in I'}\s{F}_i$. Since $X$ (being quasi-compact) has a finite affine open cover, we can also assume that $X$ is affine, in which case the statement follows directly from the fact that $\s{F}$ is of finite type.
\end{proof}

\autoref{prop:geocpt} is a strictly stronger version than \cite[Proposition 6.10]{CSUn1}, as is demonstrated by the following example.

\begin{ex}
If $X$ is an affine scheme such that $\so_X$ is coherent, then the hypotheses of \autoref{prop:geocpt} are satisfied, since $\so_X\in\Coh(X)\cap\Dp(X)$ is a generator of $\Qcoh(X)$. On the other hand, if $X$ is not noetherian, then the hypotheses of \cite[Proposition 6.10]{CSUn1} are not satisfied. An explicit example is given by $X=\spec(A)$ with $A$ the polynomial ring in infinitely many variables over a field.
\end{ex}

We conclude our presentation by recalling that in \cite[Corollary 9]{A1} the author proved that $\Dp(X)$ has a unique enhancement when $X$ is a quasi-compact, quasi-separated and $0$-complicial scheme. We refer to \cite{A1} for the precise definition of $0$-complicial which is not needed here. On the other hand, it is not complicated to see that a quasi-compact and semi-separated scheme with enough perfect coherent sheaves is $0$-complicial and thus \cite[Corollary 9]{A1} should be more general than \autoref{prop:geocpt}. It would be interesting to find an example of a scheme satisfying the assumptions in \cite[Corollary 9]{A1} but not the assumptions in the above proposition.

\bigskip

{\small\noindent{\bf Acknowledgements.} It is our great pleasure to thank Benjamin Antieau, Pieter Belmans, Henning Krause and Valerio Melani for insightful conversations about several problems related to the contents of this paper. We are also very grateful to Michel Van den Bergh, Alexander Kuznetsov, Valery Lunts and Bregje Pauwels for comments on a preliminary version of this paper.}



\begin{thebibliography}{99}

\bibitem{A1} B.\ Antieau, \emph{On the uniqueness of infinity-categorical enhancements of triangulated categories}, arXiv:1812.01526.

\bibitem{BS} P.\ Balmer, M.\ Schlichting, \emph{Idempotent Completion of Triangulated Categories}, J.\ Algebra {\bf 236} (2001), 819--834.

\bibitem{BeiBerDel82}
A.A.\ Be{\u\i}linson, J.\ Bernstein, P.\ Deligne,
  \emph{Faisceaux pervers}, Ast{\'e}risque {\bf 100} (1982), 5--171.

\bibitem{BLL} A.\ Bondal, M.\ Larsen, V.\ Lunts, \emph{Grothendieck ring of pretriangulated categories}, Int.\ Math.\ Res.\ Not.\ {\bf 29} (2004), 1461--1495.

\bibitem{BN}  M.\ B\"okstedt, A.\ Neeman, \emph{Homotopy limits in triangulated categories}, Comp.\ Math.\ {\bf 86} (1993), 209--234.

\bibitem{BB} A.\ Bondal, M. Van den Bergh, \emph{Generators and representability of functors in commutative and noncommutative geometry}, Moscow Math. J. {\bf 3} (2003), 1--36.


\bibitem{COS} A.\ Canonaco, M.\ Ornaghi, P.\ Stellari, \emph{Localizations of the category of $A_\infty$ categories and internal Homs}, Doc.\ Math.\ {\bf 24} (2019), 2463--2492.

\bibitem{CS4} A.\ Canonaco, P.\ Stellari, \emph{A tour about existence and uniqueness of dg enhancements and lifts}, J.\ Geom.\ Phys.\ {\bf 122} (2017), 28--52.

\bibitem{CS3} A.\ Canonaco, P.\ Stellari, \emph{Fourier--Mukai functors: a survey},  EMS Ser.\ Congr.\ Rep., Eur.\ Math.\ Soc.\ (2013), 27--60.

\bibitem{CSdg} A.\ Canonaco, P.\ Stellari, \emph{Internal Homs via extensions of dg functors}, Adv.\ Math.\ {\bf 277} (2015), 100--123.

\bibitem{CSFMKer} A.\ Canonaco, P.\ Stellari, \emph{Non-uniqueness of Fourier--Mukai kernels}, Math.\ Z.\ {\bf 272} (2012), 577--588.
	
\bibitem{CSUn1} A.\ Canonaco, P.\ Stellari, \emph{Uniqueness of dg enhancements for the derived category of a Grothendieck category}, J.\ Eur.\ Math.\ Soc.\ {\bf 20} (2018), 2607--2641.

\bibitem{CSUn2} A.\ Canonaco, P.\ Stellari, \emph{Uniqueness of dg enhancements for the derived category of a Grothendieck category}, arXiv:1507.05509v5.

\bibitem{Coh} L.\ Cohn, \emph{Differential graded categories are $k$-linear stable infinity categories}, arXiv:1308.2587.

\bibitem{Dr} V.\ Drinfeld, \emph{DG quotients of DG categories}, J.\ Algebra {\bf 272} (2004), 643--691.

\bibitem{DS} D.\ Dugger, B.\ Shipley, \emph{A curious example of triangulated-equivalent model categories which are not Quillen equivalent},  Algebr.\ Geom.\ Topol.\ {\bf 9} (2009), 135--166. 

\bibitem{Ge} F.\ Genovese, \emph{Quasi-functors as lifts of Fourier-Mukai functors: the uniqueness problem}, PhD thesis, available at: {\tt \url{https://fgenovese1987.github.io/documents/thesis\_phd.pdf}}.

\bibitem{GrothSGA4} A.\ Grothendieck, J.-L.\ Verdier, \emph{Pr\'efaisceaux}, in SGA 4, Th\'eorie des Topos et Cohomologie Etale des Sch\'emas, Tome 1. Th\'eorie des Topos, Lect.\ Notes in Math.\ {\bf 269}, Springer, Heidelberg, 1972--1973, pp. 1-184.

\bibitem{Hir} P.\ Hirschhorn, \emph{Model Categories and their localizations}, Math.\ Surveys and
Monographs Series {\bf 99}, Am.\ Math.\ Soc.\ (2003), xvi+457.

\bibitem{Hov} M.\ Hovey, \emph{Model categories}, Mathematical Surveys and Monographs {\bf 63}, Am.\ Math.\ Soc.\ (1998), xii+209.

\bibitem{Kelly65} G.\ M.\ Kelly, \emph{Chain maps inducing zero homology maps}, Proc.\ Cambridge Philos.\ Soc.\ {\bf 61} (1965), 847--854.

\bibitem{K0} H.\ Krause, \emph{Deriving Auslander's formula}, Documenta Math.\ {\bf 20} (2015), 669--688.

\bibitem{K1} H.\ Krause, \emph{On Neeman's well generated triangulated categories}, Documenta Math.\  {\bf 6} (2001), 119--125.

\bibitem{K2} H.\ Krause, \emph{Localization theory for triangulated categories}, in: Triangulated categories, London Math.\ Soc.\ Lecture Note Ser.\ {\bf 375}, Cambridge Univ.\ Press (2010), 161--235.

\bibitem{LMB} G.\ Laumon, L.\ Moret-Bailly, \emph{Champs alg\'{e}briques}, Ergebnisse der Mathematik und ihrer Grenzgebiete.\ 3.\ Folge.\ A Series of Modern Surveys in Mathematics {\bf 39}, Springer-Verlag (2000), xii+208.

\bibitem{LO} V.\ Lunts, D.\ Orlov, \emph{Uniqueness of enhancements for triangulated categories}, J.\ Amer.\ Math.\ Soc.\ {\bf 23} (2010), 853--908.

\bibitem{LurHA} J.\ Lurie, \emph{Higher algebra}, {\tt \url{https://www.math.ias.edu/~lurie/papers/HA.pdf}}.

\bibitem{LurSpectr} J.\ Lurie, \emph{Spectral algebraic geometry}, {\tt \url{https://www.math.ias.edu/~lurie/papers/SAG-rootfile.pdf}}.

\bibitem{LS} V.\ Lunts, O.M.\ Schn\"urer, \emph{New enhancements of derived categories of coherent sheaves and applications}, J.\ Algebra {\bf 446} (2016), 203--274.

\bibitem{NNonLeftCompl} A.\ Neeman, \emph{Non-left-complete derived categories}, Math.\ Res.\ Lett.\ {\bf 18} (2011), 827--832.

\bibitem{N3} A.\ Neeman, \emph{On the derived category of sheaves on a manifold}, Documenta Math.\ {\bf 6} (2001), 483--488.

\bibitem{N1} A.\ Neeman, \emph{The connection between the K-theory localization theorem of Thomason, Trobaugh and Yao and the smashing subcategories of Bousfield and Ravenel}, Ann.\ Sci.\ \'{E}cole Norm.\ Sup.\ {\bf 25} (1992), 547--566.

\bibitem{NInj} A.\ Neeman, \emph{The homotopy category of injectives}, Algebra Number Theory {\bf 8} (2014), 429--456.

\bibitem{N2} A.\ Neeman, \emph{Triangulated categories}, Annals of Mathematics Studies {\bf 148}, Princeton University Press (2001), viii+449.

\bibitem{Neeman17B} A.\ Neeman, \emph{Approximable triangulated categories}, to appear in Contemporary Mathematics, https://arxiv.org/abs/1806.06995.

\bibitem{OT} Y.-G.\ Oh, H.\ Lee Tanaka, \emph{$A_\infty$-categories, their $\infty$-category, and their localizations}, arXiv:2003.05806.

\bibitem{RvdBN} A.\ Rizzardo, M.\ Van den Bergh, A.\ Neeman, \emph{An example of a non-Fourier--Mukai functor between derived categories of coherent sheaves}, Invent.\ Math.\ {\bf 216} (2019), 927--1004.

\bibitem{RvdB2} A.\ Rizzardo, M.\ Van den Bergh, \emph{A note on non-unique enhancements}, Proc.\ Amer.\ Math.\ Soc.\ {\bf 147} (2019), 451--453.

\bibitem{RvdB} A.\ Rizzardo, M.\ Van den Bergh, \emph{A $k$-linear triangulated category without a model}, Ann.\ Math.\ {\bf 191} (2020), 393--437.

\bibitem{R} R.\ Rouquier, \emph{Dimensions of triangulated categories}, J.\ K-theory {\bf 1} (2008), 193--258.

\bibitem{Sc} M.\ Schlichting, \emph{A note on K-theory and triangulated categories}, Invent.\ Math.\ {\bf 150} (2002), 111--116.

\bibitem{Sc2} M.\ Schlichting, \emph{Negative K-theory of derived categories}, Math.\ Z.\ {\bf 253} (2006), 97--134.

\bibitem {Ta} G.\ Tabuada, \emph{Une structure de cat\'{e}gorie de mod\`{e}les de Quillen sur la cat\'{e}gorie des dg-cat\'{e}gories},  C.\ R.\ Math.\ Acad.\ Sci.\ Paris {\bf 340} (2005), 15--19.

\bibitem{TT} R.\ W.\ Thomason, T.\ Trobaugh, \emph{Higher Algebraic K-Theory of Schemes and of Derived Categories}, The Grothendieck Festschrift III, Birkh\"auser, Boston, Basel, Berlin, (1990), 247--436.

\bibitem{To} B.\ To\"en, \emph{The homotopy theory of dg-categories and derived Morita theory}, Invent.\ Math.\ {\bf 167} (2007), 615--667.

\bibitem{Verd} J.-L.\ Verdier, \emph{Des cat\'egories d\'eriv\'ees des cat\'egories ab\'eliennes}, Ast\'erisque {\bf 239} (1996), xii+253 pp.

\bibitem{Virili18} S.\ Virili, \emph{Morita theory for stable derivators}, arXiv:1807.01505.

\end{thebibliography}
\end{document}